\documentclass[11pt]{amsart}
\usepackage[top=1.5in, bottom=1.5in, left=1.4in, right=1.4in]{geometry}
\usepackage{color}
\usepackage{graphicx}
\usepackage{epstopdf}
\usepackage[table,dvipsnames]{xcolor}
\usepackage{hyperref}

\usepackage{amssymb}
\usepackage{amsthm}
\usepackage{array}
\usepackage{mathtools}
\usepackage{enumerate}

\usepackage{ytableau}

\allowdisplaybreaks

\usepackage[all]{xy}
\usepackage{graphicx} 
\usepackage{tikz}
\usetikzlibrary{patterns}
\usetikzlibrary{decorations.pathreplacing}
\usepackage{tkz-graph}
\usetikzlibrary{arrows}
\usetikzlibrary{decorations.markings}
\usepackage{cellspace}
\newtheorem{thm}{Theorem}[section]
\newtheorem{lemma}[thm]{Lemma}
\newtheorem{cor}[thm]{Corollary}
\newtheorem{prop}[thm]{Proposition}
\newtheorem{conj}[thm]{Conjecture}

\newtheorem{Definition}[thm]{Definition}
\newenvironment{definition}
  {\begin{Definition}\rm}{\end{Definition}}

\newtheorem{Example}[thm]{Example}
\newenvironment{example}
  {\begin{Example}\rm}{\hfill$\blacksquare$\end{Example}}

\newtheorem{Remark}[thm]{Remark}
\newenvironment{remark}
  {\begin{Remark}\rm}{\hfill$\blacksquare$\end{Remark}}

\numberwithin{equation}{section}

\usepackage{etoolbox}
\apptocmd{\sloppy}{\hbadness 10000\relax}{}{}

\newcommand{\emailhref}[1]{\email{\href{#1}{#1}}}

\def \join {\vee}
\def \meet {\wedge}

\title{The CDE property for skew vexillary permutations}
\author[S. Hopkins]{Sam Hopkins}\emailhref{shopkins@umn.edu}
\address{School of Mathematics, University of Minnesota, Minneapolis, MN 55455}

\date{\today}

\begin{document}

\begin{abstract}
We prove a conjecture of Reiner, Tenner, and Yong which says that the initial weak order intervals corresponding to certain vexillary permutations have the coincidental down-degree expectations (CDE) property. Actually our theorem applies more generally to certain ``skew vexillary'' permutations (a notion we introduce), and shows that these posets are in fact ``toggle CDE.'' As a corollary we obtain a homomesy result for rowmotion acting on semidistributive lattices in the sense of Barnard and of Thomas and Williams.
\end{abstract}

\maketitle

\section{Introduction} \label{sec:intro}

Let $w \in \mathfrak{S}_n$ be a permutation. Consider the following two probability distributions on the set of permutations $u \in \mathfrak{S}_n$ which are less than or equal to $w$ in weak order~$(\mathfrak{S}_n,\leq)$. For the first distribution: select $u$ uniformly at random among all permutations $u \leq w$. For the second distribution: choose a reduced word $w=s_{i_1}s_{i_2}\cdots s_{i_{\ell(w)}}$ of~$w$ uniformly at random; then choose $k \in \{0,1,\ldots,\ell(w)\}$ uniformly at random; and finally define $u \coloneqq  s_{i_1}s_{i_2}\cdots s_{i_k}$. In general these two distributions will be quite different. Our main result is that for a large family of $w$ (``skew vexillary permutations of balanced shape''), although these two distributions are indeed different, the expected number of descents of the random permutation~$u$ is nevertheless the same for both.

This result is an instance of the ``coincidental down-degree expectations'' phenomenon introduced by Reiner, Tenner, and Yong~\cite{reiner2018poset}.

\begin{definition}[{See~\cite[Definition 2.1]{reiner2018poset}}] \label{def:cde}
Let $P$ be a finite poset. Let $\mathrm{uni}_P$ denote the uniform probability distribution on $P$. Let $\mathrm{maxchain}_P$ denote the probability distribution where each $p \in P$ occurs with probability proportional to the number of maximal chains containing $p$. Let~$\mathrm{ddeg}\colon P \to \mathbb{N}$ denote the \emph{down-degree} statistic: $\mathrm{ddeg}(p)$ is the number of elements of $P$ which $p\in P$ covers. If $\mu$ is a discrete probability distribution on a finite set~$X$ and $f\colon X \to \mathbb{R}$ is some statistic on $X$, we use the notation~$\mathbb{E}(\mu; f)$ to denote the expectation of $f$ with respect to~$\mu$. Finally, we say that $P$ has the \emph{coincidental down-degree expectations (CDE) property} if
\[\mathbb{E}(\mathrm{uni}_P; \mathrm{ddeg}) = \mathbb{E}(\mathrm{maxchain}_P; \mathrm{ddeg}).\]
In this case we also say that $P$ \emph{is CDE}.
\end{definition}

The coincidence of the expected number of descents for the two distributions on permutations described in the first paragraph of this introduction can be recast, in the language of Definition~\ref{def:cde}, as saying that the weak order interval $[e,w]$ between the identity permutation $e$ and our chosen permutation~$w$ is CDE. This is because the maximal chains in this weak order interval naturally correspond to the reduced words of $w$, and similarly the down-degree of a permutation in weak order is its number of descents. 

Note that $\mathbb{E}(\mathrm{uni}_P; \mathrm{ddeg})$ is the \emph{edge density} of $P$, i.e., the number of edges of the Hasse diagram of $P$ divided by the number of elements of $P$. As part of our main result, we will not only establish that $\mathbb{E}(\mathrm{uni}_{[e,w]}; \mathrm{ddeg}) = \mathbb{E}(\mathrm{maxchain}_{[e,w]}; \mathrm{ddeg})$ for the aforementioned family of permutations~$w$, but we will also give a simple formula for the edge density of these posets~$[e,w]$. There is no a priori reason to expect a simple formula for the edge density of a poset, so our result says that these posets~$[e,w]$ have a very special combinatorial structure.

Let us now briefly review the history of the study of CDE posets and explain how our result fits into this history.\footnote{The papers~\cite{chan2018genera} and~\cite{chan2017expected} do not use the term ``CDE'' because they appeared before the paper of Reiner-Tenner-Yong~\cite{reiner2018poset} which introduced this term. But their results amount to showing that various posets are CDE, so that is how we will describe these results in this introduction.} 

The first instance of a poset being shown to be CDE occurred in the context of the algebraic geometry of curves. Chan, L\'{o}pez Mart\'{i}n, Pflueger, and Teixidor i Bigas~\cite{chan2018genera} showed that the interval $[\varnothing,b^a]$ in Young's lattice of partitions between the empty shape~$\varnothing$ and the $a\times b$ rectangle $b^a \coloneqq  (\overbrace{b,b,\cdots,b}^{a})$ is CDE with edge density $ab/(a+b)$. This was the key combinatorial result these authors needed to reprove a product formula for the genus of Brill-Noether loci of dimension one.

Subsequently, Chan, Haddadan, Hopkins and Moci~\cite{chan2017expected} extended the combinatorial result of~\cite{chan2018genera} to apply to many more shapes beyond rectangles. They showed that if~$\sigma = \lambda/\nu$ is a ``balanced'' skew shape of height $a$ and width $b$, then the interval $[\nu,\lambda]$ in Young's lattice is also CDE with edge density $ab/(a+b)$. Here a skew shape $\sigma$ is \emph{balanced} if it is connected and all its outward corners occur exactly on the main antidiagonal of the smallest rectangle containing~$\sigma$. Rectangles $b^a$ are balanced, as are \emph{staircases} $\delta_d \coloneqq  (d-1,d-2,\ldots,1)$. Furthermore, if $\sigma$ is a balanced shape then the shape $\sigma\circ b^a$ obtained from $\sigma$ by replacing every box with an $a\times b$ rectangle is also balanced. So for instance the \emph{rectangular staircases} $\delta_d\circ b^a$ are also balanced shapes.

Actually, Chan-Haddadan-Hopkins-Moci proved something stronger: they proved that the interval $[\nu,\lambda]$ of Young's lattice corresponding to a balanced shape $\sigma = \lambda/\nu$ is \emph{toggle CDE (tCDE)}. 

To explain what tCDE means we need to discuss ``toggling'' in a distributive lattice. Any finite distributive lattice~$L$ (e.g., any interval of Young's lattice) can be written in an unique (up to isomorphism) way as~$L=\mathcal{J}(P)$, where~$\mathcal{J}(P)$ denotes the set of order ideals of a finite poset $P$ ordered by containment (this is the Fundamental Theorem for Finite Distributive Lattices; see, e.g.,~\cite[Theorem 3.4.1]{stanley2012ec1}). For an order ideal $I \in \mathcal{J}(P)$, we say that~$p \in P$ \emph{can be toggled into~$I$} if $p \notin I$ and~$I \cup \{p\}$ is an order ideal; similarly, we say~$p$ \emph{can be toggled out of~$I$} if $p \in I$ and~$I\setminus \{p\}$ is an order ideal. \emph{Toggling} $p$ in $I$ is the operation of adding $p$ to $I$ if it can be toggled into~$I$, removing $p$ from~$I$ if it can be toggled out of~$I$, and leaving $I$ unchanged otherwise. (This terminology derives from Striker and Williams~\cite{striker2012promotion}, who popularized the study of the ``toggle group'' of a poset, following Cameron and Fon-der-Flaass~\cite{cameron1995orbits}.) A probability distribution~$\mu$ on $\mathcal{J}(P)$ is called \emph{toggle-symmetric} if for each~$p\in P$ the probability that~$p$ can be toggled into $I$ is the same as the probability that~$p$ can be toggled out of $I$ when $I$ is distributed according to~$\mu$. We say that a distributive lattice $\mathcal{J}(P)$ is \emph{toggle CDE (tCDE)} if $\mathbb{E}(\mu;\mathrm{ddeg})$ is the same for every toggle-symmetric distribution $\mu$ on $\mathcal{J}(P)$. For a distributive lattice~$L$, both $\mathrm{uni}_L$ and $\mathrm{maxchain}_L$ are always toggle-symmetric distributions, so $L$ being tCDE implies that it is~CDE. 

The upgrade from CDE to tCDE in~\cite{chan2017expected} was not just a generalization for generalization's sake: the ``toggle perspective'' was crucial in establishing the result, and was successfully applied later in other contexts as well.

For instance, Hopkins~\cite{hopkins2017cde} adapted the techniques of~\cite{chan2017expected} to establish that various intervals of the shifted version of Young's lattice (corresponding to ``shifted balanced shapes'') are tCDE. In doing so he answered in the affirmative a conjecture of Reiner-Tenner-Yong, and also completed the case-by-case proof that the distributive lattice $\mathcal{J}(P)$ corresponding to a minuscule poset~$P$ is tCDE. Shortly thereafter, Rush~\cite{rush2016orbits} proved in a uniform way that $\mathcal{J}(P)$ is tCDE when $P$ is a minuscule poset.

All of the examples of CDE posets discussed above are distributive lattices. There are some families of posets (such as chains, or self-dual posets of constant Hasse diagram degree) which are CDE for straightforward reasons (see~\cite[Corollary 2.19, Proposition 2.20]{reiner2018poset}). It is also known that the Cartesian product of graded CDE posets is CDE~\cite[Proposition~2.13]{reiner2018poset}. Beyond these simple examples and the distributive lattices discussed above, the only other major family of posets known to be CDE was found by Reiner-Tenner-Yong: this family consists of initial weak order intervals~$[e,w]$ for certain dominant permutations~$w \in \mathfrak{S}_n$. Note importantly that these intervals~$[e,w]$ are {\bf not} distributive lattices (indeed, the only intervals of weak order which are distributive lattices are isomorphic to intervals of Young's lattice).

Let us describe more precisely these CDE weak order intervals. Recall that an \emph{inversion} of $w \in \mathfrak{S}_n$ is a pair $(i,j) \in \mathbb{Z}^2$ with $1\leq i < j\leq n$ for which~$w(i)>w(j)$. Also recall that the \emph{Rothe diagram} of $w \in \mathfrak{S}_n$ is the set of all pairs~$(i,w(j)) \in \mathbb{Z}^2$ for which~$(i,j)$ is an inversion of $w$. Let $\lambda$ be a straight shape. We say that~$w \in \mathfrak{S}_n$ is \emph{dominant of shape~$\lambda$} if its Rothe diagram is equal to (the Young diagram of) $\lambda$. 

Reiner-Tenner-Yong~\cite[Theorem 1.1]{reiner2018poset} proved that if~$\lambda = \delta_d \circ b^a$ is a rectangular staircase, and $w \in \mathfrak{S}_n$ is a dominant permutation of shape~$\lambda$, then~$[e,w]$ is CDE with edge density $(d-1)ab/(a+b)$. To do this they employed tableaux and the theory of Schur polynomials, Schubert polynomials, Grothendieck polynomials, et cetera (for a more precise account of what Reiner-Tenner-Yong did, see Remarks~\ref{rem:stable_grothendieck} and~\ref{rem:dominant_counting_edges}).

Reiner-Tenner-Yong also conjectured a significant generalization of their result. Let~$\lambda$ be a straight shape. We say that $w \in \mathfrak{S}_n$ is \emph{vexillary of shape~$\lambda$} if its Rothe diagram can be transformed to~$\lambda$ via some permutation of rows and columns. There is (essentially) a unique dominant permutation of a given shape~$\lambda$, but there are in general many vexillary permutations of shape~$\lambda$. Reiner-Tenner-Yong conjectured the following:

\begin{conj}[{\cite[Conjecture 1.2]{reiner2018poset}}] \label{conj:rty}
Let $\lambda = \delta_d \circ b^a$ be a rectangular staircase and $w \in \mathfrak{S}_n$ a vexillary permutation of shape $\sigma$. Then $[e,w]$ is CDE with edge density $(d-1)ab/(a+b)$.
\end{conj}

Our main result proves that Conjecture~\ref{conj:rty} is true. In fact, we show that there is nothing particularly special about rectangular staircases in this conjecture: the important thing is that the shape is balanced. Furthermore, our result will apply to skew shapes as well. Thus to state the result, we need to introduce the notion of ``skew vexillary'' permutations: we say that $w \in \mathfrak{S}_n$ is \emph{skew vexillary of shape~$\sigma=\lambda/\nu$} if its Rothe diagram can be transformed to~$\sigma$ via some permutation of rows and columns. (Note that this notion of skew vexillary is not standard; indeed, see Remark~\ref{rem:other_skew_vex} for other papers in which ``skew vexillary'' was used to mean something different.) 

Our main result is:

\begin{thm} \label{thm:intro_main}
Let $\sigma = \lambda / \nu$ be a balanced shape of height $a$ and width $b$ and $w \in \mathfrak{S}_n$ a skew vexillary permutation of shape $\sigma$. Then $[e,w]$ is CDE with edge density $ab/(a+b)$.
\end{thm}

We prove Theorem~\ref{thm:intro_main} by once again adopting the ``toggle perspective.'' Thus our approach is very different from that of Reiner-Tenner-Yong: we will not use tableaux or symmetric functions at all in this paper.

As mentioned, the weak order intervals $[e,w]$ we are studying are not distributive lattices, so notions related to toggling do not directly make sense when applied to these posets. However, intervals of weak order are \emph{semi}distributive lattices. And recently various authors have considered extending the idea of toggling from distributive to semidistributive lattices. Following the work of Reading~\cite{reading2015noncrossing} in the case of weak order on permutations, Barnard~\cite{barnard2016canonical} defined a canonical way to label the edges of the Hasse diagram (i.e., the cover relations) of any semidistributive lattice with join irreducible elements. As Thomas and Williams~\cite{thomas2017rowmotion} further emphasized, this edge labeling gives a natural way to extend the notion of toggling to a semidistributive lattice. This extended notion of toggling also allows us to generalize the notions of toggle-symmetric distributions, and tCDE posets, to the semidistributive setting. Note that we really are generalizing the distributive setting because when these generalized notions are applied to a distributive lattice we recapture the previous notions of toggle-symmetric, tCDE, et cetera.

Hence, what we will actually end up proving is that when $w \in \mathfrak{S}_n$ is a skew vexillary permutation of balanced shape the interval $[e,w]$ is tCDE, for this semidistributive generalization of tCDE.

It has been previously observed (see~\cite[\S3.2]{chan2017expected},~\cite[\S7]{hopkins2017cde},~\cite[\S1.3]{rush2016orbits}) that the study of tCDE distributive lattices has interesting applications to ``dynamical algebraic combinatorics''~\cite{roby2016dynamical, striker2017dynamical}. Namely, for any distributive lattice $\mathcal{J}(P)$ there is a certain invertible operator $\mathrm{row}\colon \mathcal{J}(P)\to \mathcal{J}(P)$ called \emph{rowmotion} which has been the focus of research of many authors~\cite{brouwer1974period, fonderflaass1993orbits, cameron1995orbits, panyushev2009orbits, armstrong2013uniform, striker2012promotion, rush2013orbits, rush2015orbits, grinberg2016birational1, grinberg2015birational2, striker2018rowmotion}. For an order ideal $I\in\mathcal{J}(P)$, $\mathrm{row}(I)$ is the order ideal generated by the minimal elements of $P$ not in $I$. As first pointed out by Striker~\cite[Lemma~6.2]{striker2015toggle}, it is easy to see that, for any fixed rowmotion orbit~$\mathcal{O}$, the distribution $\mu$ on~$\mathcal{J}(P)$ which is uniform on~$\mathcal{O}$ and zero outside~$\mathcal{O}$ is toggle-symmetric. Hence, if~$\mathcal{J}(P)$ is tCDE with edge density~$c$ then $\mathbb{E}(\mu; \mathrm{ddeg}) = c$ for such a distribution~$\mu$. In other words, the average down-degree is equal to the same constant~$c$ along all of the orbits of rowmotion. This is a \emph{homomesy} result in the sense of Propp and Roby~\cite{propp2015homomesy, einstein2013combinatorial}: a statistic $f\colon X \to \mathbb{R}$ is said to be $c$-\emph{mesic} with respect to the action of an invertible operator $\varphi\colon X \to X$ on a combinatorial set~$X$ if the average of $f$ is equal to the same constant $c \in \mathbb{R}$ along every orbit of~$\varphi$. Note that in the context of rowmotion, the statistic $\mathrm{ddeg}$ is called the ``antichain cardinality,'' and hence this result is called the antichain cardinality homomesy for rowmotion.

For any semidistributive lattice~$L$, Barnard's edge labeling~\cite{barnard2016canonical} yields a generalization of romwotion, also denoted~$\mathrm{row}\colon L \to L$. (This generalized rowmotion was further studied by Thomas-Williams~\cite{thomas2017rowmotion}.) It is again easy to see that for a semidistributive lattice~$L$ the distribution $\mu$ which is uniform on a rowmotion orbit is toggle-symmetric in our generalized sense of toggle-symmetry. So we obtain the following corollary, generalizing the antichain cardinality homomesy for rowmotion beyond the distributive setting:

\begin{cor}
Let $\sigma = \lambda / \nu$ be a balanced shape of height $a$ and width $b$ and $w \in \mathfrak{S}_n$ a skew vexillary permutation of shape $\lambda$. Then $\mathrm{ddeg}$ is $ab/(a+b)$-mesic with respect to the action of~$\mathrm{row}$ on $[e,w]$.
\end{cor}

We now briefly outline the rest of the paper. In Section~\ref{sec:background} we review some basic background on lattices, Young's lattice, and weak order. In Section~\ref{sec:skew_vex} we introduce the family of skew vexillary permutations of a given shape, compare this family to other commonly studied families of permutations, and discuss some basic properties of this family. In Section~\ref{sec:tcde} we review Barnard's labeling of the edges of a semidistributive lattice and use it to define toggling, toggle-symmetric distributions, and the tCDE property in the semidistributive setting. We show that for intervals of weak order, the tCDE property implies the CDE property (although this is actually not true for general semidistributive lattices). In Section~\ref{sec:main} we recall the results and techniques of Chan-Haddadan-Hopkins-Moci, and then we prove our main result which says that the initial weak order intervals corresponding skew vexillary permutations of balanced shape are tCDE. The key technical tool here is the use of ``rook'' random variables (extending an idea from~\cite{chan2017expected}). In Section~\ref{sec:rowmotion} we apply our result to deduce the aforementioned rowmotion homomesy corollary. In Section~\ref{sec:future} we briefly discuss possible future directions.

\medskip

\noindent {\bf Acknowledgements}: I thank Vic Reiner for useful discussions and encouragement during this project. I also thank the anonymous referees for their many helpful comments. This research would not have been possible without the use of Sage mathematics software~\cite{Sage-Combinat, sagemath}. I was supported by NSF grant~\#1802920.

\section{Background on lattices, Young's lattice, and weak order} \label{sec:background}

In this section we review the basic background on lattices, Young's lattice, and weak order that we will need in the rest of the paper. The reader who feels comfortable with these objects may safely skip over this section.

\subsection{Lattices} 
We generally use standard notation and terminology for poset-theoretic concepts (as in, e.g.,~\cite[Chapter 3]{stanley2012ec1}), but let us briefly review these here.

All posets in this paper will be finite unless explicitly stated otherwise.

Let $P$ be a poset. An \emph{antichain} of $P$ is a subset of elements which are pairwise incomparable. A \emph{chain} of~$P$ is a subset of elements which are pairwise comparable. The \emph{length} of a chain is its number of elements minus one. A chain is \emph{maximal} if it is maximal by containment. We say $P$ is \emph{graded of rank $r$} if all its maximal chains have the same length~$r$.

We use $P^*$ to denote the \emph{dual poset} to a poset $P$, which is the poset with the same elements as $P$ but with $x \leq_{P^*} y$ if and only if $y \leq_P x$.

We use $P_1 \times P_2$ to denote the \emph{Cartesian product} of two posets $P_1$, $P_2$; this is the poset with elements $(x_1,x_2)\in P_1\times P_2$ and $(x_1,x_2) \leq (y_1,y_2)$ if and only if $x_1 \leq y_1$ and $x_2\leq y_2$.

For $x,y \in P$ we say that $y$ \emph{covers} $x$ if $x \leq y$ and for any $z\neq x \in P$ with~$x \leq z$ we have~$y \leq z$; we denote this cover relation by $x \lessdot y$. The \emph{Hasse diagram} of $P$ is the graph on the elements of $P$ with edges $\{x,y\}$ whenever $x \lessdot y$; we draw this graph in the plane with $x$ below $y$ if $x \leq y$. We often depict a poset via its Hasse diagram. We say $P$ is \emph{connected} if its Hasse diagram is connected.

If $x,y\in P$ are two elements of $P$ then the \emph{join} of $x$ and $y$, denoted $x \join y$, is, if it exists, the minimal element in $P$ greater than or equal to both $x$ and $y$; dually, the \emph{meet} of $x$ and $y$, denoted $x \meet y$, is, if it exists, the maximal element in $P$ less or equal to both $x$ and $y$.

A \emph{lattice} $L$ is a poset such that both $x \join y$ and $x \meet y$ exist for every $x,y \in L$. An element of $L$ is \emph{join-irreducible} if it cannot be written as a join of other elements; equivalently, $x \in L$ is join-irreducible if $x$ covers a unique element. We use~$\mathrm{Irr}(L)$ to denote the join-irreducible elements of $L$. We consider $\mathrm{Irr}(L)$ as a poset with its partial order induced from $L$. There is obviously a dual notion of \emph{meet-irreducible} element.

The lattice $L$ is \emph{distributive} if the meet operation distributes over the join operation in the sense that $x\meet (y\join z) = (x\meet y)\join (x\meet z)$ for every $x,y,z \in L$. 

There is another equivalent way to view finite distributive lattices. An \emph{order ideal} of a poset $(P,\leq)$ is a subset $I\subseteq P$ for which $x,y\in P$ with $x \leq y$ and $y \in I$ implies~$x \in I$. The set of order ideals of $P$ ordered by containment is denoted~$\mathcal{J}(P)$ and is always a distributive lattice. For a finite distributive lattice $L$ we have~$L \simeq \mathcal{J}(\mathrm{Irr}(L))$ via the isomorphism $x \mapsto \{p\colon p\in \mathrm{Irr}(L), p\leq x\}$. And similarly for a finite poset $P$ we have that~$P\simeq \mathrm{Irr}(\mathcal{J}(P))$ via the isomorphism $q \mapsto \{p\colon p\in P,p\leq q\}$. This establishes a bijection between finite distributive lattices and finite posets (see, e.g.,~\cite[Theorem~3.4.1]{stanley2012ec1}).

Every finite distributive lattice $L$ is graded of length $\#\mathrm{Irr}(L)$.

A lattice $L$ is \emph{semidistributive} if it satisfies two specific weakenings of the distributive law (see~\cite{freese1995free}). But, like the description of (finite) distributive lattices in terms of order ideals, there is another more combinatorial way to define semidistributivity.  For a subset $S=\{s_1,\ldots,s_k\}\subseteq L$ we write $\join S \coloneqq  s_1\join s_2 \join \cdots \join s_k$. A \emph{join representation} of $x \in L$ is way of writing $x = \join S$ where $S\subseteq L$ is an antichain. We can order the antichains of $L$ by declaring $S \leq T$ if for every $x \in S$ there exists a~$y \in T$ with~$x \leq y$. A \emph{canonical join representation} of $x \in L$ is, if it exists, a join representation $x = \join S$ such that $S \leq T$ for every other join representation $x = \join T$. There is obviously a dual notion of \emph{canonical meet representation}. A lattice $L$ is \emph{semidistributive} if a canonical join representation and a canonical meet representation exist for each $x \in L$ (see~\cite[Theorem 2.24]{freese1995free}).

All distributive lattices are semidistributive: the canonical join representation of and order ideal~$I\in \mathcal{J}(\mathrm{Irr}(L))$ is $I = \join \mathrm{max}(I)$, where $\mathrm{max}(I)$ is the set of maximal elements of $I$; and similarly for the canonical meet representation.

Unlike distributive lattices, semidistributive lattices need not be graded.

For a poset $P$ and $x,y\in P$ with $x \leq y$, the \emph{(closed) interval from $x$ to $y$} is defined to be $[x,y] \coloneqq  \{z\in P\colon x \leq z\leq y\}$. Every interval of a lattice is again a lattice; similarly, every interval of a distributive lattice is a distributive lattice and every interval of a semidistributive lattice is a semidistributive lattice.

\subsection{Young's lattice} \label{sec:young_lat_defs}

We also follow standard notation for partitions and Young's lattice (as in, e.g.,~\cite[\S7.2]{stanley1999ec2}) but we will review these briefly here.

A \emph{partition} $\lambda = (\lambda_1,\lambda_2,\ldots)$ is a sequence $\lambda_1 \geq \lambda_2 \geq \cdots \in \mathbb{N}$  of weakly decreasing nonnegative integers that is eventually zero. The \emph{size} of $\lambda$ is $|\lambda| \coloneqq  \sum_{i=1}^{\infty} \lambda_i$ and the \emph{length} of $\lambda$ is $\ell(\lambda) \coloneqq  \mathrm{min}\{i\in \mathbb{N}\colon \lambda_{i+1}=0\}$. We also write $\lambda = (\lambda_1,\ldots,\lambda_{\ell(\lambda)})$ as a shorthand. There is a unique partition of size $0$ called the \emph{empty shape} and denoted~$\varnothing$. \emph{Young's lattice} $\mathbb{Y}$ is the infinite poset of all partitions with $\nu \leq \lambda$ if $\nu_i \leq \lambda_i$ for all~$i=1,2,\ldots$. Young's lattice is a distributive lattice: in fact we have $\mathbb{Y}=\mathcal{J}(\mathbb{N}\times\mathbb{N})$. 

We will always consider partitions ordered according to Young's lattice, and if we write $[\nu,\lambda]$ for partitions $\nu,\lambda$, that always means an interval of Young's lattice.

It is very helpful to visualize partitions via their Young diagrams. So let us briefly discuss diagrams in general. By a \emph{diagram} $D$ we mean a finite subset~$D\subseteq \mathbb{Z}^2$. We think of the elements $(i,j)\in D$ of a diagram as being ``boxes'' placed at those coordinates. We use ``matrix coordinates'' where the box at $(1,1)$ is northwest of the one at~$(2,2)$. A \emph{row} of a diagram $D$ is the set of boxes $(i,j) \in D$ for some fixed~$i\in\mathbb{Z}$; similarly, a \emph{column} of $D$ is the set of boxes~$(i,j)\in D$ for some fixed~$j\in\mathbb{Z}$. The \emph{transpose} of~$D$, denoted $D^t$, is the diagram with boxes $(j,i)$ for~$(i,j)\in D$. The \emph{rotation} of $D$, denoted~$D^{\mathrm{rot}}$, is the diagram obtained from $D$ by rotating it  $180^\circ$ about some fixed point. An important diagram is the rectangle $[a]\times[b]$ for $a,b\in\mathbb{N}$, where we use the standard notation $[a]\coloneqq \{1,2,\ldots,a\}$.

The \emph{Young diagram} of a partition $\lambda$ is $\{(i,j)\colon 1 \leq i \leq \ell(\lambda), 1 \leq j \leq \lambda_i\}$. In other words, we have $\lambda_i$ boxes left-justified in the $i$th row of the Young diagram. For example, the Young diagram of $\lambda=(4,3,3,3)$ is
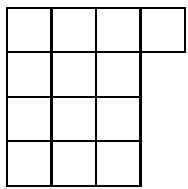
\[\ydiagram{4,3,3,3} \]
Note that the initial interval $[\varnothing,\lambda]$ of Young's lattice is equal to $\mathcal{J}(P_{\lambda})$ where $P_{\lambda}$ is the poset on the boxes of the Young diagram of $\lambda$ with the partial order induced from $\mathbb{Z}^2$. We often implicitly identify a partition with its Young diagram. Observe that $\nu \leq \lambda$ if and only if the Young diagram of~$\nu$ is contained in the Young diagram of~$\lambda$.

For two partitions $\nu \leq \lambda$, the \emph{skew shape} $\sigma = \lambda/\nu$ is the set-theoretic difference of the Young diagrams of $\lambda$ and $\nu$. For example, the skew shape $\sigma=(4,3,3,3)/(2,1,1)$ is
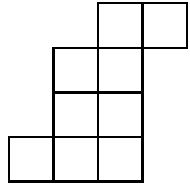
\[ \ydiagram{2+2,1+2,1+2,3} \]
The interval $[\nu,\lambda]$  of Young's lattice is equal to $\mathcal{J}(P_{\sigma})$ where $P_{\sigma}$ is the poset on the boxes of $\sigma=\lambda/\nu$ with the partial order induced from $\mathbb{Z}^2$. We also often refer to a skew shape as just a \emph{shape}, and sometimes we refer to an ordinary partition as a \emph{straight shape} to distinguish it from a skew shape.

Some straight shapes of particular significance for us, also discussed in Section~\ref{sec:intro}, are the $a\times b$ \emph{rectangles} $b^a \coloneqq  (\overbrace{b,b,\cdots,b}^{a})$ and the \emph{staircases} $\delta_d \coloneqq  (d-1,d-2,\ldots,1)$. Observe that the Young diagram of $b^a$ is $[a]\times[b]$. For any shape~$\sigma$, we use $\sigma\circ b^a$ to denote the shape obtained from $\sigma$ by replacing each box with an $a\times b$ rectangle. In this way we get the \emph{rectangular staircases} $\delta_d \circ b^a$. For instance, $\delta_3\circ 2^1$ is
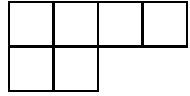
\[\ydiagram{4,2} \]

We say that the shape $\sigma$ is \emph{connected} if the poset $P_{\sigma}$ on its boxes is connected. If~$\sigma$ is disconnected we write $\sigma = \sigma_1 \sqcup \sigma_2$ to denote that $\sigma$ is the union of the shapes $\sigma_1$ and $\sigma_2$, where no box of $\sigma_1$ is in the same row or column as any box of $\sigma_2$. If $\sigma$ is connected, then we say it has \emph{height} $a$ and \emph{width} $b$ if the rectangle $b^a$ is the smallest rectangle which (up to translation) $\sigma$ is contained within.

Now we review the nonstandard notion of ``balanced'' shapes from Chan-Haddadan-Hopkins-Moci~\cite{chan2017expected}, which is crucial for the statement of our main result.

\begin{definition}
Let $\sigma = \lambda/\nu$ be a connected skew shape of height $a$ and width~$b$. We assume that $\sigma$ has been translated so that it lies in the rectangle $a\times b$. The \emph{main antidiagonal} of $\sigma$ is the line connecting the northeast and southwest corners of the boundary of the rectangle $a\times b$. A \emph{corner} of $\sigma$ is a point where two line segments which are part of the boundary of $\sigma$ meet. We say the corner is \emph{outward} if no box of~$\sigma$ intersects both of the line segments (except at the corner point). We say that $\sigma$ is \emph{balanced} if all its outward corners are exactly on its main antidiagonal.
\end{definition}

\begin{example}
Let $\sigma=(8,8,8,2)/(4,4)$. Below we draw $\sigma$ with its main antidiagonal in red and its two outward corners marked with black circles:
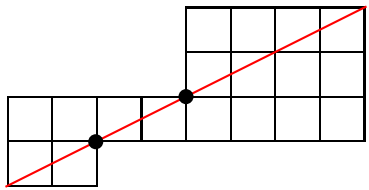
\begin{center}
\begin{tikzpicture}
\node at (0,0) {\ydiagram{4+4,4+4,8,2}};
\def\x{0.6}
\draw[red,thick] (-4*\x,-2*\x) -- (4*\x,2*\x);
\fill[black] (-2*\x,-1*\x) circle (0.1cm);
\fill[black] (0*\x,0*\x) circle (0.1cm);
\end{tikzpicture}
\end{center}
All the outward corners of $\sigma$ are on its main antidiagonal, so $\sigma$ is balanced.
\end{example}

\begin{example}
Let $\sigma=(4,4,3,1)/(2)$. Below we draw $\sigma$ with its main antidiagonal in red and its three outward corners marked with black circles:
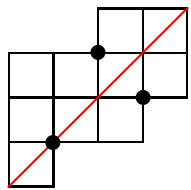
\begin{center}
\begin{tikzpicture}
\node at (0,0) {\ydiagram{2+2,4,3,1}};
\def\x{0.6}
\draw[red,thick] (-2*\x,-2*\x) -- (2*\x,2*\x);
\fill[black] (-1*\x,-1*\x) circle (0.1cm);
\fill[black] (0*\x,1*\x) circle (0.1cm);
\fill[black] (1*\x,0*\x) circle (0.1cm);
\end{tikzpicture}
\end{center}
Two of the outward corners of $\sigma$ are off its main antidiagonal, so $\sigma$ is not balanced.
\end{example}

It is easy to see that rectangles $b^a$ and staircases $\delta_d$ are balanced. Moreover, it is also easy to see that if $\sigma$ is balanced then $\sigma \circ b^a$ is balanced for any $a,b\in\mathbb{N}$. So  in particular the rectangular staircases $\delta_d \circ b^a$ are balanced.

We note that there are $3^{\mathrm{gcd}(a,b)-1}$ balanced shapes of height $a$ and width $b$, of which $2^{\mathrm{gcd}(a,b)-1}$ are straight shapes.

Young's lattice is intimately related to the combinatorics of tableaux, and hence symmetric function theory. For instance, maximal chains in $[\nu,\lambda]$ correspond exactly to the ``Standard Young Tableaux'' of shape $\sigma=\lambda/\nu$. But we will not review tableaux here because we will not need them.

\subsection{Weak order} \label{sec:weak_order_def}

We will now review weak order on the symmetric group. A good reference for weak order in the more general context of Coxeter groups is~\cite[\S3]{bjorner2005coxeter}.

The \emph{symmetric group} $\mathfrak{S}_n$ is the group of all permutations of~$\{1,\ldots,n\}$. We write a permutation $w \in \mathfrak{S}_n$ in \emph{one-line notation} as $w=w_1w_2\ldots w_n$ with $w_i \coloneqq  w(i)$. Sometimes we refer to these $w_i$ as the \emph{letters} of the permutation. We multiply permutations ``on the right'' so that $w = vu$ means that~$w(i) = v(u(i))$ for~$u,v,w\in \mathfrak{S}_n$. The \emph{inverse} of $w \in \mathfrak{S}_n$ is the permutation $w^{-1} \in \mathfrak{S}_n$ with $w^{-1}(i)=j$ if and only if $w(j) = i$. The \emph{reverse-complement} of $w\in \mathfrak{S}_n$ is the permutation $w^{\mathrm{rc}}$ defined by~$w^{\mathrm{rc}}\coloneqq  (n+1-w_n)(n+1-w_{n-1})\cdots (n+1-w_1) \in \mathfrak{S}_n$. Inversion and reverse-complementation commute. We use $e$ to denote the \emph{identity permutation} of $\mathfrak{S}_n$ which has $e(i) = i$ for all $1\leq i \leq n$. Note that by abuse of notation we use this same notation $e$ for the identity element of $\mathfrak{S}_n$ for all $n$.

For $1 \leq i < j \leq n$ we use $s_{(i,j)} \in \mathfrak{S}_n$ to denote the \emph{transposition} of $i$ and $j$: this is the permutation with $s_{(i,j)}(i) = j$, $s_{(i,j)}(j) = i$ and $s_{(i,j)}(k) = k$ for all $k\notin \{i,j\}$. For~$1 \leq i < n$ we use $s_{i}\coloneqq s_{(i,i+1)}$ to denote the $i$th \emph{simple transposition}. It is well-known that the simple transpositions $s_i$ for $1\leq i < n$ generate $\mathfrak{S}_n$ and satisfy the \emph{Coxeter relations}:
\begin{align}
s_i^2 &= 1 \qquad &\textrm{ for $1\leq i < n$}; \label{eqn:cox_inv} \\
s_is_{i+1}s_i &= s_{i+1}s_is_{i+1} \qquad &\textrm{ for $1\leq i < n-1$}; \label{eqn:cox_braid} \\
s_is_j &= s_js_i \qquad &\textrm{ for $1\leq i,j< n$ with $|i-j| \geq 2$}. \label{eqn:cox_comm}
\end{align}
We remark that reverse complementation corresponds to the automorphism of $\mathfrak{S}_n$ that swaps generators $s_i$ and $s_{n-i}$ for all $1 \leq i < n$.

The \emph{length} $\ell(w)$ of $w \in \mathfrak{S}_n$ is the minimum length $k$ of an expression $w=s_{i_1}s_{i_2}\cdots s_{i_k}$ as a product of simple transpositions. (We have $\ell(e)=0$ corresponding to the empty product.)

An \emph{inversion} of a permutation $w \in \mathfrak{S}_n$ is a pair $(i,j) \in\mathbb{Z}^2$ with $1 \leq i < j \leq n$ such that $w(i) > w(j)$. We use $\mathrm{Inv}(w)$ to denote the set of inversions of $w$. We write~$\mathrm{Inv}^{-1}(w) \coloneqq  \mathrm{Inv}(w^{-1})$. Note~$(i,j)\in\mathrm{Inv}(w)$ if and only if $(w(j),w(i))\in\mathrm{Inv}^{-1}(w)$. The length of a permutation is equal to its number of inversions: $\ell(w) = \#\mathrm{Inv}(w)$ for all $w \in \mathfrak{S}_n$. Hence $\ell(w)=\ell(w^{-1})$ for all~$w\in \mathfrak{S}_n$.

\emph{Weak order} is the poset $(\mathfrak{S}_n,\leq)$ with $u \leq w$ for $u,w\in \mathfrak{S}_n$ if and only if there is some sequence $s_{i_1},s_{i_2},\ldots,s_{i_k}$ of simple transpositions  such that $w = us_{i_1},s_{i_2},\ldots,s_{i_k}$ and~$\ell(us_{i_1}\cdots s_{i_j})= \ell(u)+j$ for all $1\leq j \leq k$. Observe that if $u \lessdot w$ in weak order then there is some simple transposition $s_i$ with $w = us_i$ and $\ell(w)=\ell(u)+1$. It is well-known that the weak order relation corresponds exactly to containment of inverse inversions: that is, we have $u \leq w$ if and only if $\mathrm{Inv}^{-1}(u) \subseteq \mathrm{Inv}^{-1}(w)$.

Clearly the identity $e \in \mathfrak{S}_n$ is the unique minimal element in weak order. There is also a unique maximal element in weak order denoted $w_0 \in \mathfrak{S}_n$: this is the permutation $w_0\coloneqq n (n-1) (n-2)\cdots 1$. Hence weak order is graded of rank $\ell(w_0)=\binom{n}{2}$. It is a theorem of Bj\"{o}rner that weak order is a lattice~\cite{bjorner1984orderings}. In fact weak order is known to be a semidistributive lattice~\cite{duquenne1994permutation}.

We will always consider permutations in~$\mathfrak{S}_n$ partially ordered according to weak order, and if we write $[u,w]$ for permutations $u,w\in \mathfrak{S}_n$ then we mean an interval of weak order. 

Here are a few basic facts about weak order intervals. First of all, we always have that~$[u,w]\simeq [e,u^{-1}w]$ via the map $v \mapsto u^{-1}v$. So, at least if we care only about the abstract poset structure of these intervals, we lose nothing by restricting our attention to \emph{initial} intervals $[e,w]$. Next, note that $[e,w]\simeq [e,w^{-1}]^*$ via the map $v \mapsto w^{-1}v$. Finally, note that~$[e,w]\simeq [e,w^{\mathrm{rc}}]$ via $v \mapsto v^{\mathrm{rc}}$. 

Weak order is very useful for understanding \emph{reduced words}. A \emph{reduced word} of $w\in \mathfrak{S}_n$ is a minimal length way of writing $w$ as a product of simple transpositions: in other words, it is a choice of sequence $s_{i_1},\ldots,s_{i_{\ell(w)}}$ such that $w=s_{i_1}\cdots s_{i_{\ell(w)}}$. Reduced words of $w$ correspond bijectively to maximal chains in $[e,w]$. It is a theorem of Matsumoto and Tits that all reduced words for $w$ are related by a sequence moves of the form~\eqref{eqn:cox_braid} and~\eqref{eqn:cox_comm}, i.e., we never need to use the relation~\eqref{eqn:cox_inv} to go between reduced words (see, e.g.,~\cite[Theorem 3.3.1]{bjorner2005coxeter}).

Reduced words are very important in Schubert calculus, but we will not explain the fundamental objects of Schubert calculus (like ``Schubert polynomials'') because we will not need these.

As mentioned, weak order is a lattice. Hence we might be interested in its join irreducible elements. To describe these we need to introduce descents of permutations. 

A \emph{descent} of $w\in \mathfrak{S}_n$ is a pair $(w_k,w_{k+1})$ for $1\leq k < n$ with $w_k > w_{k+1}$; we say this descent is \emph{at position $k$}. The descents of $w\in \mathfrak{S}_n$ exactly correspond to the elements which $w$ covers in weak order: for $1\leq k < n$ we have $ws_k \lessdot w$ if and only if~$w$ has a descent at position $k$. And if $(w_k,w_{k+1})$ is a descent of $w\in \mathfrak{S}_n$ at position~$k$, then we have $\mathrm{Inv}^{-1}(w)=\mathrm{Inv}^{-1}(u)\cup\{(w_{k+1},w_k)\}$ where $u\coloneqq ws_k$.

\begin{definition} \label{def:gr}
A permutation $w \in \mathfrak{S}_n$ is \emph{Grassmannian} if it has at most one descent.
\end{definition}

By definition we have that $\mathrm{Irr}(\mathfrak{S}_n)$ consists of the non-identity Grassmannian permutations.

Let us also explain another way in which the Grassmannian permutations are significant. For $1 \leq k < n$ we view $\mathfrak{S}_k \times \mathfrak{S}_{n-k}\subseteq \mathfrak{S}_n$ where the first factor acts only on $\{1,2,\ldots,k\}$ and the second factor acts only on $\{k+1,\ldots,n\}$. This is a so-called ``maximal parabolic'' subgroup of $\mathfrak{S}_n$, and the general theory of Coxeter groups says that the cosets of such a subgroup have distinguished minimal length representatives. The minimal length representatives of the left cosets $w(\mathfrak{S}_k \times \mathfrak{S}_{n-k})$ are exactly the Grassmannian permutations $w\in \mathfrak{S}_n$ having at most one descent at position $k$. 

As mentioned, weak order is a \emph{semidistributive} lattice, so we should also be interested in its canonical join representations. But in fact we will save discussion of canonical join representations in weak order for Section~\ref{sec:weak_order_edge_labels}.

\section{Skew vexillary permutations} \label{sec:skew_vex}

\subsection{Skew vexillary permutations and sub-families} \label{sec:skew_vex_def}

Our principal objects of interest in this paper are the initial weak order intervals $[e,w]$ where $w$ is a skew vexillary permutation. In this section we introduce this notion of ``skew vexillary.''

Recall that for a permutation $w\in \mathfrak{S}_n$, the \emph{Rothe diagram} of $w$ is the diagram which has boxes $(i,w(j))$ for all $(i,j) \in \mathrm{Inv}(w)$.

\begin{definition}
Let $\sigma = \lambda/\nu$ be a skew shape. We say that $w \in \mathfrak{S}_n$ is \emph{skew vexillary of shape~$\sigma$} if its Rothe diagram can be transformed to~$\sigma$ via some permutation of rows and columns.\footnote{Here and throughout ``permutation of rows and columns'' means that we permute the rows and columns separately; that is, we are not allowed to replace a row with a column or vice-versa.} We say that $w$ is \emph{skew vexillary} if it is skew vexillary of some shape.
\end{definition}

\begin{example}
Consider $w=31542$. To draw the Rothe diagram of $w\in \mathfrak{S}_n$, we can do the following: place an $X$ at every $(i,w(i)) \in [n]\times[n]$, and at each $X$ draw a ray emanating to the right and a ray emanating downwards; the boxes of $[n]\times [n]$ which have no $X$ markings and no rays passing through them are the boxes of the Rothe diagram of~$w$. So the Rothe diagram in our case looks like:
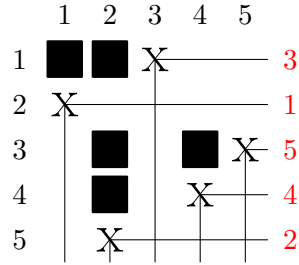
\begin{center}
\begin{tikzpicture}[scale=0.6]
\node at (6,5) {\color{red} $3$};
\node at (6,4) {\color{red} $1$};
\node at (6,3) {\color{red} $5$};
\node at (6,2) {\color{red} $4$};
\node at (6,1) {\color{red} $2$};
\node at (0,5) {$1$};
\node at (0,4) {$2$};
\node at (0,3) {$3$};
\node at (0,2) {$4$};
\node at (0,1) {$5$};
\node at (1,6) {$1$};
\node at (2,6) {$2$};
\node at (3,6) {$3$};
\node at (4,6) {$4$};
\node at (5,6) {$5$};
\node at (1,5) {\huge $\blacksquare$};
\node at (2,5) {\huge $\blacksquare$};
\node at (3,5) {\huge x};
\draw (3,5) -- (5.5,5);
\draw (3,5) -- (3,0.5);
\node at (1,4) {\huge x};
\draw (1,4) -- (5.5,4);
\draw (1,4) -- (1,0.5);
\node at (2,3) {\huge $\blacksquare$};
\node at (4,3) {\huge $\blacksquare$};
\node at (5,3) {\huge x};
\draw (5,3) -- (5.5,3);
\draw (5,3) -- (5,0.5);
\node at (2,2) {\huge $\blacksquare$};
\node at (4,2) {\huge x};
\draw (4,2) -- (5.5,2);
\draw (4,2) -- (4,0.5);
\node at (2,1) {\huge x};
\draw (2,1) -- (5.5,1);
\draw (2,1) -- (2,0.5);
\end{tikzpicture}
\end{center}
By applying the permutation $\pi_r=14325$ to the rows and $\pi_c=42135$ to the columns of this Rothe diagram, we transform it to the following diagram:
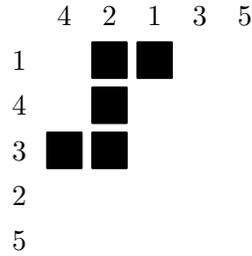
\begin{center}
\begin{tikzpicture}[scale=0.6]
\node at (0,5) {$1$};
\node at (0,4) {$4$};
\node at (0,3) {$3$};
\node at (0,2) {$2$};
\node at (0,1) {$5$};
\node at (1,6) {$4$};
\node at (2,6) {$2$};
\node at (3,6) {$1$};
\node at (4,6) {$3$};
\node at (5,6) {$5$};
\node at (2,5) {\huge $\blacksquare$};
\node at (3,5) {\huge $\blacksquare$};
\node at (2,4) {\huge $\blacksquare$};
\node at (1,3) {\huge $\blacksquare$};
\node at (2,3) {\huge $\blacksquare$};
\end{tikzpicture}
\end{center}
which is the shape $(3,2,2)/(1,1)$. Hence $w$ is skew vexillary of shape $(3,2,2)/(1,1)$.
\end{example}

\begin{example}
Consider $w=246153$.  Its Rothe diagram is:
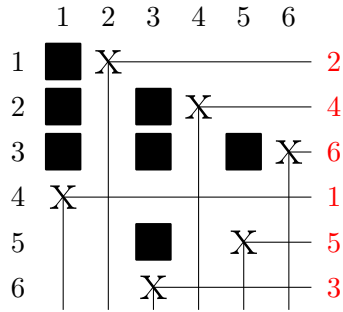
\begin{center}
\begin{tikzpicture}[scale=0.6]
\node at (7,6) {\color{red} $2$};
\node at (7,5) {\color{red} $4$};
\node at (7,4) {\color{red} $6$};
\node at (7,3) {\color{red} $1$};
\node at (7,2) {\color{red} $5$};
\node at (7,1) {\color{red} $3$};
\node at (0,6) {$1$};
\node at (0,5) {$2$};
\node at (0,4) {$3$};
\node at (0,3) {$4$};
\node at (0,2) {$5$};
\node at (0,1) {$6$};
\node at (1,7) {$1$};
\node at (2,7) {$2$};
\node at (3,7) {$3$};
\node at (4,7) {$4$};
\node at (5,7) {$5$};
\node at (6,7) {$6$};
\node at (1,6) {\huge $\blacksquare$};
\node at (2,6) {\huge x};
\draw (2,6) -- (6.5,6);
\draw (2,6) -- (2,0.5);
\node at (1,5) {\huge $\blacksquare$};
\node at (3,5) {\huge $\blacksquare$};
\node at (4,5) {\huge x};
\draw (4,5) -- (6.5,5);
\draw (4,5) -- (4,0.5);
\node at (1,4) {\huge $\blacksquare$};
\node at (3,4) {\huge $\blacksquare$};
\node at (5,4) {\huge $\blacksquare$};
\node at (6,4) {\huge x};
\draw (6,4) -- (6.5,4);
\draw (6,4) -- (6,0.5);
\node at (1,3) {\huge x};
\draw (1,3) -- (6.5,3);
\draw (1,3) -- (1,0.5);
\node at (3,2) {\huge $\blacksquare$};
\node at (5,2) {\huge x};
\draw (5,2) -- (6.5,2);
\draw (5,2) -- (5,0.5);
\node at (3,1) {\huge x};
\draw (3,1) -- (6.5,1);
\draw (3,1) -- (3,0.5);
\end{tikzpicture}
\end{center}
We claim that no permutation of rows and columns will transform this Rothe diagram to a skew shape. Hence $w$ is not skew vexillary.
\end{example}

We note that the following are obviously equivalent:
\begin{itemize}
\item $w$ is skew vexillary of shape $\sigma$;
\item $w$ is skew vexillary of shape $\sigma^{\mathrm{rot}}$;
\item $\mathrm{Inv}(w)$ can be transformed to $\sigma$ via some permutation of rows and columns;
\item $w^{-1}$ is skew vexillary of shape $\sigma^t$;
\item $(w^{\mathrm{rc}})^{-1} = (w^{-1})^{\mathrm{rc}}$ is skew vexillary of shape $\sigma$.
\end{itemize}

\begin{remark} \label{rem:other_skew_vex}
Our notion of skew vexillary is not standard. Indeed, the term ``skew vexillary'' has been used by other authors to mean other things on at least two occasions. In~\cite[Proposition 2.3]{billey1993schubert}, Billey, Jockusch, and Stanley call a permutation $w$ skew vexillary if its Schubert polynomial $\mathfrak{S}_w$ is equal to a flagged skew Schur function. This is morally similar to our definition (see Remark~\ref{rem:stable_grothendieck}), but not exactly the same. Billey-Jockusch-Stanley do at one point discuss permutations whose Rothe diagram is a skew shape up to permutation of rows and columns (see~\cite[Proposition 2.4]{billey1993schubert}), but they give no special name to these permutations. Meanwhile, in~\cite[\S5.2]{klein2014counting} Klein, Lewis, and Morales call a permutation $w\in \mathfrak{S}_n$ skew vexillary if the {\bf complement} in $[n]\times[n]$ of its Rothe diagram is a skew shape up to permutation of rows and columns. This is almost the ``opposite'' of our notion of skew vexillary! Klein-Lewis-Morales~\cite[Proposition 5.6]{klein2014counting} show that their notion of skew vexillary is characterized by the avoidance of nine patterns. (Recall that we say the permutation $w=w_1w_2\cdots w_n$ \emph{contains} the pattern $\pi=\pi_1\cdots\pi_k$ if there is some sequence $1\leq i_1 < i_2 < \cdots < i_k \leq n$ with $w_{i_1}\cdots w_{i_k}$ order-isomorphic to $\pi_1\cdots \pi_k$; and we say that $w$ \emph{avoids} $\pi$ otherwise.)  As we will discuss below (see Remark~\ref{rem:pattern_avoid}), we do not believe our notion of skew vexillary is characterized by any pattern avoidance condition.
\end{remark}

Now let us discuss some important sub-families of skew vexillary permutations.

\begin{definition}
Let $\lambda$ be a straight shape. We say that $w \in \mathfrak{S}_n$ is \emph{vexillary of shape~$\lambda$} if its Rothe diagram can be transformed to~$\lambda$ via some permutation of rows and columns. We say that $w$ is \emph{vexillary} if it is vexillary of some shape.
\end{definition}

This definition of vexillary permutations directly motivated our definition of skew vexillary permutations. Lascoux and Sch\"{u}tzenberger~\cite{lascoux1982polynomes} (see also Wachs~\cite{wachs1985flagged}) showed that a permutation is vexillary if and only if it is $2143$-avoiding. Another equivalent characterization of vexillary permutations: a permutation $w$ is vexillary of shape $\lambda$ if and only if its Stanley symmetric function~$F_w$ is equal to the Schur function $s_{\lambda}$~\cite[Theorem 4.1]{stanley1984reduced}. We will discuss this characterization more in a moment; see Remark~\ref{rem:stable_grothendieck}. Clearly vexillary permutations are by definition skew vexillary.

\begin{definition}
Let $\lambda$ be a straight shape. We say that $w \in \mathfrak{S}_n$ is \emph{dominant of shape~$\lambda$} if the Rothe diagram of $w$ is equal to $\lambda$. We say that $w$ is \emph{dominant} if it is dominant of some shape.
\end{definition}

It is well-known that $w$ is dominant if and only if it is 132-avoiding (see, e.g.,~\cite[Proposition 13.2]{postnikov2009chains}). There are other equivalent characterizations of dominant permutations: for instance, the Schubert polynomial $\mathfrak{S}_w$ is equal to a single monomial if and only if $w$ is dominant (see, e.g.,~\cite[Proposition 13.3]{postnikov2009chains}). Dominant permutations are often the ``simplest'' permutations in the context of combinatorial Schubert calculus. Dominant permutations are vexillary, and hence skew vexillary.

\begin{definition}
We say that $w\in \mathfrak{S}_n$ is \emph{fully commutative} if starting from any reduced word for $w$ we can obtain any other reduced word for $w$ by applying a sequence of commutation relations (i.e., we only need to use the relations~\eqref{eqn:cox_comm}, not~\eqref{eqn:cox_braid}).
\end{definition}

It is well-known that $w$ is fully commutative if and only if it is 321-avoiding~\cite[Theorem 2.1]{billey1993schubert}. Stembridge~\cite{stembridge1995fully} showed, in the more general context of Coxeter groups, that the initial interval $[e,w]$ of weak order is a distributive lattice if and only if $w$ is fully commutative. In fact, in the case of the symmetric group, we can say more precisely the following (see~\cite[\S2]{billey1993schubert}): if $w \in \mathfrak{S}_n$ is fully commutative, then after deleting empty rows and columns and reflecting across a vertical line, the Rothe diagram of $w$ becomes a skew shape $\sigma=\lambda/\nu$; and in this case $[e,w] \simeq [\nu,\lambda]$. So we see that fully commutative permutations are skew vexillary. We also see that the intervals of weak order which are distributive lattices are isomorphic to intervals of Young's lattice.

In some sense the skew vexillary permutations are the ``smallest'' natural family of permutations containing both the vexillary and the fully commutative permutations.

We note that the intersection of the set of vexillary permutations and the set of fully commutative permutations is exactly the set of permutations which are either Grassmannian or inverse Grassmannian (see~\cite[Corollary 2.5]{billey1993schubert}). (Recall the notion of Grassmannian permutation was introduced in Definition~\ref{def:gr}; and $w\in \mathfrak{S}_n$ is \emph{inverse Grassmannian} if its inverse $w^{-1}$ is Grassmannian.) For an inverse Grassmannian permutation $w$, we have $[e,w]\simeq [\varnothing,\lambda]$ for some straight shape $\lambda$; and similarly, for a Grassmannian permutation $w$ we have  $[e,w]\simeq [\varnothing,\lambda]^*$ for some straight shape $\lambda$.

Figure~\ref{fig:perm_classes} depicts the containment relationships between the various classes of permutations discussed in this section. We remark that all these classes are closed under inversion, except that the inverse of a Grassmannian permutation is inverse Grassmannian, and vice-versa. Meanwhile, all these classes except for the dominant permutations are closed under reverse-complementation.

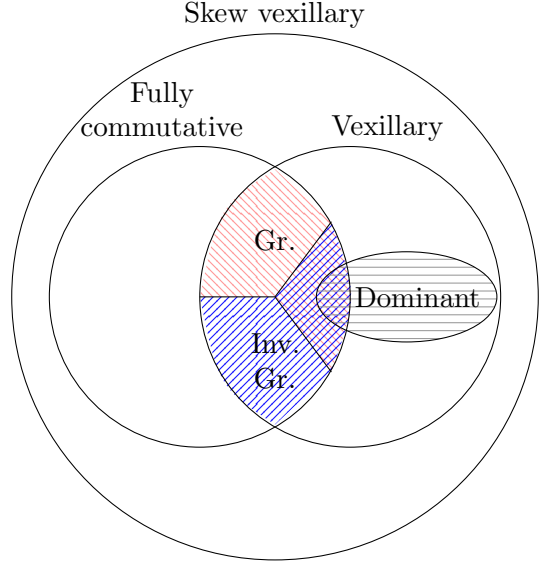
\begin{figure}
\begin{tikzpicture}
\fill[pattern=north west lines, pattern color=red!40] (-1,0)  .. controls (-0.75,1.2) .. (0,1.7)--(0.75,1) .. controls (1.1,0) .. (0.75,-1) -- (0,0);
\fill[pattern=north east lines, pattern color=blue!90] (-1,0)  .. controls (-0.75,-1.2) .. (0,-1.7)--(0.75,-1) .. controls (1.1,0) .. (0.75,1) -- (0,0);
\fill[pattern=horizontal lines, pattern color=black!40] (1.75,0) ellipse (1.2 and 0.6);
\draw (1.75,0) ellipse (1.2 and 0.6);
\node at (0,3.75) {Skew vexillary};
\node at (1.5,2.25) {Vexillary};
\node at (-1.5,2.5) {\parbox{1in}{\begin{center}Fully \\ commutative\end{center}}};
\node at (0,0.75) {Gr.};
\node at (0,-0.75) {\parbox{0.25in}{\begin{center}Inv. \\ Gr.\end{center}}};
\node at (1.9,0) {Dominant};
\draw (0,0) circle (3.5);
\draw (1,0) circle (2);
\draw (-1,0) circle (2);
\draw (-1,0) -- (0,0);
\draw (0,0) -- (0.75,1);
\draw (0.75,-1) -- (0,0);
\end{tikzpicture}
\caption{The containment relationships between the various classes of permutations discussed in Section~\ref{sec:skew_vex_def}. Permutations which are both Grassmannian and inverse Grassmannian are called \emph{bigrassmannian} permutations. The bigrassmannian permutations include the dominant permutations of rectangular shape, which are those permutations that simultaneously avoid $132$ and $321$.} \label{fig:perm_classes}
\end{figure}

\begin{remark} \label{rem:pattern_avoid}
As mentioned above, many of the important sub-families of skew vexillary permutations are characterized by pattern avoidance conditions. We doubt that the skew vexillary permutations themselves are characterized by pattern avoidance. This is because, based on computations with Sage~\cite{Sage-Combinat, sagemath}, all $120$ permutations in~$\mathfrak{S}_5$ are skew vexillary, whereas exactly $682$ of the $720$ permutations in $\mathfrak{S}_6$ are skew vexillary. Thus, if there was a pattern avoidance classification of skew vexillary permutations, it would have to include at a minimum the avoidance of $38$ patterns of length six.
\end{remark}

\begin{remark} \label{rem:stable_grothendieck}
Reiner-Tenner-Yong~\cite[Theorem 5.1]{reiner2018poset} showed that for any vexillary permutation $w$ of shape $\lambda$, one has $\mathbb{E}(\mathrm{maxchain}_{[e,w]};\mathrm{ddeg})=\mathbb{E}(\mathrm{maxchain}_{[\varnothing,\lambda]};\mathrm{ddeg})$. Let us explain how they did this. First note that if $P$ is a graded poset of rank $r$ then
\begin{equation} \label{eqn:maxchain}
 \mathbb{E}(\mathrm{maxchain}_P;\mathrm{ddeg}) = \frac{\#\{(C,x,y)\colon \textrm{$C$ is a maximal chain of $P$, $y \in C$, and $x \lessdot y$} \}}{(r+1)\cdot\#\{C\colon \textrm{$C$ is a maximal chain of $P$}\}}.
 \end{equation}
Reiner-Tenner-Yong in fact showed something stronger than~$\mathbb{E}(\mathrm{maxchain}_{[e,w]};\mathrm{ddeg})=\mathbb{E}(\mathrm{maxchain}_{[\varnothing,\lambda]};\mathrm{ddeg})$ when $w$ is a vexillary permutation of shape $\lambda$; they showed that all of the individual quantities appearing in the quotient~\eqref{eqn:maxchain} are the same for $P=[e,w]$ and for $P=[\varnothing,\lambda]$. Here is their argument. First of all, it is clear in this situation that $[e,w]$ and $[\varnothing,\lambda]$ are both graded posets of rank~$\ell(w) = |\lambda|$. Next, recall that the number of maximal chains of~$[e,w]$ is the number of reduced words of $w$, and the number of maximal chains of~$[\varnothing,\lambda]$ is the number of Standard Young Tableaux of shape $\lambda$. These quantities can be read off from the stable Schubert polynomial (a.k.a., Stanley symmetric function) $F_w$ of $w$, and from the Schur function $s_{\lambda}$ of $\lambda$, respectively. Thus the equality of these two numbers when $w$ is a vexillary permutation of shape~$\lambda$ follows from the fact that in this case $F_w = s_{\lambda}$~\cite[Corollary 4.2]{stanley1984reduced}. (A bijective proof of this equality is given by the famous Edelman-Greene bijection~\cite{edelman1987balanced}.) Finally, the quantity 
\[ \#\{(C,x,y)\colon \textrm{$C$ is a maximal chain of $P$, $y \in C$, and $x \lessdot y$} \} \] 
is the number of ``nearly reduced words'' for $w$ when $P=[e,w]$, and is the number of ``barely set-valued tableaux'' of shape~$\lambda$ when $P=[\varnothing,\lambda]$ (these terms were coined by Reiner-Tenner-Yong). These quantities can be read off from the stable Grothendieck polynomial $G_w$ of $w$, and from the stable Grothendieck polynomial $G_{\lambda}$ of shape~$\lambda$, respectively. So the equality of these two numbers when $w$ is a vexillary permutation of shape~$\lambda$ follows from the fact that in this case $G_w = G_{\lambda}$ (see~\cite[Lemma 5.4]{reiner2018poset}, who cite a formula from~\cite{knutson2009grobner}). It is reasonable to ask whether the above approach can be extended to address skew vexillary permutations as well. When $w$ is a skew vexillary permutation of skew shape $\sigma = \lambda/\nu$, we do have that $F_w = s_{\lambda/\nu}$ (this follows from the work of~\cite{reiner1995key} and~\cite{kraskiewicz2004schubert}; see~\cite[Proposition 2.4]{billey1993schubert}).\footnote{Note however that unlike the situation for vexillary permutations, it is not known whether $F_w$ being a skew Schur function is equivalent to $w$ being skew vexillary. Complicating matters is the fact that unlike for usual Schur functions, there can be nontrivial equalities (i.e., beyond $s_{\sigma}=s_{\sigma^{\mathrm{rot}}}$) for skew Schur functions~\cite{billera2006decomposable, reiner2007coincidences, mcnamara2009towards}.} There is a notion of skew stable Grothendieck polynomial $G_{\lambda/\nu}$ going back to Buch~\cite{buch2002littlewood}. We suspect that when $w$ is a skew vexillary permutation of shape $\lambda/\nu$ we have $G_w=G_{\lambda/\nu}$, but we have not found any statement to this effect in the literature (and we do not have much computational evidence to support our suspicion). If this were the case, then one could conclude that $\mathbb{E}(\mathrm{maxchain}_{[e,w]};\mathrm{ddeg})=\mathbb{E}(\mathrm{maxchain}_{[\nu,\lambda]};\mathrm{ddeg})$ for $w$ a skew vexillary permutation of shape $\lambda/\nu$. We do not know if this equality of expectation always holds. Our methods will avoid any use of Schubert polynomials, Grothendieck polynomials, et cetera, but will apply only to certain skew shapes $\sigma=\lambda/\nu$ (the balanced ones).
\end{remark}

\begin{remark} \label{rem:dominant_counting_edges}
It is worth contrasting what happens for the maxchain distribution (as discussed in the previous Remark~\ref{rem:stable_grothendieck}) with what happens for the uniform distribution. First of all, it certainly need not be the case that $\mathbb{E}(\mathrm{uni}_{[e,w]};\mathrm{ddeg})=\mathbb{E}(\mathrm{uni}_{[\varnothing,\lambda]};\mathrm{ddeg})$ when $w$ is a vexillary permutation of shape $\lambda$ (see~\cite[Example 5.2]{reiner2018poset}). Even when this equality of expectations does hold (which we will show happens for instance when $\lambda$ is balanced), it need not be the case that the terms of the quotient 
\[\mathbb{E}(\mathrm{uni}_P;\mathrm{ddeg})=\#P/\#\{y\lessdot x \in P\}\] 
match up for $P=[e,w]$ and $P=[\varnothing,\lambda]$ (see Examples~\ref{ex:21} and~\ref{ex:322_11} below). This is another difference from the case of the maxchain distribution. In order to compute $\mathbb{E}(\mathrm{uni}_{[e,w]};\mathrm{ddeg})$ for dominant permutations $w$ of rectangular staircase shape, Reiner-Tenner-Yong directly computed the number of elements and number of edges of (the Hasse diagram of) the poset $[e,w]$. To do this, they explained how for any permutation $w$ the quantity $\#[e,w]$ is given by the number of linear extensions of some other poset~$P_w$ (the \emph{noninversion poset} of~$w$; see~\cite[\S5.2]{reiner2018poset}), and showed that when $w$ is a dominant permutation this poset $P_w$ is a \emph{forest} poset, which means that it has a simple product formula computing its number of linear extensions. The number of edges can similarly be obtained by counting linear extensions of certain quotients of $P_w$. It appears that the techniques used by Reiner-Tenner-Yong to compute $\mathbb{E}(\mathrm{uni}_{[e,w]};\mathrm{ddeg})$ for dominant permutations rely heavily on very special properties of dominant permutations (e.g., the poset $P_w$ will not be a forest in general), and so cannot be adapted to other vexillary permutations. We believe it is hopeless to try to compute the number of elements and number of edges of the poset $[e,w]$ in general.\footnote{Indeed, Dittmer and Pak~\cite{dittmer2018counting} give a precise sense in which counting the number of elements of an initial interval of weak order is a computationally intractable problem.} We will only obtain formulas for the edge densities of posets, not for the number of edges or number of elements.
\end{remark}

\subsection{Basic properties of initial weak order intervals for skew vexillary permutations} \label{sec:skew_vex_basic_props}

We want to study the collection of permutations $w$ which are skew vexillary of some fixed shape $\sigma = \lambda/\nu$ (or more precisely, the collection of isomorphism classes of posets $[e,w]$ for such $w$) as one ``family.'' Let us review some basic properties of this family and then go over some examples.
 
\begin{prop} \label{prop:rothe_disconnected}
Let $w \in \mathfrak{S}_n$. Let $\Gamma_w$ denote the graph on the boxes of the Rothe diagram of $w$ where two boxes are adjacent if they belong to the same row or the same column. Suppose that $\Gamma_w$ is disconnected. Then there is some $1 \leq k < n$ for which we can write $w =(u,v) \in \mathfrak{S}_k \times \mathfrak{S}_{n-k}\subseteq \mathfrak{S}_n$ with $u\neq e$ and $v \neq e$.
\end{prop}
\begin{proof}
Imagine building up the Rothe diagram of $w$ row-by-row. Let $i$ be the first row which contains a box of the Rothe diagram in it (such a row exists because $\Gamma_w$ is disconnected). All the boxes in row~$i$ are certainly connected in $\Gamma_w$, so there must be some other row which has a box in it as well. So let~$i\leq k < n$ and suppose that the $(k+1)$st row of the Rothe diagram contains a box in it. Suppose further that $\{1,\ldots,k\}\neq \{w(1),\ldots,w(k)\}$. Then let $x \in \{1,\ldots,k\}$ be minimal such that $x \notin \{w(1),\ldots,w(k)\}$. Note that $w(k) > x$; indeed otherwise the $(k+1)$st row would not contain a box of the Rothe diagram in it. Thus there is a box in the $(k+1)$st row and $x$th column of the Rothe diagram. But there also must be $1 \leq j \leq k$ such that there is a box in the $j$th row and $x$th column, because some $j \in \{1,\ldots,k\}$ must have $w(j)>x$. So in this case all boxes in the $(k+1)$st row belong to the connected component of $\Gamma_w$ containing a box in some previous row. Thus if for each $i \leq k < n$ we have that $\{1,\ldots,k\}\neq\{w(1),\ldots,w(k)\}$, by induction we get that $\Gamma_w$ must be connected. Since $\Gamma_w$ is disconnected by assumption, there is $i \leq k < n$ with $\{1,\ldots,k\}= \{w(1),\ldots,w(k)\}$. Writing $w =(u,v) \in \mathfrak{S}_k \times \mathfrak{S}_{n-k}\subseteq \mathfrak{S}_n$, we see that $u\neq e$ because $k\geq i$; by choosing the minimal such $k$ we can also guarantee~$v \neq e$.
\end{proof}

\begin{cor} \label{cor:shape_disconnected}
Let $\sigma$ be a disconnected skew shape. Let $w$ be a skew vexillary permutation of shape $\sigma$. Then $[e,w] \simeq [e,u]\times [e,v]$, where $u$ is skew vexillary of shape $\sigma_1\neq \varnothing$ and $v$ is skew vexillary of shape $\sigma_2\neq \varnothing$, with $\sigma=\sigma_1\sqcup\sigma_2$.
\end{cor}
\begin{proof}
Permuting rows and columns preserves the relationship ``being in the same row or column'' for two boxes of a diagram, so $\sigma$ being disconnected implies that the graph~$\Gamma_w$ from Proposition~\ref{prop:rothe_disconnected} is disconnected. Hence there is some $1 \leq k < n$ with $w = (u,v) \in \mathfrak{S}_k \times \mathfrak{S}_{n-k}\subseteq \mathfrak{S}_n$ such that $u\neq e$ and $v \neq e$. The boxes of the Rothe diagram in the upper $k\times k$ square of $[n]\times[n]$ are equal, up to permutation of rows and columns, to a skew shape $\sigma_1\neq \varnothing$. In other words, $u$ is skew vexillary of shape~$\sigma_1$. Hence, $v$ must be skew vexillary of shape $\sigma_2$, where $\sigma=\sigma_1\sqcup\sigma_2$. The fact that $w = (u,v)$ implies that $[e,w] \simeq [e,u]\times [e,v]$, as desired.
\end{proof}

The properties of posets we are interested in studying (e.g., the CDE property) are preserved under Cartesian products (at least for graded posets; see, e.g.,~\cite[Proposition 2.13]{reiner2018poset}). So in light of Corollary~\ref{cor:shape_disconnected}, we will throughout the rest of the paper restrict our attention to connected shapes.

By the \emph{initial fixed points} of a permutation $w \in \mathfrak{S}_n$ we mean the maximal sequence of letters $w_1,w_2,\cdots,w_k$ for which $w(i) =i $ for all $1 \leq i \leq k$. Similarly, the \emph{terminal fixed points} are the maximal sequence of letters $w_{k},w_{k+1},\cdots,w_n$ for which $w(i) =i $ for all $k \leq i \leq n$. In the following proposition and corollary, let us use $\widetilde{w}$ to denote the permutation obtained from $w$ by deleting its initial and terminal fixed points and then reindexing (so $\widetilde{w}$ will belong to a smaller symmetric group than~$w$, in general).

\begin{prop} \label{prop:finite}
Let $\sigma$ be a connected skew shape of height $a$ and width $b$ and~$w \in \mathfrak{S}_n$ a skew vexillary permutation of shape $\sigma$. Then $\widetilde{w} \in \mathfrak{S}_m$ where $m \leq a+b$.
\end{prop}
\begin{proof}
First observe that $\widetilde{w}$ is also skew vexillary of shape $\sigma$. Now suppose $1 \leq k \leq m$ is such that $(i,k),(k,j) \notin \mathrm{Inv}(\widetilde{w})$ for all $i \in \{1,\ldots,k-1\}$ and $j \in \{k+1,\ldots,m\}$. In other words, $\widetilde{w}(i) \leq \widetilde{w}(k)$ for all $i \in \{1,\ldots,k-1\}$ and $\widetilde{w}(j) \geq \widetilde{w}(k)$ for all $j \in \{k+1,\ldots,m\}$. This implies that $\widetilde{w} \in \mathfrak{S}_k \times \mathfrak{S}_{n-k}$. Because $\widetilde{w}$ has no initial or terminal fixed points, both the upper $k\times k$ and the lower $(n-k)\times(n-k)$ square of $[n]\times[n]$ must have some boxes of the Rothe diagram of $\widetilde{w}$ in them. That $\widetilde{w} \in \mathfrak{S}_k \times \mathfrak{S}_{n-k}$ thus implies this Rothe diagram is disconnected. But that would imply that $\sigma$ is disconnected. Hence for each $1 \leq k \leq m$ there must be $i \in \{1,\ldots,k-1\}$ with $(i,k) \in \mathrm{Inv}(\widetilde{w})$ or $j \in \{k+1,\ldots,m\}$ with $(k,j) \in \mathrm{Inv}(\widetilde{w})$. In other words, each $1 \leq k \leq m$ corresponds to either a row or a column of $\sigma$ (possibly both). But the number of rows plus the number of columns of $\sigma$ equals $a+b$, so indeed $m \leq a+b$.
\end{proof}

\begin{cor} \label{cor:finite}
Let $\sigma$ be a connected skew shape of height $a$ and width $b$. Let $w$ be a skew vexillary permutation of shape $\sigma$. Then $[e,w]\simeq [e,u]$, where $u \in \mathfrak{S}_{a+b}$ is a skew vexillary permutation of shape $\sigma$.
\end{cor}
\begin{proof}
Initial and terminal fixed points are not ``seen'' by weak order, so $[e,w]\simeq [e,\widetilde{w}]$. We have $\widetilde{w} \in \mathfrak{S}_m$ with $m \leq a+b$  by Proposition~\ref{prop:finite}; then by appending some terminal fixed points we can obtain the desired $u \in \mathfrak{S}_{a+b}$.
\end{proof}

Corollary~\ref{cor:finite} (together with  Corollary~\ref{cor:shape_disconnected}) says that for any fixed shape $\sigma = \lambda / \nu$, the collection of isomorphism classes of posets $[e,w]$ for $w$ a skew vexillary permutation of shape $\sigma$ is always finite. Let us make a few more remarks about this collection of (isomorphism classes of) posets.

This collection always contains the distributive lattices $[\nu,\lambda]$ and $[\nu,\lambda]^*$ which correspond to fully commutative~$w$. It is not hard to show that if $\sigma = b^a$ is a rectangle, then this collection consists of a single element which is the self-dual poset $[\varnothing,b^a]$; and otherwise, this collection has some other posets in it which are not distributive lattices (see also Section~\ref{sec:constructing_skew_vex}). As we will see in the examples below, the distributive lattices~$[\nu,\lambda]$ and $[\nu,\lambda]^*$ appear to be the ``smallest'' posets in this collection.

If $\sigma=\lambda$ is actually a straight shape, then this collection also always has the poset~$[e,w]$ for $w$ dominant of shape $\lambda$. This dominant poset appears to be the ``biggest'' poset in the collection.

Another observation about this collection: it is closed under duality. This is because, as mentioned at the beginning of Section~\ref{sec:skew_vex_def}, if $w$ is skew vexillary of shape $\sigma$ then so is $(w^{\mathrm{rc}})^{-1}$; but $[e,w]\simeq [e,w^{\mathrm{rc}}]$, and hence $[e,w]\simeq [e,(w^{\mathrm{rc}})^{-1}]^{*}$.

Now let's see some examples of these collections of posets.

\begin{figure}
\begin{tikzpicture}
\node (1) at (0,0) {$1234$};
\node (2) at (0,1) {$1324$};
\node (3) at (1,2) {$3124$};
\node (4) at (-1,2) {$1342$};
\node (5) at (0,3) {$3142$};
\draw (1) -- (2);
\draw (2) -- (3);
\draw (2) -- (4);
\draw (3) -- (5);
\draw (4) -- (5);
\end{tikzpicture} \qquad
\begin{tikzpicture}
\node (1) at (0,0) {$1234$};
\node (2) at (-1,1) {$1243$};
\node (3) at (1,1) {$2134$};
\node (4) at (0,2) {$2143$};
\node (5) at (0,3) {$2413$};
\draw (1) -- (2);
\draw (1) -- (3);
\draw (2) -- (4);
\draw (3) -- (4);
\draw (4) -- (5);
\end{tikzpicture} \qquad
\begin{tikzpicture}
\node (1) at (0,0) {$123$};
\node (2) at (-1,1) {$213$};
\node (3) at (1,1) {$132$};
\node (4) at (-1,2) {$231$};
\node (5) at (1,2) {$312$};
\node (6) at (0,3) {$321$};
\draw (1) -- (2);
\draw (1) -- (3);
\draw (2) -- (4);
\draw (3) -- (5);
\draw (4) -- (6);
\draw (5) -- (6);
\end{tikzpicture}
\caption{Example~\ref{ex:21}: the initial weak order intervals corresponding to vexillary permutations of shape~$(2,1)$.} \label{fig:21}
\end{figure}

\begin{example} \label{ex:21}
Let $\lambda = (2,1)$. Up to isomorphism there are three posets of the form~$[e,w]$ for $w$ a vexillary permutation of shape~$\lambda$: these are $[1234,3142]\simeq[\varnothing,(2,1)]$ for the inverse Grassmannian permutation $3142$; $[1234,2413]\simeq[\varnothing,(2,1)]^{*}$ for the Grassmannian permutation $2413$; and the self-dual poset $[123,321]$ for the dominant permutation $321$. These are depicted in Figure~\ref{fig:21}.  One can check that all these posets are CDE with edge density~$1$.
\end{example}

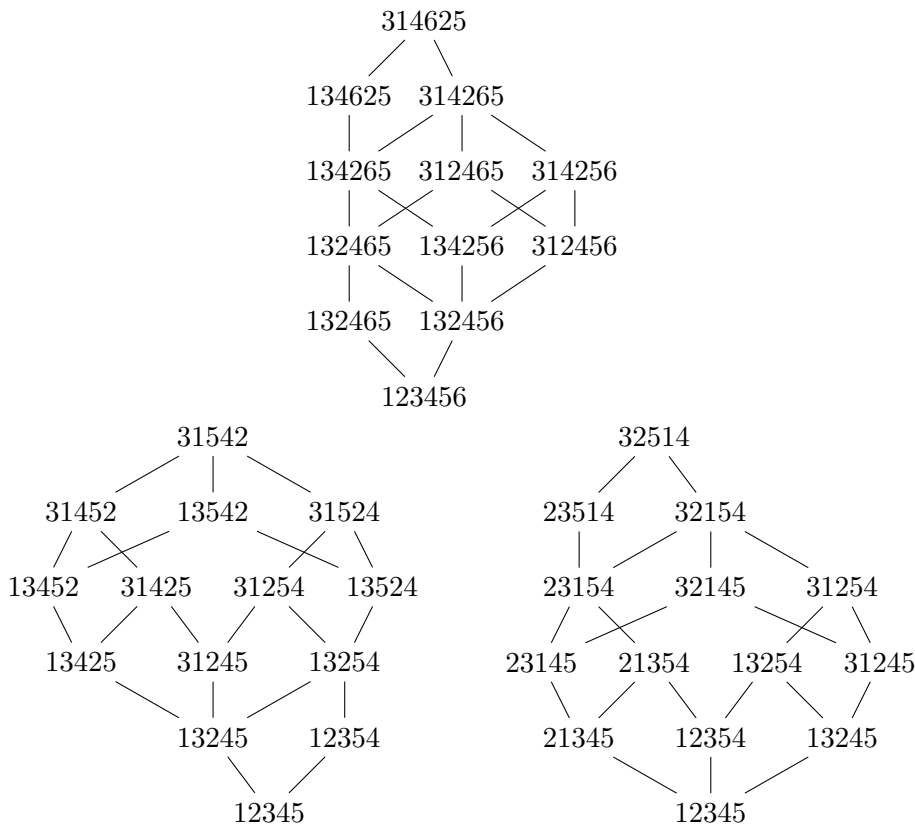
\begin{figure}
\begin{tikzpicture}
\node (1) at (0,0) {$123456$};
\node (2) at (-1,1) {$132465$};
\node (3) at (0.5,1) {$132456$};
\node (4) at (-1,2) {$132465$};
\node (5) at (0.5,2) {$134256$};
\node (6) at (2,2) {$312456$};
\node (7) at (-1,3) {$134265$};
\node (8) at (0.5,3) {$312465$};
\node (9) at (2,3) {$314256$};
\node (10) at (-1,4) {$134625$};
\node (11) at (0.5,4) {$314265$};
\node (12) at (0,5) {$314625$};
\draw (1) -- (2);
\draw (1) -- (3);
\draw (2) -- (4);
\draw (3) -- (4);
\draw (3) -- (5);
\draw (3) -- (6);
\draw (4) -- (7);
\draw (4) -- (8);
\draw (5) -- (7);
\draw (5) -- (9);
\draw (6) -- (8);
\draw (6) -- (9);
\draw (7) -- (10);
\draw (7) -- (11);
\draw (8) -- (11);
\draw (9) -- (11);
\draw (10) -- (12);
\draw (11) -- (12);
\end{tikzpicture}

\begin{tikzpicture}
\node (1) at (2,0) {$12345$};
\node (2) at (1.25,1) {$13245$};
\node (3) at (3,1) {$12354$};
\node (4) at (-0.5,2) {$13425$};
\node (5) at (1.25,2) {$31245$};
\node (6) at (3,2) {$13254$};
\node (7) at (-1,3) {$13452$};
\node (8) at (0.5,3) {$31425$};
\node (9) at (2,3) {$31254$};
\node (10) at (3.5,3) {$13524$};
\node (11) at (-0.5,4) {$31452$};
\node (12) at (1.25,4) {$13542$};
\node (13) at (3,4) {$31524$};
\node (14) at (1.25,5) {$31542$};
\draw (1) -- (2);
\draw (1) -- (3);
\draw (2) -- (4);
\draw (2) -- (5);
\draw (2) -- (6);
\draw (3) -- (6);
\draw (4) -- (7);
\draw (4) -- (8);
\draw (5) -- (8);
\draw (5) -- (9);
\draw (6) -- (9);
\draw (6) -- (10);
\draw (7) -- (11);
\draw (7) -- (12);
\draw (8) -- (11);
\draw (9) -- (13);
\draw (10) -- (12);
\draw (10) -- (13);
\draw (11) -- (14);
\draw (12) -- (14);
\draw (13) -- (14);
\end{tikzpicture} \qquad
\begin{tikzpicture}
\node (2) at (1.25,1) {$12345$};
\node (4) at (-0.5,2) {$21345$};
\node (5) at (1.25,2) {$12354$};
\node (6) at (3,2) {$13245$};
\node (7) at (-1,3) {$23145$};
\node (8) at (0.5,3) {$21354$};
\node (9) at (2,3) {$13254$};
\node (10) at (3.5,3) {$31245$};
\node (11) at (-0.5,4) {$23154$};
\node (12) at (1.25,4) {$32145$};
\node (13) at (3,4) {$31254$};
\node (14) at (1.25,5) {$32154$};
\node (1) at (-0.5,5) {$23514$};
\node (3) at (0.5,6) {$32514$};
\draw (2) -- (4);
\draw (2) -- (5);
\draw (2) -- (6);
\draw (4) -- (7);
\draw (4) -- (8);
\draw (5) -- (8);
\draw (5) -- (9);
\draw (6) -- (9);
\draw (6) -- (10);
\draw (7) -- (11);
\draw (7) -- (12);
\draw (8) -- (11);
\draw (9) -- (13);
\draw (10) -- (12);
\draw (10) -- (13);
\draw (11) -- (14);
\draw (12) -- (14);
\draw (13) -- (14);
\draw (11) -- (1);
\draw (1) -- (3);
\draw (14) -- (3);
\end{tikzpicture}
\caption{Example~\ref{ex:322_11}: the initial weak order intervals corresponding to skew vexillary permutations of shape~$(3,2,2)/(1,1)$.} \label{fig:322_11}
\end{figure}

\begin{example} \label{ex:322_11}
Let $\sigma = (3,2,2)/(1,1)$. Up to isomorphism there are three posets of the form $[e,w]$ for $w$ a skew vexillary permutation of shape~$\sigma$: these are the self-dual poset $[123456,314625]\simeq[(1,1),(3,2,2)]$ for the fully commutative permutation~$314625$; and the dual posets $[12345,31542]$ and $[12345,32514]$. These are depicted in Figure~\ref{fig:322_11}.  One can check that all these posets are CDE with edge density~$3/2$.
\end{example}

Examples~\ref{ex:21} and~\ref{ex:322_11} both concern balanced shapes so our main result will explain why all the posets in these examples are CDE.

\section{Toggle-symmetric distributions for semidistributive lattices} \label{sec:tcde}

\subsection{Toggling for semidistributive lattices} \label{sec:tcde_def}
Let $L$ be a semidistributive lattice. We now explain a canonical labeling of the cover relations (a.k.a., edges of the Hasse diagram) of $L$ which was first introduced by Barnard~\cite{barnard2016canonical}, building on work of Reading~\cite{reading2015noncrossing} in the case of weak order on $\mathfrak{S}_n$. 

For any cover relation $x \lessdot y \in L$, we define the edge labeling
\[\gamma(x \lessdot y) \coloneqq  \mathrm{min}\{z\in L\colon x\join z = y\}. \]
The semidistributivity of $L$ guarantees that for any $x\lessdot y$ the set $\{z\in L\colon x\join z = y\}$ has a unique minimal element (see~\cite[Proposition 3.4]{barnard2016canonical}), and thus that this label $\gamma(x\lessdot y)\in L$ always exists. Moreover, it is easy to see that we always have $\gamma(x\lessdot y) \in \mathrm{Irr}(L)$.

We call $\gamma$ the \emph{canonical $\gamma$-labeling} of the edges (of the Hasse diagram) of $L$. As mentioned, these labels are join irreducible elements of $L$.

For example, consider the case of a distributive lattice $L= \mathcal{J}(P)$: for $I,J \in \mathcal{J}(P)$ with $I\lessdot J$ we have $\gamma(I\lessdot J) = p$ where $p\in P$ is such that $J = I\cup\{p\}$. This fundamental example explains the ``toggling'' terminology: the toggles we define below will toggle the status of $p$ in $I$ (when possible).

It follows from work of Barnard that for any $y \in L$, among edges incident to $y$ (i.e., cover relations in which $y$ is involved, either in the form $x \lessdot y$ or $y\lessdot z$), each join irreducible element appears as a $\gamma$-label at most once. Indeed, she showed moreover that the canonical join representation of any $y \in L$ is $\join \{\gamma(x\lessdot y)\colon x \in P \textrm{ with $x\lessdot y$}\}$ (see~\cite[Lemma 3.3]{barnard2016canonical}). 

Barnard's results allow us to define a notion of toggling in this semidistributive context (see also Thomas-Williams~\cite{thomas2017rowmotion}). For each join irreducible element $p\in \mathrm{Irr}(L)$ we define \emph{toggling at $p$} to be the involution $\tau_p\colon L \to L$ defined by
\[ \tau_p(y) \coloneqq  \begin{cases} x &\textrm{ if $\gamma(x\lessdot y) = p$}; \\ z &\textrm{ if $\gamma(y\lessdot z) = p$}; \\ y &\textrm{otherwise}. \end{cases} \]
Note that $\tau_p(y)$ is well-defined precisely because at most one edge incident to $y$ has $\gamma$-label $p$. This notion of toggle generalizes the toggles studied by Striker and Williams~\cite{striker2012promotion} in the distributive lattice setting (and discussed in Section~\ref{sec:intro}).

For $y\in L$ and $p \in \mathrm{Irr}(L)$, we say that \emph{$p$ can be toggled into $y$} if $y \lessdot \tau_p(y)$; similarly, we say that \emph{$p$ can be toggled out of~$y$} if $\tau_p(y)\lessdot y$. Let us also define the toggleability statistics $\mathcal{T}^+_p, \mathcal{T}^-_p,\mathcal{T}_p\colon L\to\mathbb{Z}$ by
\begin{align*}
\mathcal{T}^+_p(y) &\coloneqq  \begin{cases} $1$ &\textrm{if $p$ can be toggled into $y$}; \\ 0 &\textrm{otherwise};\end{cases} \\
\mathcal{T}^-_p(y) &\coloneqq  \begin{cases} $1$ &\textrm{if $p$ can be toggled out of $y$}; \\ 0 &\textrm{otherwise};\end{cases}  \\
\mathcal{T}_p(y) &\coloneqq  \mathcal{T}^+_p(y)-\mathcal{T}^-_p(y). \\
\end{align*}

\begin{definition}
Let $\mu$ be a probability distribution on $L$. We say that $\mu$ is \emph{toggle-symmetric} if $\mathbb{E}(\mu;\mathcal{T}_p)=0$ for all $p\in \mathrm{Irr}(L)$.
\end{definition}

This notion of toggle-symmetric distribution generalizes the notion for distributive lattices defined in~\cite{chan2017expected} (and discussed in Section~\ref{sec:intro}). 

\begin{prop} \label{prop:uni_toggle_sym}
The uniform distribution $\mathrm{uni}_L$ is toggle-symmetric.
\end{prop}
\begin{proof}
For each $p \in \mathrm{Irr}(L)$ and $y \in L$, we have $\mathcal{T}_p(y) = -\mathcal{T}_p(\tau_p(y))$; but $y$ and $\tau_p(y)$ are equally probable in the uniform distribution.
\end{proof}

For any semidistributive lattice~$L$ there are some other toggle-symmetric distributions, beyond the uniform distribution, coming from ``rowmotion.'' We will explain these in Section~\ref{sec:semi_rowmotion}.

\begin{definition}
We say that the semidistributive lattice $L$ is \emph{toggle CDE (tCDE)} if $\mathbb{E}(\mu;\mathrm{ddeg})=\mathbb{E}(\mathrm{uni}_L;\mathrm{ddeg})$ for every toggle-symmetric distribution~$\mu$ on~$L$.
\end{definition}

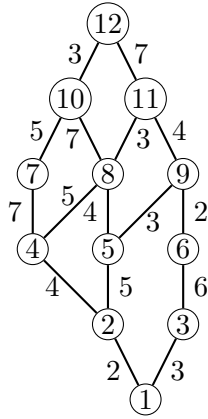
\begin{figure}
 \begin{tikzpicture}
	\node[draw,circle,inner sep=1pt,minimum size=1pt] (1) at (0,0) {1};
	\node[draw,circle,inner sep=1pt,minimum size=1pt] (2) at (-0.5,1) {2};
	\node[draw,circle,inner sep=1pt,minimum size=1pt] (3) at (0.5,1) {3};
	\node[draw,circle,inner sep=1pt,minimum size=1pt] (4) at (-1.5,2) {4};
	\node[draw,circle,inner sep=1pt,minimum size=1pt] (5) at (-0.5,2) {5};
	\node[draw,circle,inner sep=1pt,minimum size=1pt] (6) at (0.5,2) {6};
	\node[draw,circle,inner sep=1pt,minimum size=1pt] (7) at (-1.5,3) {7};
	\node[draw,circle,inner sep=1pt,minimum size=1pt] (8) at (-0.5,3) {8};
	\node[draw,circle,inner sep=1pt,minimum size=1pt] (9) at (0.5,3) {9};
	\node[draw,circle,inner sep=1pt,minimum size=1pt] (10) at (-1,4) {10};
	\node[draw,circle,inner sep=1pt,minimum size=1pt] (11) at (0,4) {11};
	\node[draw,circle,inner sep=1pt,minimum size=1pt] (12) at (-0.5,5) {12};
	\path[draw,thick] (1) edge node[left,pos=0.3] {2} (2);
	\path[draw,thick] (1) edge node[right,pos=0.3] {3} (3);
	\path[draw,thick] (2) edge node[left,pos=0.5] {4} (4);
	\path[draw,thick] (2) edge node[right,pos=0.5] {5} (5);
	\path[draw,thick] (3) edge node[right,pos=0.5] {6} (6);
	\path[draw,thick] (4) edge node[left,pos=0.5] {7} (7);
	\path[draw,thick] (4) edge node[left,pos=0.8] {5} (8);
	\path[draw,thick] (5) edge node[left,pos=0.5] {4} (8);
	\path[draw,thick] (5) edge node[right,pos=0.3] {3} (9);
	\path[draw,thick] (6) edge node[right,pos=0.5] {2} (9);
	\path[draw,thick] (7) edge node[left,pos=0.7] {5} (10);
	\path[draw,thick] (8) edge node[left,pos=0.5] {7} (10);
	\path[draw,thick] (8) edge node[right,pos=0.5] {3} (11);
	\path[draw,thick] (9) edge node[right,pos=0.7] {4} (11);
	\path[draw,thick] (10) edge node[left,pos=0.7] {3} (12);
	\path[draw,thick] (11) edge node[right,pos=0.7] {7} (12);
\end{tikzpicture}
\caption{Example~\ref{ex:tcde_not_cde}: a semidistributive lattice which is tCDE but not CDE.} \label{fig:tcde_not_cde}
\end{figure}

This notion of tCDE generalizes the notion for distributive lattices defined in~\cite{hopkins2017cde} (and discussed in Section~\ref{sec:intro}). In~\cite[Corollary 2.20]{chan2017expected} it was shown that for a distributive lattice $L$ the distribution $\mathrm{maxchain}_L$ is toggle-symmetric; hence, if $L$ is tCDE, then it is CDE. As the next example shows, for a semidistributive lattice $L$ it need not be the case that $\mathrm{maxchain}_L$ is toggle-symmetric, and moreover $L$ being tCDE does {\bf not} necessarily imply that is CDE. (At this point the reader may with some cause object to our terminology.)

\begin{example} \label{ex:tcde_not_cde}
Let $L$ be the semidistributive lattice depicted in Figure~\ref{fig:tcde_not_cde}. In this figure the edges are labeled by the canonical $\gamma$-labels. One can check that in this example $\mathbb{E}(\mathrm{maxchain}_L;\mathcal{T}_3) = \frac{1}{14}$, and so $\mathrm{maxchain}_L$ is not a toggle-symmetric distribution. We claim that $L$ is tCDE with edge density $\frac{4}{3}$. To see this, one can check that we have the following equality of functions $L\to\mathbb{R}$:
\[-\mathcal{T}_2 -\frac{1}{3}\mathcal{T}_3 -\frac{2}{3}\mathcal{T}_4 -\frac{2}{3}\mathcal{T}_5 -\frac{2}{3}\mathcal{T}_6 -\frac{1}{3}\mathcal{T}_7+\frac{4}{3}\mathbf{1}=\mathrm{ddeg},\]
where $\mathbf{1}\colon L\to\mathbb{R}$ denotes the constant function $\mathbf{1}(x)=1$. Hence, if $\mu$ is any toggle-symmetric distribution on $L$, then $\mathbb{E}(\mu;\mathrm{ddeg}) = \frac{4}{3}=\mathbb{E}(\mathrm{uni}_L;\mathrm{ddeg})$ by taking $\mathbb{E}(\mu;\cdot)$ of both sides of the above equality. But~$L$ is {\bf not} CDE since $\mathbb{E}(\mathrm{maxchain}_L;\mathrm{ddeg})=\frac{55}{42}$.
\end{example}
 
In spite of this example, in the rest of this section we will show that in our case of interest (namely, intervals of weak order) tCDE does imply CDE. We do not know exactly which properties are required of semidistributive lattices for tCDE to imply CDE. Note that the $L$ appearing in Example~\ref{ex:tcde_not_cde} is even graded.

\subsection{The canonical edge labeling for weak order} \label{sec:weak_order_edge_labels}

We now explain what that the canonical $\gamma$-labels look like for weak order intervals $[e,w]$. We basically follow the account of Reading~\cite[\S2]{reading2015noncrossing}; see that paper for proofs of these statements. 

Recall that the join irreducible elements of weak order are the nonidentity Grassmannian permutations. We want a way to record these Grassmannian permutations more compactly. Let $(i,j)$ be a pair with $1 \leq i < j \leq n$ and let $\mathbf{x} \subseteq \{i+1,i+2,\ldots,j-1\}$ be any subset. We define the permutation $g((i,j),\mathbf{x})\in \mathfrak{S}_n$ in one-line notation as
\[g((i,j),\mathbf{x}) \coloneqq  1,2,\cdots,(i-1), x_1,x_2,\cdots, x_m, j, i, y_1,y_2,\ldots,y_{(j-i-1)-m},j+1,j+2,\cdots,n\]
where $\mathbf{x}=\{x_1 <\ldots <x_m\}$ and $\{y_1<\ldots<y_{(j-i)-m}\}=\{i+1,\ldots,j-1\}\setminus \mathbf{x}$. The unique descent of $g((i,j),\mathbf{x})$ is $(j,i)$ and hence $g((i,j),\mathbf{x})$ is a Grassmannian permutation. Moreover, all nonidentity Grassmannian permutations are of this form:
\[\mathrm{Irr}(\mathfrak{S}_n) = \{g((i,j),\mathbf{x})\colon1\leq i < j \leq n, \mathbf{x}\subseteq \{i+1,i+2,\ldots,j-1\} \}.\]

Let $u,w \in \mathfrak{S}_n$ with $u \lessdot w$. Thus $w = us_k$ for some $1 \leq k < n$. Then we have that~$\gamma(u\lessdot w) = g((i,j),\mathbf{x})$ where $(i,j) = (u_k,u_{k+1}) = (w_{k+1},w_k)$ and 
\[\mathbf{x} = \{x\colon i+1\leq x\leq j-1, (i,x)\in \mathrm{Inv}^{-1}(w)\}.\]
Recall that in this situation we have that $\mathrm{Inv}^{-1}(w) = \mathrm{Inv}^{-1}(u)\cup \{(i,j)\}$. So the canonical $\gamma$-labels record which inverse inversion is added/removed along each edge, but they record more information than this as well.

Once we know the $\gamma$-labels for weak order, we know all the canonical join representations as well: we can just apply Barnard's result mentioned in Section~\ref{sec:tcde_def} which tells that the canonical join representation of $w\in \mathfrak{S}_n$ is $\join \{\gamma(u\lessdot w)\colon u\in \mathfrak{S}_n, u\lessdot w\}$.

Finally, we remark that for any initial interval $[e,w]$ of weak order, the $\gamma$-labels of cover relations are exactly the same as in the full weak order. Note in particular that $\mathrm{Irr}([e,w]) \subseteq \mathrm{Irr}(\mathfrak{S}_n)$. (These facts are true more generally for initial intervals of any semidistributive lattice.)

\begin{figure}
\begin{tikzpicture}
\node (1) at (0,0) {$123$};
\node (2) at (-1,1) {$213$};
\node (3) at (1,1) {$132$};
\node (4) at (-1,2) {$231$};
\node (5) at (1,2) {$312$};
\node (6) at (0,3) {$321$};
\path[draw,thick] (1) edge node[left,pos=0.3] {\tiny $g((1,2),\varnothing)$} (2);
\path[draw,thick] (1) edge node[right,pos=0.3] {\tiny $g((2,3),\varnothing)$} (3);
\path[draw,thick] (2) edge node[left,pos=0.5] {\tiny$g((1,3),\{2\})$} (4);
\path[draw,thick] (3) edge node[right,pos=0.5] { \tiny $g((1,3),\varnothing)$} (5);
\path[draw,thick] (4) edge node[left,pos=0.8] {\tiny$g((2,3),\varnothing)$} (6);
\path[draw,thick] (5) edge node[right, pos=0.8] { \tiny $g((1,2),\varnothing)$} (6);
\end{tikzpicture}
\caption{The canonical $\gamma$-labeling for $\mathfrak{S}_n$.} \label{fig:321_labels}
\end{figure}
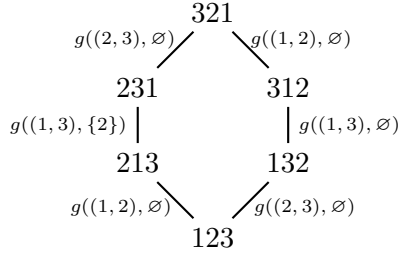

See Figure~\ref{fig:321_labels} for the canonical labeling of~$\mathfrak{S}_3$.

\subsection{The maxchain distribution is toggle-symmetric for weak order intervals} \label{sec:wo_maxchain}

Fix $w\in \mathfrak{S}_n$. In this subsection we will prove that the maxchain distribution is toggle-symmetric for the initial weak order interval $[e,w]$. 

As explained in the previous subsection, the canonical $\gamma$-labels of the edges of $[e,w]$ contain more information than just which inverse inversion is added/removed along each edge. But which inverse inversion is added/removed is of particular salience. Thus for any~$(i,j) \in \mathrm{Inv}^{-1}(w)$ we define the toggleability statistics $\mathcal{T}^{+}_{(i,j)}, \mathcal{T}^{-}_{(i,j)}, \mathcal{T}_{(i,j)}\colon[e,w]\to\mathbb{Z}$ by
\begin{align*}
\mathcal{T}^{+}_{(i,j)} &\coloneqq  \sum_{g((i,j),\mathbf{x})\in\mathrm{Irr}([e,w])} \mathcal{T}^{+}_{g((i,j),\mathbf{x})}; \\
\mathcal{T}^{-}_{(i,j)} &\coloneqq  \sum_{g((i,j),\mathbf{x})\in\mathrm{Irr}([e,w])} \mathcal{T}^{-}_{g((i,j),\mathbf{x})}; \\
\mathcal{T}_{(i,j)} &\coloneqq  \mathcal{T}^{+}_{(i,j)} - \mathcal{T}^{-}_{(i,j)}.
\end{align*}

The next proposition says that, at least if we are focused on maximal chains, we do not lose too much information when passing from the $\mathcal{T}_{g((i,j),\mathbf{x})}$ statistics to the the $\mathcal{T}_{(i,j)}$ statistics.

\begin{prop} \label{prop:gr_nonzero}
Let $u_0\lessdot \cdots \lessdot u_{\ell(w)}$ be a maximal chain of $[e,w]$. Let $(i,j)\in\mathrm{Inv}^{-1}(w)$. If~$\mathcal{T}_{g((i,j),\mathbf{x})}(u_k) \neq 0$ for some $\mathbf{x}$ and some~$0 \leq k \leq \ell(w)$, then for any $\mathbf{y}\neq \mathbf{x}$ we have that $\mathcal{T}_{g((i,j),\mathbf{y})}(u_m) = 0$ for all $0\leq m \leq \ell(w)$.
\end{prop}
\begin{proof}
The sequence $u_{\ell(w)}, u_{\ell(w)-1},\ldots,u_0$ is a ``bubble-sorting process'' for~$w$, i.e., a way of starting from $w$ and successively transposing descents to reach the identity permutation. For instance, with $w=53124$ we could have the sequence
\[5\underline{31}24 \to 51\underline{32}4 \to \underline{51}234 \to 1\underline{52}34 \to 12\underline{53}4 \to 123\underline{54} \to 12345. \]
Let us consider the way the one-line notations of the permutations $u_k$ evolve during this process. At some point in this process the letters $j$ and $i$ will become adjacent: so define $k^-$ be the maximal $0 \leq k^{-} \leq \ell(w)$ such that $(j,i)$ is a descent of~$u_{k^-}$. Also, at some point in this process we will have to transpose $j$ and $i$: so define~$k^+$ to be the maximal $0 \leq k^+ \leq \ell(w)$ such that $(i,j) \notin \mathrm{Inv}^{-1}(u_{k^+})$. Note~$k^+ < k^-$. In the above example, with $(i,j)=(2,5)$, we have that $u_{k^-} = 15234$ and $u_{k^+}=12534$. By definition of~$k^-$ we have $\mathcal{T}_{g((i,j),\mathbf{x})}(u_{k^-}) = -1$ for some $\mathbf{x}$. First note that clearly $\mathcal{T}_{g((i,j),\mathbf{y})}(u_{k}) = 0$ for any $\mathbf{y}$ and $k^- < k \leq \ell(w)$. Next, we claim that for any $k^+ < k \leq k^-$, no letter~$x$ with $x \in \{i+1,\ldots,j-1\}$ can be between $j$ and $i$ in $u_{k}$. Indeed, if $x$ is to the left of~$j$ in $u_{k^-}$ then it cannot be moved to the right of~$j$ because it will not form a descent with $j$; and similarly if $x$ is to the right of $i$ it cannot be moved to the left of $i$; so~$j$ and $i$ act as ``barriers'' which prevent such an~$x$ from coming between them. This implies that~$\mathcal{T}_{g((i,j),\mathbf{y})}(u_{k}) = 0$ for any $\mathbf{y}\neq\mathbf{x}$ and $k^+ < k \leq k^-$. Now, once we have transposed~$j$ and $i$, it is possible for an $x \in \{i+1,\ldots,j-1\}$ to come between $i$ and $j$. But if such an~$x$ does come between them, then by the same logic which says it cannot form a descent on the left with $j$ or on the right with $i$, this $x$ must always remain between $i$ and~$j$ during our sorting process. So in particular if this happens then $i$ and $j$ will never again be adjacent. This implies that  $\mathcal{T}_{g((i,j),\mathbf{y})}(u_{k}) = 0$ for any $\mathbf{y}\neq\mathbf{x}$ and~$0 \leq k \leq k^+$, completing the proof of the proposition.
\end{proof}

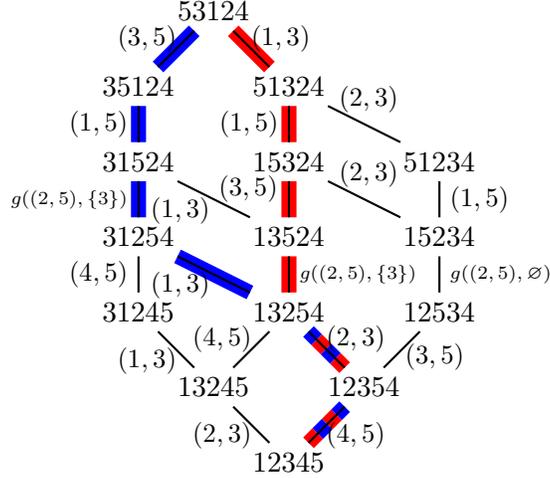
\begin{figure}
\begin{tikzpicture}
\node (1) at (0,0) {$12345$};
\node (2) at (-1,1) {$13245$};
\node (3) at (1,1) {$12354$};
\node (4) at (-2,2) {$31245$};
\node (5) at (0,2) {$13254$};
\node (6) at (2,2) {$12534$};
\node (7) at (-2,3) {$31254$};
\node (8) at (0,3) {$13524$};
\node (9) at (2,3) {$15234$};
\node (10) at (-2,4) {$31524$};
\node (11) at (0,4) {$15324$};
\node (12) at (2,4) {$51234$};
\node (13) at (-2,5) {$35124$};
\node (14) at (0,5) {$51324$};
\node (15) at (-1,6) {$53124$};

\tikzset{
bicolor/.style 2 args={
  dashed,dash pattern=on 5pt off 5pt,#1,
  postaction={draw,dashed,dash pattern=on 5pt off 5pt,#2,dash phase=5pt}
  },
}

\draw[bicolor={red}{blue},line width=0.2cm] (1) -- (3) -- (5);
\draw[red,line width=0.2cm] (5)--(8)--(11)--(14)--(15);
\draw[blue,line width=0.2cm] (5)--(7)--(10)--(13)--(15);

\path[draw,thick] (1) edge node[left,pos=0.2] {\small $(2,3)$} (2);
\path[draw,thick] (1) edge node[right,pos=0.2] {\small $(4,5)$} (3);
\path[draw,thick] (2) edge node[left,pos=0.2] {\small $(1,3)$} (4);
\path[draw,thick] (2) edge node[left,pos=0.8] {\small $(4,5)$} (5);
\path[draw,thick] (3) edge node[right,pos=0.8] {\small $(2,3)$} (5);
\path[draw,thick] (3) edge node[right,pos=0.3] {\small $(3,5)$} (6);
\path[draw,thick] (4) edge node[left,pos=0.5] {\small $(4,5)$} (7);
\path[draw,thick] (5) edge node[left=2mm,pos=0.2] {\small $(1,3)$} (7);
\path[draw,thick] (5) edge node[right,pos=0.5] {\tiny $g((2,5),\{3\})$} (8);
\path[draw,thick] (6) edge node[right,pos=0.5] {\tiny $g((2,5),\varnothing)$} (9);
\path[draw,thick] (7) edge node[left,pos=0.5] {\tiny $g((2,5),\{3\})$} (10);
\path[draw,thick] (8) edge node[left=2mm,pos=0.2] {\small $(1,3)$} (10);
\path[draw,thick] (8) edge node[left=0.1mm,pos=0.8] {\small $(3,5)$} (11);
\path[draw,thick] (9) edge node[right=2mm,pos=1.2] {\small $(2,3)$} (11);
\path[draw,thick] (9) edge node[right,pos=0.5] {\small $(1,5)$} (12);
\path[draw,thick] (10) edge node[left,pos=0.5] {\small $(1,5)$} (13);
\path[draw,thick] (11) edge node[left,pos=0.5] {\small $(1,5)$} (14);
\path[draw,thick] (12) edge node[right=2mm,pos=1.2] {\small $(2,3)$} (14);
\path[draw,thick] (13) edge node[left,pos=0.8] {\small $(3,5)$} (15);
\path[draw,thick] (14) edge node[right,pos=0.8] {\small $(1,3)$} (15);
\end{tikzpicture}
\caption{Example~\ref{ex:disjoint} illustrating Proposition~\ref{prop:gr_nonzero}, and also Example~\ref{ex:involution} depicting the $\mathcal{T}_{(i,j)}$-negating involution $\tau_{(i,j)}$.} \label{fig:involution}
\end{figure}

\begin{example} \label{ex:disjoint}
Let $w = 53124$. The interval $[e,w]$ is depicted in Figure~\ref{fig:involution}. Note that in this figure we have suppressed the $\mathbf{x}$ information on the $\gamma$-labels $g((i,j),\mathbf{x})$ of the edges, except in the case $(i,j)=(2,5)$. Observe that there are two values of $\mathbf{x}\subseteq \{3,4\}$ with $g((2,5),\mathbf{x})\in \mathrm{Irr}([e,w])$: namely, $\mathbf{x}=\varnothing$ and $\mathbf{x}=\{3\}$. The set $A$ of $u\in [e,w]$ with $\mathcal{T}_{g((2,5),\varnothing)}(u)\neq0$ is $A = \{12534,15234\}$. The set $B$ of $u\in [e,w]$ with $\mathcal{T}_{g((2,5),\{3\})}(u)\neq0$ is $B = \{13254,31254,13524,31524\}$. For any maximal chain $C$ of $[e,w]$, either:
\begin{itemize}
\item $C \cap A=\varnothing$ and $C\cap B\neq\varnothing$;
\item or $C\cap B=\varnothing$ and $C\cap A \neq \varnothing$.
\end{itemize}
This is an illustration of Proposition~\ref{prop:gr_nonzero}. The reader can also check in this example that the same thing happens for $(i,j)=(1,5)$.
\end{example}

The next two propositions will allow us to define an involution on the set of maximal chains which negates the $\mathcal{T}_{(i,j)}$ statistic.

\begin{prop} \label{prop:maxchain_toggle1}
Let $u_0  \lessdot \cdots \lessdot u_{\ell(w)}$ be a maximal chain of $[e,w]$. Let $(i,j)\in\mathrm{Inv}^{-1}(w)$. Let $k_1^{-} > k_2^{-} > \cdots > k_a^{-}$ be all the indices $k$ for which $\mathcal{T}_{(i,j)}(u_{k})=-1$. Then for any $1 \leq c \leq a$, the following is also a maximal chain of $[e,w]$:
\[u_0\lessdot \ldots \lessdot u_{k_a^{-}-1} \lessdot s_{(i,j)}u_{k_a^{-}+1}  \lessdot s_{(i,j)}u_{k_a^{-}+2} \lessdot \cdots \lessdot s_{(i,j)}u_{k^{-}_c} \lessdot u_{k^{-}_c}\lessdot \cdots \lessdot u_{\ell(w)}. \]
\end{prop}
\begin{proof}
That $k_a^{-}$ is minimal with $\mathcal{T}_{(i,j)}(u_{k_a^{-}})=-1$ means that going from $u_{k_a^{-}}$ to $u_{k_a^{-}-1}$, we transposed $j$ and $i$. Hence, as was shown in the proof of Proposition~\ref{prop:gr_nonzero}, no letter~$x$ with $x \in \{i+1,\ldots,j-1\}$ can be between $j$ and $i$ in the one-line notation of~$u_{k}$ for any $k_a^{-} \leq k \leq k_1^{-}$; that is, in the language of that proof we have $k^{-} = k_1^{-}$ and $k^{+}=k_a^{-}-1$. Therefore, for any such~$k_a^{-} \leq k \leq k_1^{-}$ we have that
\[\mathrm{Inv}^{-1}(s_{(i,j)}u_k) = \{(s_{(i,j)}(x),s_{(i,j)}(y))\colon (x,y)\in\mathrm{Inv}^{-1}(u_k), (x,y)\neq(i,j) \}. \]
Also note that $(i,j) \in \mathrm{Inv}^{-1}(u_k)$ for all $k_a^{-} \leq k \leq k_1^{-}$. Thus for  $k_a^{-} \leq k < k_1^{-}$, the fact that $u_k\lessdot u_{k+1}$ implies that $s_{(i,j)}u_k\lessdot s_{(i,j)}u_{k+1}$. And we certainly have $s_{(i,j)}u_{k^{-}_c} \lessdot u_{k^{-}_c}$ because $\mathcal{T}_{(i,j)}(u_{k_c})=-1$. Finally, as mentioned above, $u_{k_a^{-}-1}=s_{(i,j)}u_{k_a^{-}}$, which gives us~$u_{k_a^{-}-1} \lessdot s_{(i,j)}u_{k_a^{-}+1}$.
\end{proof}

\begin{prop} \label{prop:maxchain_toggle2}
Let $u_0  \lessdot \cdots \lessdot u_{\ell(w)}$ be a maximal chain of $[e,w]$. Let $(i,j)\in\mathrm{Inv}^{-1}(w)$. Let $k_1^{+} < k_2^{+} < \cdots < k_b^{+}$ be all the indices $k$ for which $\mathcal{T}_{(i,j)}(u_{k})=1$. Then for any $1 \leq c \leq b$, the following is also a maximal chain of $[e,w]$:
\[u_0\lessdot \ldots \lessdot u_{k^+_c} \lessdot s_{(i,j)}u_{k^+_c} \lessdot  s_{(i,j)}u_{k^+_c+1} \lessdot \cdots \lessdot s_{(i,j)}u_{k^+_b-1} \lessdot u_{k^+_{b}+1}\lessdot \cdots \lessdot u_{\ell(w)}. \]
\end{prop}
\begin{proof}
This is directly analogous to Proposition~\ref{prop:maxchain_toggle1}.
\end{proof}

For $C$ a chain of $[e,w]$ and $f\colon [e,w]\to \mathbb{R}$ a statistic, we write $f(C) \coloneqq  \sum_{u\in C}f(u)$.

Fix some $(i,j) \in \mathrm{Inv}^{-1}(w)$. We will now define an involution $\tau_{(i,j)}$ on the set of maximal chains of $[e,w]$, which will negate the $\mathcal{T}_{(i,j)}$ statistic. Let $C=u_0  \lessdot \cdots \lessdot u_{\ell(w)}$ be a maximal chain of $[e,w]$.  Let $k_1^{-} > k_2^{-} > \cdots > k_a^{-}$ be all the indices for which $\mathcal{T}_{(i,j)}(u_{k})=-1$, and let $k_1^{+} < k_2^{+} < \cdots < k_b^{+}$ be all the indices for which $\mathcal{T}_{(i,j)}(u_{k})=1$. Note that we must have $a,b \geq 1$ because at some point we have to add the inverse inversion~$(i,j)$ to obtain $w$ from $e$. We proceed to define $\tau_{(i,j)}(C)$. First suppose that~$a > b$. Then define $\tau_{(i,j)}(C)$ to be
\[u_0\lessdot \ldots \lessdot u_{k_a^{-}-1} \lessdot s_{(i,j)}u_{k_a^{-}+1}  \lessdot s_{(i,j)}u_{k_a^{-}+2} \lessdot \cdots \lessdot s_{(i,j)}u_{k^{-}_b} \lessdot u_{k^{-}_b}\lessdot \cdots \lessdot u_{\ell(w)}, \]
which by Proposition~\ref{prop:maxchain_toggle1} really is a maximal chain of $[e,w]$. Next suppose that $a < b$. Then define $\tau_{(i,j)}(C)$ to be 
\[u_0\lessdot \ldots \lessdot u_{k^+_a} \lessdot s_{(i,j)}u_{k^+_a} \lessdot  s_{(i,j)}u_{k^+_a+1} \lessdot \cdots \lessdot s_{(i,j)}u_{k^+_b-1} \lessdot u_{k^+_{b}+1}\lessdot \cdots \lessdot u_{\ell(w)}, \]
which by Proposition~\ref{prop:maxchain_toggle2} is a maximal chain of $[e,w]$. Finally, if $a=b$ then we set~$\tau_{(i,j)}(C)\coloneqq  C$. It is easy to see that $\mathcal{T}_{(i,j)}(C) = -\mathcal{T}_{(i,j)}(\tau_{(i,j)}(C))$ because $\mathcal{T}_{(i,j)}(C) = b-a$ and $\mathcal{T}_{(i,j)}(\tau_{(i,j)}(C))=a-b$. It is also clear that $\tau_{(i,j)}$ is an involution.

\begin{example} \label{ex:involution}
Let $w = 53124$. The interval $[e,w]$ is depicted above in Figure~\ref{fig:involution} (which was used also for Example~\ref{ex:disjoint}). Let $C$ be the maximal chain:
\[ 12345 \lessdot 12354 \lessdot  13254 \lessdot 31254 \lessdot 31524 \lessdot 35124 \lessdot 53124. \]
In Figure~\ref{fig:involution} the chain $C$ is drawn in blue. Note that $\mathcal{T}_{(1,3)}(C) = -2$. We can compute that $\tau_{(1,3)}(C)$ is:
\[ 12345 \lessdot 12354 \lessdot  13254 \lessdot 13524 \lessdot 15324 \lessdot 51324 \lessdot 53124.\]
In Figure~\ref{fig:involution} the chain $\tau_{(1,3)}(C)$ is drawn in red. Observe that $\mathcal{T}_{(1,3)}(\tau_{(1,3)}(C)) = 2$.
\end{example}

\begin{lemma} \label{lem:tcde_toggle_sym}
The distribution $\mathrm{maxchain}_{[e,w]}$ on $[e,w]$ is toggle-symmetric.
\end{lemma}
\begin{proof}
Let $g((i,j),\mathbf{x})\in\mathrm{Irr}([e,w])$. It is enough to show that $\sum_{C}\mathcal{T}_{g((i,j),\mathbf{x})}(C)=0$, where the sum is over all maximal chains $C$ of $[e,w]$. To show this sum is zero, we group together the term $\mathcal{T}_{g((i,j),\mathbf{x})}(C)$ and the term $\mathcal{T}_{g((i,j),\mathbf{x})}(\tau_{(i,j)}(C))$. Indeed, we claim that $\mathcal{T}_{g((i,j),\mathbf{x})}(C)=-\mathcal{T}_{g((i,j),\mathbf{x})}(\tau_{(i,j)}(C))$. It is clear from the construction of~$\tau_{(i,j)}(C)$ that there is some $u\in C$ for which $\mathcal{T}_{(i,j)}(u)\neq 0$ and for which $u \in \tau_{(i,j)}(C)$. So thanks to Proposition~\ref{prop:gr_nonzero}, we conclude that either:
\begin{itemize}
\item $\mathcal{T}_{g((i,j),\mathbf{x})}(C)=0$ and $\mathcal{T}_{g((i,j),\mathbf{x})}(\tau_{(i,j)}(C))$=0;
\item or $\mathcal{T}_{g((i,j),\mathbf{x})}(C)=\mathcal{T}_{(i,j)}(C)$ and $\mathcal{T}_{g((i,j),\mathbf{x})}(C)=\mathcal{T}_{(i,j)}(\tau_{(i,j)}(C))$.
\end{itemize}
But as explained above, $\tau_{(i,j)}$ negates $\mathcal{T}_{(i,j)}$: $\mathcal{T}_{(i,j)}(C) = -\mathcal{T}_{(i,j)}(\tau_{(i,j)}(C))$. So either way, we have $\mathcal{T}_{g((i,j),\mathbf{x})}(C)=-\mathcal{T}_{g((i,j),\mathbf{x})}(\tau_{(i,j)}(C))$, as claimed.
\end{proof}

\begin{cor} \label{cor:tcde_implies_cde}
If the initial interval $[e,w]$ of weak order is tCDE, then it is CDE.
\end{cor}

\begin{remark} \label{rem:mcde}
Reiner-Tenner-Yong defined for any poset $P$ and any $m \geq 0$ the distribution $\mathrm{multichain}(m)_P$ in which each $p\in P$ occurs with probability proportional to the number of multichains of length $m$ in which it is contained~\cite[Definition~2.2]{reiner2018poset}. And they called $P$ \emph{mCDE} if $\mathbb{E}(\mathrm{multichain}(m)_P;\mathrm{ddeg})$ is the same for all values of~$m\geq 0$~\cite[Definition 2.3]{reiner2018poset} (see also~\cite[Proposition 1.4]{hopkins2017cde}). For graded $P$, $P$ being mCDE implies it is CDE because for such $P$ we have $\mathrm{uni}_P=\mathrm{multichain}(0)_P$ and $\mathrm{maxchain}_P=\mathrm{lim}_{m\to\infty}\mathrm{multichain}(m)_P$. For a distributive lattice $L=\mathcal{J}(P)$, the multichain distributions are toggle-symmetric (see~\cite[Lemma 2.8]{chan2017expected}~\cite[Proposition 2.5]{reiner2018poset}) and hence $L$ being tCDE implies it is mCDE. But in fact the multichain distributions need not be toggle-symmetric for weak order intervals. Indeed, Reiner-Tenner-Yong observed that for the permutation $w=53124$, the interval~$[e,w]$ is not mCDE~\cite[Remark 2.7]{reiner2018poset} (this interval is depicted in Figure~\ref{fig:involution}). But this $w$ is a vexillary permutation of the balanced straight shape $\lambda=(4,2)$, so it will follow from our main result that~$[e,w]$ is tCDE. We will not consider the mCDE property further in this paper.
\end{remark}

\section{Skew vexillary permutations of balanced shape are tCDE} \label{sec:main}

In this section we prove the main result of this paper which says that the initial weak order intervals corresponding to skew vexillary permutations of balanced shape are tCDE.

\subsection{Balanced shapes} \label{sec:balanced}

In this subsection we review the results and techniques o Chan-Haddadan-Hopkins-Moci~\cite{chan2017expected}, in order to prepare for the next section where we will generalize them to apply to weak order intervals.

Recall the definition of balanced shapes from Section~\ref{sec:young_lat_defs}. Chan-Haddadan-Hopkins-Moci introduced the balanced shapes in order to prove the following theorem about them.

\begin{thm}[{See~\cite[Theorem 3.4]{chan2017expected}}] \label{thm:chhm}
Let $\sigma=\lambda/\nu$ be a balanced shape of height~$a$ and width~$b$. Then $[\nu,\lambda]$ is tCDE with edge density $ab/(a+b)$.
\end{thm}

The main technical tool in the proof of Theorem~\ref{thm:chhm} was the use of certain ``rook'' random variables, which we will now explain. 

\begin{figure}
{\ytableausetup{boxsize=2em}\begin{ytableau}
\scalebox{1.0}[1.1]{$^{1}\,\,_{\text{-}1}$} & \scalebox{1.0}[1.1]{$^{1}\,\,\,\,\,\,$} & &  \\
\scalebox{1.0}[1.1]{$^{1}\,\,_{\text{-}1}$} & \scalebox{1.0}[1.1]{$^{1}\,\,\,\,\,\,$} & & \\
\scalebox{1.0}[1.1]{$^{1}\,\,\,\,\,\,$} & \scalebox{1.0}[1.1]{$^{1}\,\,\,_{1}$} & \scalebox{1.0}[1.1]{$\,\,\,\,\,\,_{1}$} & \scalebox{1.0}[1.1]{$\,\,\,\,_{1}$}  \\
 & \scalebox{1.0}[1.1]{$\,\,\,\,\,\,_{1}$} & \scalebox{1.0}[1.1]{$^{\text{-}1}\,\,_{1}$} & \scalebox{1.0}[1.1]{$^{\text{-}1}\,\,_{1}$}  \\
  & \scalebox{1.0}[1.1]{$\,\,\,\,\,\,_{1}$} & \scalebox{1.0}[1.1]{$^{\text{-}1}\,\,_{1}$} & \scalebox{1.0}[1.1]{$^{\text{-}1}\,\,_{1}$} 
\end{ytableau}\ytableausetup{boxsize=1.5em} }
\caption{The ``rook'' $R_{(3,2)}$ for a $5\times 4$ rectangle.}\label{fig:rook_rect}
\end{figure}

Let $a,b \in \mathbb{N}$. Let $1 \leq i \leq a$ and $1 \leq j \leq b$. The \emph{rook} statistic $R_{(i,j)}\colon [\varnothing,b^a]\to \mathbb{Z}$ on the distributive lattice $\mathcal{J}([a]\times[b])=[\varnothing,b^a]$ is the following linear combination of toggleability statistics:
\begin{equation} \label{eqn:rook_rect}
R_{(i,j)} \coloneqq  \sum_{\substack{1\leq i' \leq i, \\ 1 \leq j' \leq j}} \mathcal{T}^{+}_{(i',j')} + \sum_{\substack{i\leq i' \leq a, \\ j \leq j' \leq b}} \mathcal{T}^{-}_{(i',j')} - \sum_{\substack{1\leq i' < i, \\ 1 \leq j' < j}} \mathcal{T}^{-}_{(i',j')} - \sum_{\substack{i < i' \leq a, \\ j < j' \leq b}} \mathcal{T}^{+}_{(i',j')}.
\end{equation}
For example, Figure~\ref{fig:rook_rect} depicts the rook $R_{(3,2)}$ for $a=5$ and $b=4$. In this picture, the number in the northwest corner of a box $(i',j')$ is the coefficient of $ \mathcal{T}^{+}_{(i',j')}$ in the rook, and the number in the southeast corner is the coefficient of  $\mathcal{T}^{-}_{(i',j')}$. We omit the coefficients of zero.

The following two lemmas explain the significance of the rooks.

\begin{lemma}[{See~\cite[Lemma 3.5]{chan2017expected}}] \label{lem:rook_rect_attack}
For any $a,b\in \mathbb{N}$ and $1 \leq i \leq a$, $1 \leq j \leq b$, and any toggle-symmetric distribution $\mu$ on $\mathcal{J}([a]\times[b])=[\varnothing,b^a]$, we have 
\[\mathbb{E}(\mu; R_{(i,j)})=\sum_{1 \leq i' \leq a} \mathbb{E}(\mu;\mathcal{T}^{-}_{(i',j)}) + \sum_{1 \leq j' \leq b} \mathbb{E}(\mu;\mathcal{T}^{-}_{(i,j')}) .\]
\end{lemma}

\begin{lemma}[{See~\cite[Lemma 3.6]{chan2017expected}}] \label{lem:rook_rect_one}
For any $a,b\in \mathbb{N}$ and $1 \leq i \leq a$, $1 \leq j \leq b$, we have $R_{(i,j)}(\nu)=1$ for all $\nu \in [\varnothing,b^a]$.
\end{lemma}

In particular, note that Lemma~\ref{lem:rook_rect_attack} explains the name ``rook'' for these random variables: they ``attack'' in expectation every box in the same row or same column as the box they are placed on (and they ``doubly attack'' the box they are placed on). Lemma~\ref{lem:rook_rect_attack} is immediate from the definition~\eqref{eqn:rook_rect} of the rooks; but Lemma~\ref{lem:rook_rect_one} is the subtle fact which makes rooks useful.

In fact, Chan-Haddadan-Hopkins-Moci~\cite{chan2017expected} defined rooks not just for the boxes of rectangles, but for the boxes of any skew shape $\sigma=\lambda/\nu$: these rooks are statistics~$R_{(i,j)}\colon [\nu,\lambda]\to \mathbb{Z}$ for $(i,j)$ a box of $\sigma$ defined exactly as in~\eqref{eqn:rook_rect}. Lemma~\ref{lem:rook_rect_attack} continues to hold for these more general rooks; and Lemma~\ref{lem:rook_rect_one} holds assuming that the boxes of~$\sigma$ weakly northwest of the box $(i,j)$ form a rectangle, and that the boxes of~$\sigma$ weakly southeast of $(i,j)$ also form a rectangle. In the next subsection we will refer to this situation as $(i,j)$ being a ``cross-saturated'' box of $\sigma$. What makes balanced shapes special is that they have many cross-saturated boxes: enough to be able to place rooks on the cross-saturated boxes so that all the boxes of the shape are attacked the same number of times.

This same fact (that balanced shapes have many cross-saturated boxes) will also explain the significance of skew vexillary permutations of balanced shape. Indeed, the key to proving our main result will be generalizing the rectangle rooks to apply to permutations, which we do in the next subsection.

\subsection{Rooks for permutations} \label{sec:rooks} 

This is the most involved and technical part of the paper, in which we define rooks for permutations.

Given that Rothe diagrams appear in the definition of skew vexillary permutations, one might expect that we will place rooks on Rothe diagrams. But actually, rather than Rothe diagrams we will work with the inverse inversion sets of permutations. This is because inverse inversion sets are directly related to weak order: recall that~$u \leq w$ if and only if $\mathrm{Inv}^{-1}(u)\subseteq \mathrm{Inv}^{-1}(w)$. Of course, there is not really a big difference between Rothe diagrams and inverse inversion sets: $w\in \mathfrak{S}_n$ is skew vexillary of shape $\sigma$ if and only if some permutation of rows and columns transforms $\mathrm{Inv}^{-1}(w)$ to $\sigma^t$. 

Hence we now take some time to discuss inversion sets in more detail. From now on in this subsection we fix $n\in \mathbb{N}$. Set $\Phi^{+} \coloneqq   \{(i,j)\colon 1\leq i <j \leq n\}$. We view $\Phi^{+}$ as a poset, but not with the partial order induced from $\mathbb{Z}^2$: rather, for $(i,j), (i',j')\in\Phi^{+}$ we have $(i,j) \leq_{\Phi^{+}} (i',j')$ if and only if $i \geq i'$ and $j \leq j'$. (Thus $\Phi^{+}$ is the positive root poset of the root system of Type $A_{n-1}$.) The minimal elements of~$\Phi^{+}$ are $(k,k+1)$ for~$1 \leq k < n$; and $\Phi^{+}$ has a unique maximal element $(1,n)$.

For any $w\in \mathfrak{S}_n$ we have~$\mathrm{Inv}(w)\subseteq\Phi^{+}$, but not all subset of $\Phi^{+}$ arise as inversion sets of permutations. The first thing we want to review is when a subset of $\Phi^{+}$ is an inversion set of a permutation. In fact, there is a well-known classification:

\begin{lemma}[{See, e.g.,~\cite[Lemma 4.1]{dyer2011weak}} for the more general case of Coxeter groups] \label{lem:inv_classification}
Let $S\subseteq \Phi^{+}$. Then there exists $w\in \mathfrak{S}_n$ with $\mathrm{Inv}(w) = S$ if and only if for every $1 \leq a < b < c \leq n$ we have:
\begin{itemize}
\item if $(a,c) \in S$, then $(a,b) \in S$ or $(b,c)\in S$;
\item if $(a,b)\in S$ and $(b,c) \in S$, then $(a,c) \in S$.
\end{itemize}
\end{lemma}

Now before we define the rooks for permutations, let us explain how the set of partitions~$\nu \in [\varnothing,b^a]$ contained in a rectangle~$b^a$ arises when studying inversion sets.

Let $1 \leq k \leq n-1$. Set $\square_k \coloneqq  \{1,2,\ldots,k\}\times \{k+1,k+2,\ldots,n\} \subseteq \Phi^{+}$. We consider~$\square_k$ as a poset with its partial order induced from $\Phi^{+}$. Note that we have a natural isomorphism of posets $\Psi\colon\square_k\simeq  [n-k]\times[k]$ (where $[n-k]\times [k] = k^{n-k}$ has its partial order induced from $\mathbb{Z}^2$). This isomorphism $\Psi$ is just the $90^\circ$ clockwise rotation plus translation which sends $(k,k+1)$ to $(1,1)$ and $(1,n)$ to $(n-k,k)$. Consequently, we also have the natural isomorphism $\Psi\colon\mathcal{J}(\square_k)\simeq[\varnothing,k^{n-k}]$. 

The point of $\square_k $ is the following:

\begin{prop} \label{prop:perm_rect}
Let $w \in \mathfrak{S}_n$ and $1 \leq k < n$. Then:
\begin{itemize}
\item $\mathrm{Inv}^{-1}(w) \setminus \square_k = \mathrm{Inv}^{-1}((u,v))$ for some $(u,v) \in \mathfrak{S}_k \times \mathfrak{S}_{n-k}$;
\item viewing $u$ as a permutation $\pi_{r}$ of the rows of $\square_k$, and $v$ as a permutation $\pi_{c}$ of the columns of $\square_k$, and with $\Pi \coloneqq  \pi_{c} \circ \pi_{r}\colon  \square_k \to \square_k $ and $S \coloneqq  \Pi(\mathrm{Inv}^{-1}(w) \cap \square_k)$, we have that~$\Psi(S)  \in [\varnothing,k^{n-k}]$;
\item for $(i,j) \in \square_k$ with $(i,j) \notin \mathrm{Inv}^{-1}(w)$, we have $\mathrm{Inv}^{-1}(w)\cup\{(i,j)\}=\mathrm{Inv}^{-1}(w')$ for some $w' \in \mathfrak{S}_n$ if and only if $\Psi(S\cup\{\Pi((i,j))\})\in[\varnothing,k^{n-k}]$;
\item for $(i,j) \in \square_k$ with $(i,j) \in \mathrm{Inv}^{-1}(w)$, we have $\mathrm{Inv}^{-1}(w)\setminus\{(i,j)\}=\mathrm{Inv}^{-1}(w')$ for some $w' \in \mathfrak{S}_n$ if and only if $\Psi(S\setminus\{\Pi((i,j))\})\in[\varnothing,k^{n-k}]$;
\end{itemize}
\end{prop}

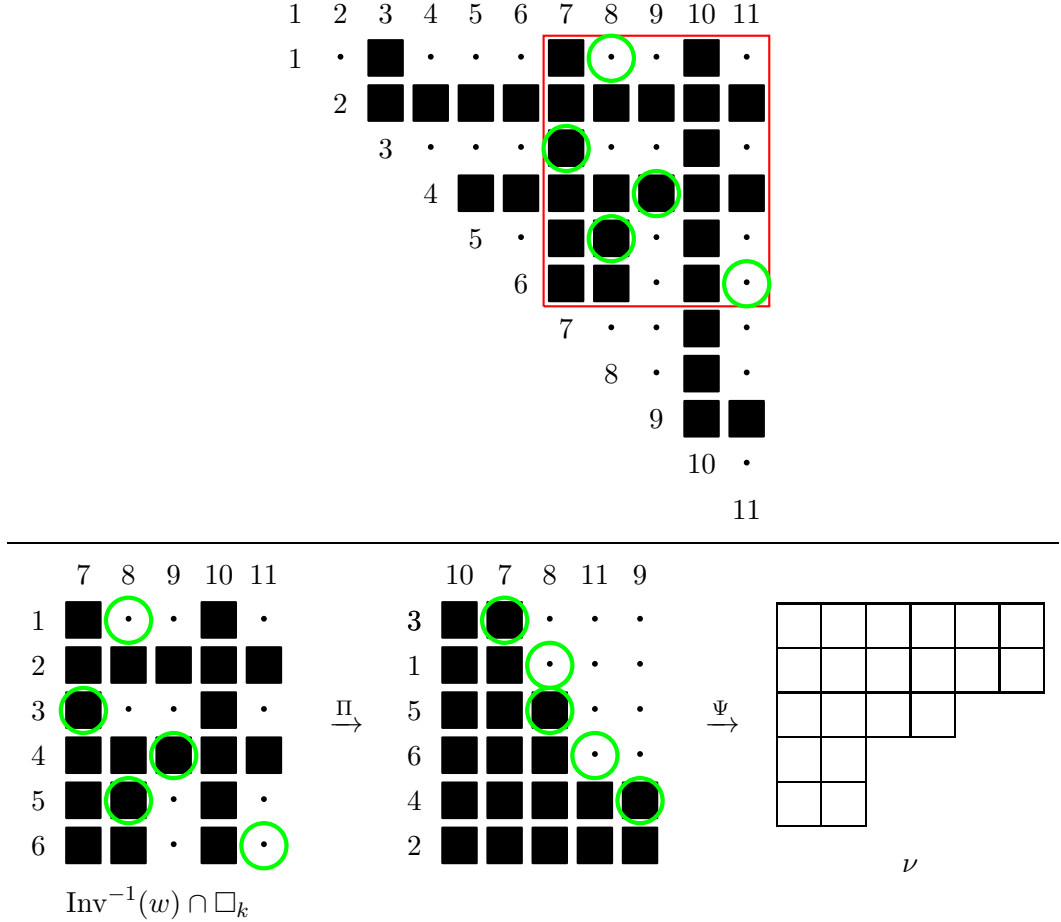
\begin{figure}
\begin{center}
\begin{tikzpicture}[scale=0.6]
\node at (1,0) {1};
\node at (2,0) {2};
\node at (3,0) {3};
\node at (4,0) {4};
\node at (5,0) {5};
\node at (6,0) {6};
\node at (7,0) {7};
\node at (8,0) {8};
\node at (9,0) {9};
\node at (10,0) {10};
\node at (11,0) {11};
\node at (1,-1) {1};
\node at (2,-2) {2};
\node at (3,-3) {3};
\node at (4,-4) {4};
\node at (5,-5) {5};
\node at (6,-6) {6};
\node at (7,-7) {7};
\node at (8,-8) {8};
\node at (9,-9) {9};
\node at (10,-10) {10};
\node at (11,-11) {11};
\draw[red,thick] (6.5,-6.5) -- (6.5,-0.5) -- (11.5,-0.5) -- (11.5,-6.5) -- (6.5,-6.5);
\node at (2,-1) {\huge $\cdot$};
\node at (3,-1) {\huge $\blacksquare$};
\node at (4,-1) {\huge $\cdot$};
\node at (5,-1) {\huge $\cdot$};
\node at (6,-1) {\huge $\cdot$};
\node at (7,-1) {\huge $\blacksquare$};
\node at (8,-1) {\huge $\cdot$};
\node at (9,-1) {\huge $\cdot$};
\node at (10,-1) {\huge $\blacksquare$};
\node at (11,-1) {\huge $\cdot$};
\node at (3,-2) {\huge $\blacksquare$};
\node at (4,-2) {\huge $\blacksquare$};
\node at (5,-2) {\huge $\blacksquare$};
\node at (6,-2) {\huge $\blacksquare$};
\node at (7,-2) {\huge $\blacksquare$};
\node at (8,-2) {\huge $\blacksquare$};
\node at (9,-2) {\huge $\blacksquare$};
\node at (10,-2) {\huge $\blacksquare$};
\node at (11,-2) {\huge $\blacksquare$};
\node at (4,-3) {\huge $\cdot$};
\node at (5,-3) {\huge $\cdot$};
\node at (6,-3) {\huge $\cdot$};
\node at (7,-3) {\huge $\blacksquare$};
\node at (8,-3) {\huge $\cdot$};
\node at (9,-3) {\huge $\cdot$};
\node at (10,-3) {\huge $\blacksquare$};
\node at (11,-3) {\huge $\cdot$};
\node at (5,-4) {\huge $\blacksquare$};
\node at (6,-4) {\huge $\blacksquare$};
\node at (7,-4) {\huge $\blacksquare$};
\node at (8,-4) {\huge $\blacksquare$};
\node at (9,-4) {\huge $\blacksquare$};
\node at (10,-4) {\huge $\blacksquare$};
\node at (11,-4) {\huge $\blacksquare$};
\node at (6,-5) {\huge $\cdot$};
\node at (7,-5) {\huge $\blacksquare$};
\node at (8,-5) {\huge $\blacksquare$};
\node at (9,-5) {\huge $\cdot$};
\node at (10,-5) {\huge $\blacksquare$};
\node at (11,-5) {\huge $\cdot$};
\node at (7,-6) {\huge $\blacksquare$};
\node at (8,-6) {\huge $\blacksquare$};
\node at (9,-6) {\huge $\cdot$};
\node at (10,-6) {\huge $\blacksquare$};
\node at (11,-6) {\huge $\cdot$};
\node at (8,-7) {\huge $\cdot$};
\node at (9,-7) {\huge $\cdot$};
\node at (10,-7) {\huge $\blacksquare$};
\node at (11,-7) {\huge $\cdot$};
\node at (9,-8) {\huge $\cdot$};
\node at (10,-8) {\huge $\blacksquare$};
\node at (11,-8) {\huge $\cdot$};
\node at (10,-9) {\huge $\blacksquare$};
\node at (11,-9) {\huge $\blacksquare$};
\node at (11,-10) {\huge $\cdot$};
\draw[ultra thick, green] (7,-3) circle(0.5cm);
\draw[ultra thick, green] (8,-5) circle(0.5cm);
\draw[ultra thick, green] (9,-4) circle(0.5cm);
\draw[ultra thick, green] (8,-1) circle(0.5cm);
\draw[ultra thick, green] (11,-6) circle(0.5cm);
\end{tikzpicture}
\vspace{-0.25cm} \;
\rule{5.5in}{0.5pt}
\begin{tikzpicture}
\node at (-5,0) {\begin{tikzpicture}[scale=0.6]
\node at (0,-1) {1};
\node at (0,-2) {2};
\node at (0,-3) {3};
\node at (0,-4) {4};
\node at (0,-5) {5};
\node at (0,-6) {6};
\node at (1,0) {7};
\node at (2,0) {8};
\node at (3,0) {9};
\node at (4,0) {10};
\node at (5,0) {11};
\node at (1,-1) {\huge $\blacksquare$};
\node at (2,-1) {\huge $\cdot$};
\node at (3,-1) {\huge $\cdot$};
\node at (4,-1) {\huge $\blacksquare$};
\node at (5,-1) {\huge $\cdot$};
\node at (1,-2) {\huge $\blacksquare$};
\node at (2,-2) {\huge $\blacksquare$};
\node at (3,-2) {\huge $\blacksquare$};
\node at (4,-2) {\huge $\blacksquare$};
\node at (5,-2) {\huge $\blacksquare$};
\node at (1,-3) {\huge $\blacksquare$};
\node at (2,-3) {\huge $\cdot$};
\node at (3,-3) {\huge $\cdot$};
\node at (4,-3) {\huge $\blacksquare$};
\node at (5,-3) {\huge $\cdot$};
\node at (1,-4) {\huge $\blacksquare$};
\node at (2,-4) {\huge $\blacksquare$};
\node at (3,-4) {\huge $\blacksquare$};
\node at (4,-4) {\huge $\blacksquare$};
\node at (5,-4) {\huge $\blacksquare$};
\node at (1,-5) {\huge $\blacksquare$};
\node at (2,-5) {\huge $\blacksquare$};
\node at (3,-5) {\huge $\cdot$};
\node at (4,-5) {\huge $\blacksquare$};
\node at (5,-5) {\huge $\cdot$};
\node at (1,-6) {\huge $\blacksquare$};
\node at (2,-6) {\huge $\blacksquare$};
\node at (3,-6) {\huge $\cdot$};
\node at (4,-6) {\huge $\blacksquare$};
\node at (5,-6) {\huge $\cdot$};
\draw[ultra thick, green] (1,-3) circle(0.5cm);
\draw[ultra thick, green] (2,-5) circle(0.5cm);
\draw[ultra thick, green] (3,-4) circle(0.5cm);
\draw[ultra thick, green] (2,-1) circle(0.5cm);
\draw[ultra thick, green] (5,-6) circle(0.5cm);
\end{tikzpicture}};
\node at (-5,-2.5){$\mathrm{Inv}^{-1}(w)\cap\square_k$};
\node at (-2.5,0) {$\xrightarrow[]{\Pi}$};
\node at (0,0) {\begin{tikzpicture}[scale=0.6]
\node at (0,-1) {3};
\node at (0,-2) {1};
\node at (0,-3) {5};
\node at (0,-4) {6};
\node at (0,-5) {4};
\node at (0,-6) {2};
\node at (0,-1) {3};
\node at (1,0) {10};
\node at (2,0) {7};
\node at (3,0) {8};
\node at (4,0) {11};
\node at (5,0) {9};
\node at (1,-1) {\huge $\blacksquare$};
\node at (2,-1) {\huge $\blacksquare$};
\node at (3,-1) {\huge $\cdot$};
\node at (4,-1) {\huge $\cdot$};
\node at (5,-1) {\huge $\cdot$};
\node at (1,-2) {\huge $\blacksquare$};
\node at (2,-2) {\huge $\blacksquare$};
\node at (3,-2) {\huge $\cdot$};
\node at (4,-2) {\huge $\cdot$};
\node at (5,-2) {\huge $\cdot$};
\node at (1,-3) {\huge $\blacksquare$};
\node at (2,-3) {\huge $\blacksquare$};
\node at (3,-3) {\huge $\blacksquare$};
\node at (4,-3) {\huge $\cdot$};
\node at (5,-3) {\huge $\cdot$};
\node at (1,-4) {\huge $\blacksquare$};
\node at (2,-4) {\huge $\blacksquare$};
\node at (3,-4) {\huge $\blacksquare$};
\node at (4,-4) {\huge $\cdot$};
\node at (5,-4) {\huge $\cdot$};
\node at (1,-5) {\huge $\blacksquare$};
\node at (2,-5) {\huge $\blacksquare$};
\node at (3,-5) {\huge $\blacksquare$};
\node at (4,-5) {\huge $\blacksquare$};
\node at (5,-5) {\huge $\blacksquare$};
\node at (1,-6) {\huge $\blacksquare$};
\node at (2,-6) {\huge $\blacksquare$};
\node at (3,-6) {\huge $\blacksquare$};
\node at (4,-6) {\huge $\blacksquare$};
\node at (5,-6) {\huge $\blacksquare$};
\draw[ultra thick, green] (2,-1) circle(0.5cm);
\draw[ultra thick, green] (3,-2) circle(0.5cm);
\draw[ultra thick, green] (3,-3) circle(0.5cm);
\draw[ultra thick, green] (4,-4) circle(0.5cm);
\draw[ultra thick, green] (5,-5) circle(0.5cm);
\end{tikzpicture}};
\node at (2.5,0) {$\xrightarrow[]{\Psi}$};
\node at (5,0) {\ydiagram{6,6,4,2,2}};
\node at (5,-2) {$\nu$};
\end{tikzpicture}
\end{center}
\caption{Example~\ref{ex:perm_rect} of Proposition~\ref{prop:perm_rect} showing how toggleable boxes of~$\mathrm{Inv}^{-1}(w)\cap\square_k$ correspond to toggleable boxes of~$\nu\subseteq k^{n-k}$.} \label{fig:perm_rect}
\end{figure}

Before proving Proposition~\ref{prop:perm_rect}, let us give an example of this proposition in action.

\begin{example} \label{ex:perm_rect}
Let $n =11$. Let $w= 10\; 7\; 3\; 1\; 8\; 5\; 6\; 11\; 9\; 4\; 2$ (written in one-line notation but with spaces between the letters because $n >9$) and $k=6$. At the top of Figure~\ref{fig:perm_rect} we depict $\mathrm{Inv}^{-1}(w)$: the $(i,j)\in\Phi^{+}$ with $(i,j)\in\mathrm{Inv}^{-1}(w)$ have black boxes drawn on them, and the other $(i,j)\in\Phi^{+}$ have small black dots drawn on them. In this picture we have drawn in red the boundary of $\square_k$. We have also circled in green the boxes of~$\square_k$ which could be added or removed from $\mathrm{Inv}^{-1}(w)$ to remain an inversion set of a permutation. With the language of Proposition~\ref{prop:perm_rect}, we have in this example that~$(u,v) = (3\;1\;5\;6\;4\;2,10\;7\;8\;11\;9)$. At the bottom of Figure~\ref{fig:perm_rect} we show $\Pi=\pi_c\circ\pi_r$ acting on $\mathrm{Inv}^{-1}(w)\cap\square_k$. Observe that $\nu\coloneqq \Psi(\Pi(\mathrm{Inv}^{-1}(w)\cap\square_k) )= (6,6,4,2,2) \in [\varnothing,6^5]$. Also observe that the boxes of~$\square_k$ which can be added to or removed from $\mathrm{Inv}^{-1}(w)$ exactly correspond to the boxes of~$6^5$ which can be added to or removed from~$\nu$.
\end{example}

\begin{proof}[Proof of Proposition~\ref{prop:perm_rect}]
As discussed in Section~\ref{sec:weak_order_def}, Grassmannian permutations with at most one descent at position~$k$ are the unique minimal length coset representatives for the left cosets of $\mathfrak{S}_k \times \mathfrak{S}_{n-k}$ (this follows from the general theory of Coxeter groups; see, e.g.,~\cite[Corollary 2.4.5]{bjorner2005coxeter}). This also implies that the minimal length coset representatives for the {\bf right} cosets of $\mathfrak{S}_k \times \mathfrak{S}_{n-k}$ are the inverse Grassmannian permutations~$w' \in \mathfrak{S}_n$ for which~$(w')^{-1}$ has at most one descent at position~$k$.

Thus for any $w\in\mathfrak{S}_n$ we can write in a unique way $w = (u,v) \cdot w'$, where $w'$ is an inverse Grassmannian permutation such that $(w')^{-1}$ has at most one descent at position~$k$, and $(u,v)\in\mathfrak{S}_k\times\mathfrak{S}_{n-k}$. Note that in the one-line notation of $w'$ the subsequences $1,2,\ldots,k$ and $k+1,k+2,\ldots,n$ appear in increasing order ($w'$ is a ``shuffle'' of these two increasing sequences). It is thus clear that $\mathrm{Inv}^{-1}(w') \subseteq \square_k$. It is also easy to see that $\mathrm{Inv}^{-1}(w) \setminus \square_k = \mathrm{Inv}^{-1}((u,v))$ and that, defining $\Pi\colon\square_k\to\square_k$ as in the statement of the proposition, we have $\Pi(\mathrm{Inv}^{-1}(w) \cap\square_k )=\mathrm{Inv}^{-1}(w')$.

Now let $S\subseteq \square_k$. We claim $S = \mathrm{Inv}^{-1}(w')$ for some $w' \in \mathfrak{S}_n$ an inverse Grassmannian permutation for which $(w')^{-1}$ has at most one descent at position~$k$ if and only if $S$ is an order ideal of $\square_k$ (that is, if and only if $\Psi(S) \in [\varnothing,k^{n-k}]$). To see that this is true one can employ Lemma~\ref{lem:inv_classification}. One can also directly consider the possibilities for the inverse inversion set of a shuffle of two increasing subsequences. 

At any rate, this classification of the inverse inversion sets of inverse Grassmannian permutations finishes the proof of the proposition. It also explains the fact mentioned in Section~\ref{sec:skew_vex_def} that for $w'$ an inverse Grasmmanian permutation we have $[e,w] \simeq [\varnothing,\lambda]$ for some partition~$\lambda$.
\end{proof}

We now proceed to define the rooks for permutations. But before we do that we need to introduce the special class of boxes on which we will allow rooks to be placed. These are the ``cross-saturated'' boxes of $\mathrm{Inv}^{-1}(w)$.

\begin{definition}
Let $D\subseteq \mathbb{Z}^2$ be a diagram. Let $(i,j)\in D$ be a box of~$D$. We say that $(i,j)$ is \emph{cross-saturated in $D$} if whenever $(i,j')\in D$ is a box of $D$ in the same row as~$(i,j)$ and $(i',j)\in D$ is a box of $D$ in the same column as~$(i,j)$, we have~$(i',j')\in D$. If $(i,j)$ is cross-saturated in $D$, then we define its \emph{cross-saturation in $D$} to be the set of all boxes $(i',j')$, where $(i,j')\in D$ is a box of $D$ in the same row as~$(i,j)$ and $(i',j)\in D$ is a box of $D$ in the same column as~$(i,j)$.
\end{definition}

From now in this subsection we fix~$w\in \mathfrak{S}_n$, and fix~$(i,j)$ to be a cross-saturated box of $\mathrm{Inv}^{-1}(w)$, with $\mathcal{C}\subseteq\mathrm{Inv}^{-1}(w)$ its cross-saturation. 

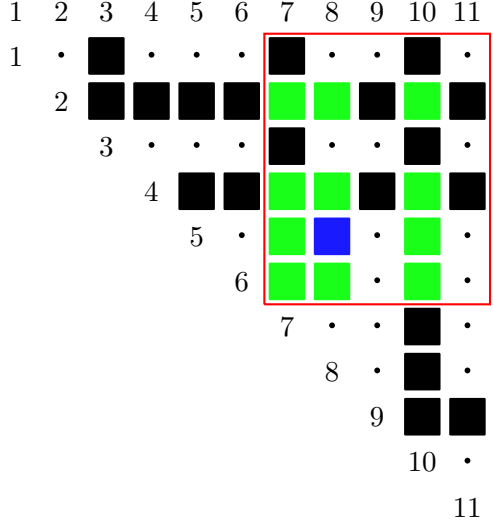
\begin{figure}
\begin{center}
\begin{tikzpicture}[scale=0.6]
\node at (1,0) {1};
\node at (2,0) {2};
\node at (3,0) {3};
\node at (4,0) {4};
\node at (5,0) {5};
\node at (6,0) {6};
\node at (7,0) {7};
\node at (8,0) {8};
\node at (9,0) {9};
\node at (10,0) {10};
\node at (11,0) {11};
\node at (1,-1) {1};
\node at (2,-2) {2};
\node at (3,-3) {3};
\node at (4,-4) {4};
\node at (5,-5) {5};
\node at (6,-6) {6};
\node at (7,-7) {7};
\node at (8,-8) {8};
\node at (9,-9) {9};
\node at (10,-10) {10};
\node at (11,-11) {11};
\draw[red,thick] (6.5,-6.5) -- (6.5,-0.5) -- (11.5,-0.5) -- (11.5,-6.5) -- (6.5,-6.5);
\node at (2,-1) {\huge $\cdot$};
\node at (3,-1) {\huge $\blacksquare$};
\node at (4,-1) {\huge $\cdot$};
\node at (5,-1) {\huge $\cdot$};
\node at (6,-1) {\huge $\cdot$};
\node at (7,-1) {\huge $\blacksquare$};
\node at (8,-1) {\huge $\cdot$};
\node at (9,-1) {\huge $\cdot$};
\node at (10,-1) {\huge $\blacksquare$};
\node at (11,-1) {\huge $\cdot$};
\node at (3,-2) {\huge $\blacksquare$};
\node at (4,-2) {\huge $\blacksquare$};
\node at (5,-2) {\huge $\blacksquare$};
\node at (6,-2) {\huge $\blacksquare$};
\node at (7,-2) {\huge \color{green!90} $\blacksquare$};
\node at (8,-2) {\huge \color{green!90} $\blacksquare$};
\node at (9,-2) {\huge $\blacksquare$};
\node at (10,-2) {\huge \color{green!90} $\blacksquare$};
\node at (11,-2) {\huge $\blacksquare$};
\node at (4,-3) {\huge $\cdot$};
\node at (5,-3) {\huge $\cdot$};
\node at (6,-3) {\huge $\cdot$};
\node at (7,-3) {\huge $\blacksquare$};
\node at (8,-3) {\huge $\cdot$};
\node at (9,-3) {\huge $\cdot$};
\node at (10,-3) {\huge $\blacksquare$};
\node at (11,-3) {\huge $\cdot$};
\node at (5,-4) {\huge $\blacksquare$};
\node at (6,-4) {\huge $\blacksquare$};
\node at (7,-4) {\huge \color{green!90} $\blacksquare$};
\node at (8,-4) {\huge \color{green!90} $\blacksquare$};
\node at (9,-4) {\huge $\blacksquare$};
\node at (10,-4) {\huge \color{green!90} $\blacksquare$};
\node at (11,-4) {\huge $\blacksquare$};
\node at (6,-5) {\huge $\cdot$};
\node at (7,-5) {\huge \color{green!90} $\blacksquare$};
\node at (8,-5) {\huge \color{blue!90} $\blacksquare$};
\node at (9,-5) {\huge $\cdot$};
\node at (10,-5) {\huge \color{green!90} $\blacksquare$};
\node at (11,-5) {\huge $\cdot$};
\node at (7,-6) {\huge \color{green!90} $\blacksquare$};
\node at (8,-6) {\huge \color{green!90} $\blacksquare$};
\node at (9,-6) {\huge $\cdot$};
\node at (10,-6) {\huge \color{green!90} $\blacksquare$};
\node at (11,-6) {\huge $\cdot$};
\node at (8,-7) {\huge $\cdot$};
\node at (9,-7) {\huge $\cdot$};
\node at (10,-7) {\huge $\blacksquare$};
\node at (11,-7) {\huge $\cdot$};
\node at (9,-8) {\huge $\cdot$};
\node at (10,-8) {\huge $\blacksquare$};
\node at (11,-8) {\huge $\cdot$};
\node at (10,-9) {\huge $\blacksquare$};
\node at (11,-9) {\huge $\blacksquare$};
\node at (11,-10) {\huge $\cdot$};
\end{tikzpicture} 
\end{center}
\caption{Example~\ref{ex:cross_sat_rect} of a cross-saturated box of an inverse inversion set.} \label{fig:cross_sat_rect}
\end{figure}

\begin{example} \label{ex:cross_sat_rect}
Let $w= 10\, 7\, 3\, 1\, 8\, 5\, 6\, 11\, 9\, 4\, 2$. Let $(i,j) = (5,8)\in\mathrm{Inv}^{-1}(w)$. Figure~\ref{fig:cross_sat_rect} depicts $\mathrm{Inv}^{-1}(w)$, as in Example~\ref{ex:perm_rect}. This figure shows that $(i,j)$, which is colored blue, is a cross-saturated box of $\mathrm{Inv}^{-1}(w)$. Its cross-saturation $\mathcal{C}$ consists of~$(i,j)$ together with all boxes drawn in green in this figure. The unique $1\leq k < n$ with $(k,k+1) \in \mathcal{C}$ is $k=6$. The border of~$\square_k$ is drawn in red in Figure~\ref{fig:cross_sat_rect}. Observe that~$\mathcal{C}\subseteq \square_k$ and that~$\{i,\ldots,k\}\times\{k+1,\ldots,j\} = \{5,6\}\times\{7,8\} \subseteq \mathcal{C}$.
\end{example}

\begin{prop} \label{prop:cross_sat_rect}
There is a unique $1\leq k < n$ with $(k,k+1) \in \mathcal{C}$. Moreover, we have that $\mathcal{C}\subseteq \square_k$ and that $\{i,i+1,\ldots,k\}\times \{k+1,k+2,\ldots,j\} \subseteq \mathcal{C}$.
\end{prop}

\begin{proof}
Lemma~\ref{lem:inv_classification} implies that for any $i < x < j$, we either have~$(i,x)\in\mathrm{Inv}^{-1}(w)$ or~$(x,j) \in \mathrm{Inv}^{-1}(w)$. We claim that there exists some~$i \leq k \leq j$ such that for any $i < x < j$ we have  $(x,j) \in \mathrm{Inv}^{-1}(w)$ if and only if $x \leq k$, and $(i,x) \in \mathrm{Inv}^{-1}(w)$ if and only if $k < x$. 

To see this, first note that we cannot have for any $i < x < j$ that $(i,x)\in\mathrm{Inv}^{-1}(w)$ and that~$(x,j) \in \mathrm{Inv}^{-1}(w)$, because then the cross-saturation condition would require that~$(x,x) \in\mathrm{Inv}^{-1}(w)$ but $\mathrm{Inv}^{-1}(w)\subseteq\Phi^{+}$ and $(x,x)\notin \Phi^{+}$. Next, observe that if~$(i,x) \notin \mathrm{Inv}^{-1}(w)$ for some $i < x < j$ then $(i,x') \notin \mathrm{Inv}^{-1}(w)$ for any $x' < x$: indeed, $(i,x) \notin \mathrm{Inv}^{-1}(w)$ implies $(x,j)\in \mathrm{Inv}^{-1}(w)$, and so if $(i,x') \in \mathrm{Inv}^{-1}(w)$ for some $x' < x$, then the cross-saturation condition would imply that $(x,x')\in\mathrm{Inv}^{-1}(w)$, which again is a contradiction because $(x,x')\notin\Phi^{+}$. Similarly, if $(x,j) \notin \mathrm{Inv}^{-1}(w)$ for some $i < x < j$ then $(x',j) \notin \mathrm{Inv}^{-1}(w)$ for any $x' > x$. Together, these observations do imply that there is some~$i \leq k \leq j$ such that for any $i < x < j$ we have  $(x,j) \in \mathrm{Inv}^{-1}(w)$ if and only if $x \leq k$, and $(i,x) \in \mathrm{Inv}^{-1}(w)$ if and only if $k < x$.

It is then easy to see that the unique $1\leq k < n$ with $(k,k+1)\in\mathcal{C}$ is the $k$ from the previous paragraph. Similarly, it is easy to see from our definition of $k$ that $\mathcal{C}\subseteq \square_k$ and that~$\{i,i+1,\ldots,k\}\times \{k+1,k+2,\ldots,j\} \subseteq \mathcal{C}$, as required.
\end{proof}

Now we can give the formula defining permutation rooks. The \emph{rook $\widehat{R}_{(i,j)}\colon[e,w]\to \mathbb{Z}$} is the following linear combination of toggleability statistics:
\begin{align} \label{eqn:rook_perm}
\widehat{R}_{(i,j)}  \hspace{-0.05cm} &\coloneqq  \hspace{-0.3cm} \sum_{\substack{i'>i, \, j'<j, \\ (i',j')\in \mathcal{C}}} \mathcal{T}_{(i',j')} + \sum_{\substack{j'<j, \\ (i,j')\in \mathcal{C}}} \mathcal{T}^{+}_{(i,j')} + \sum_{\substack{i'>i, \\ (i',j)\in \mathcal{C}}} \mathcal{T}^{+}_{(i',j)}+\left(\mathcal{T}^{+}_{(i,j)}+\mathcal{T}^{-}_{(i,j)}\right) \\ \nonumber
& + \sum_{\substack{i'<i, \, j'<j, \\ (i',j')\in \mathcal{C}}} \sum_{\substack{g((i',j'),\mathbf{x}) \\ \in\mathrm{Irr}([e,w]),\\ i \in \mathbf{x}}} \mathcal{T}_{g((i',j'),\mathbf{x})} +  \sum_{\substack{i'>i, \, j'>j, \\ (i',j')\in \mathcal{C}}} \sum_{\substack{g((i',j'),\mathbf{x}) \\ \in\mathrm{Irr}([e,w]),\\ j \notin \mathbf{x}}} \mathcal{T}_{g((i',j'),\mathbf{x})} \\ \nonumber
&+  \hspace{-0.35cm} \sum_{\substack{j'>j, \\ (i,j')\in \mathcal{C}}} \hspace{-0.1cm} \left(  \hspace{-0.15cm} \sum_{\substack{g((i,j'),\mathbf{x}) \\ \in\mathrm{Irr}([e,w]),\\ j \in \mathbf{x}}} \hspace{-0.4cm} \mathcal{T}^{-}_{g((i,j'),\mathbf{x})} + \hspace{-0.45cm} \sum_{\substack{g((i,j'),\mathbf{x}) \\ \in\mathrm{Irr}([e,w]),\\ j \notin \mathbf{x}}} \hspace{-0.4cm} \mathcal{T}^{+}_{g((i,j'),\mathbf{x})} \hspace{-0.15cm} \right) \hspace{-0.1cm} + \hspace{-0.35cm} \sum_{\substack{i'<i, \\ (i',j)\in \mathcal{C}}} \hspace{-0.1cm} \left(  \hspace{-0.15cm} \sum_{\substack{g((i',j),\mathbf{x}) \\ \in\mathrm{Irr}([e,w]),\\ i \notin \mathbf{x}}} \hspace{-0.4cm} \mathcal{T}^{-}_{g((i',j),\mathbf{x})} + \hspace{-0.45cm} \sum_{\substack{g((i',j),\mathbf{x}) \\ \in\mathrm{Irr}([e,w]),\\ i \in \mathbf{x}}} \hspace{-0.4cm} \mathcal{T}^{+}_{g((i',j),\mathbf{x})} \hspace{-0.15cm} \right) \\ \nonumber
&+ \sum_{\substack{i'<i, \, j'>j, \\ (i',j')\in \mathcal{C}}} \left( \sum_{\substack{g((i',j'),\mathbf{x}) \\ \in\mathrm{Irr}([e,w]),\\ i \in \mathbf{x}, j \notin \mathbf{x}}} \mathcal{T}_{g((i',j'),\mathbf{x})}  -\sum_{\substack{g((i',j'),\mathbf{x}) \\ \in\mathrm{Irr}([e,w]),\\ i \notin \mathbf{x}, j \in \mathbf{x}}} \mathcal{T}_{g((i',j'),\mathbf{x})} \right).
\end{align}
(We use a hat on top of the $R$ just to differentiate it from the rectangle rooks.) 

This definition might look like a bit of a mess. One thing to notice right away about the definition is that for each $(i',j')\in \mathcal{C}$ and $g((i',j'),\mathbf{x}) \in \mathrm{Irr}([e,w])$, we have in $\widehat{R}_{(i,j)}$ a term of~$\mathcal{T}^{+}_{g((i',j'),\mathbf{x})}$ and a term of~$\mathcal{T}^{-}_{g((i',j'),\mathbf{x})}$, each with coefficient $-1$, $0$, or $+1$, and these are all the terms of $\widehat{R}_{(i,j)}$. We give a schematic of what these coefficients look like in Figure~\ref{fig:rook_rect_perm}. In this figure we have shifted the coordinated system by applying the map $\Psi$ in order to make it more easily comparable to Figure~\ref{fig:rook_rect}.  In this picture, the number in the northwest corner of a box $\Psi((i',j'))$ is the coefficient of $\mathcal{T}^{+}_{g((i',j'),\mathbf{x})}$ in the rook, and the number in the southeast corner is the coefficient of $\mathcal{T}^{-}_{g((i',j'),\mathbf{x})}$. We omit the coefficients of zero (and we assuming that $(i',j')\in \mathcal{C}$ for the relevant $(i',j')$). The boxes with ``$1/1$'' indicate that we either have a term of $\mathcal{T}^{+}_{g((i',j'),\mathbf{x})}$ or $\mathcal{T}^{-}_{g((i',j'),\mathbf{x})}$ with coefficient $1$, but which one depends on the particular $\mathbf{x}$. The boxes with ``$\pm1$'' in the northwest corner and ``$\mp1$'' in the southeast corner indicate that we have terms of~$\mathcal{T}^{+}_{g((i',j'),\mathbf{x})}$ and $\mathcal{T}^{-}_{g((i',j'),\mathbf{x})}$ with opposing signs, but whether $\mathcal{T}^{+}_{g((i',j'),\mathbf{x})}$ has a positive or negative coefficient depends on the particular~$\mathbf{x}$.

\begin{figure}
{\ytableausetup{boxsize=2em}\begin{ytableau}
\scalebox{1.0}[1.1]{$^{1}\,\,_{\text{-}1}$} & \scalebox{1.0}[1.1]{$^{1}\,\,\,\,\,\,$} & &  \\
\scalebox{1.0}[1.1]{$^{1}\,\,_{\text{-}1}$} & \scalebox{1.0}[1.1]{$^{1}\,\,\,\,\,\,$} & & \\
\scalebox{1.0}[1.1]{$^{1}\,\,\,\,\,\,$} & \scalebox{1.0}[1.1]{$^{1}\,\,\,_{1}$} & \scalebox{1.0}[1.1]{$^{1} / \,_{1}$} & \scalebox{1.0}[1.1]{$^{1} / \,_{1}$}  \\
 & \scalebox{1.0}[1.1]{$^{1} / \,_{1}$} & \scalebox{0.82}[1.2]{$^{\pm 1}\,\,_{\mp 1}$} & \scalebox{0.82}[1.2]{$^{\pm 1}\,\,_{\mp 1}$}  \\
  & \scalebox{1.0}[1.1]{$^{1} / \,_{1}$} & \scalebox{0.82}[1.2]{$^{\pm 1}\,\,_{\mp 1}$} & \scalebox{0.82}[1.2]{$^{\pm 1}\,\,_{\mp1}$} 
\end{ytableau}\ytableausetup{boxsize=1.5em} }
\caption{A schematic for the coefficients defining the permutation rook $\widehat{R}_{(i,j)}$, where $\Psi((i,j))=(3,2)$. Compare to Figure~\ref{fig:rook_rect}.}\label{fig:rook_rect_perm}
\end{figure}

We want the rooks for permutations defined by~\eqref{eqn:rook_perm} to ``behave like'' the rooks for rectangles defined by~\eqref{eqn:rook_rect}. The previous propositions suggest this might be possible: Proposition~\ref{prop:cross_sat_rect} says that our cross-saturation $\mathcal{C}$ is contained in a rectangle~$\square_k$; and Proposition~\ref{prop:perm_rect} says that for any $w' \in [e,w]$ the toggleable boxes of~$\mathrm{Inv}^{-1}(w')\cap\square_k$ correspond to toggleable boxes of a partition~$\nu\in[\varnothing,k^{n-k}]$. Indeed, using these propositions we will show that any evaluation of one of the permutation rooks is equal to an evaluation of a rectangle rook, and hence by Lemma~\ref{lem:rook_rect_one} is always equal to~$1$. This explains, to some degree, the complicated form of the permutation rook $\widehat{R}_{(i,j)}$ above: we want it to be that after applying an appropriate row-and-column permutation $\Pi$ and the map $\Psi$, the toggleability statistics  appearing in $\widehat{R}_{(i,j)}$ (for the toggleable boxes) correspond to those of a rectangle rook. We will see this in action in Example~\ref{ex:rook_perm_one} below.

Our goal is now to show that the two fundamental lemmas about the rectangle rooks (Lemmas~\ref{lem:rook_rect_attack} and~\ref{lem:rook_rect_one}) continue to hold for the permutation rooks. The first of these, which says that for a toggle-symmetric distribution the rook random variable attacks in expectation each box in its row and column, again follows essentially from the definition of $\widehat{R}_{(i,j)}$.

\begin{lemma} \label{lem:rook_perm_attack}
For any toggle-symmetric distribution $\mu$ on $[e,w]$, we have 
\[\mathbb{E}(\mu;\widehat{R}_{(i,j)})=\sum_{(i',j)\in\mathrm{Inv}^{-1}(w)} \mathbb{E}(\mu;\mathcal{T}^{-}_{(i',j)}) +\sum_{(i,j')\in\mathrm{Inv}^{-1}(w)} \mathbb{E}(\mu;\mathcal{T}^{-}_{(i,j')}) .\]
\end{lemma}
\begin{proof}
For any $(i',j)\in\mathrm{Inv}^{-1}(w)$ we certainly have $(i',j)\in\mathcal{C}$, and similarly for any $(i,j')\in\mathrm{Inv}^{-1}(w)$. Hence, the claimed equality follows immediately from the definition~\eqref{eqn:rook_perm} of the permutation rook $\widehat{R}_{(i,j)}$, together with the fact that for a toggle-symmetric distribution $\mu$ on $[e,w]$ we have $\mathbb{E}(\mu;\mathcal{T}^{+}_{g((i',j'),\mathbf{x})})=\mathbb{E}(\mu;\mathcal{T}^{-}_{g((i',j'),\mathbf{x})})$ and $\mathbb{E}(\mu;\mathcal{T}_{g((i',j'),\mathbf{x})})=0$ for any $g((i',j'),\mathbf{x})\in\mathrm{Irr}([e,w])$.
\end{proof}

The next fundamental lemma, which says that $\widehat{R}_{(i,j)}$ is always equal to~$1$, is more subtle.

\begin{lemma} \label{lem:rook_perm_one}
We have~$\widehat{R}_{(i,j)}(w')=1$ for all $w' \in [e,w]$.
\end{lemma}

Before proving Lemma~\ref{lem:rook_perm_one}, let's show some examples.

\begin{figure}
\begin{center}
\begin{tikzpicture}
\node at (-10.25,0){\begin{tikzpicture}[scale=0.6]
\node at (1.1,0) {1};
\node at (2,0) {2};
\node at (3,0) {3};
\node at (4,0) {4};
\node at (5,0) {5};
\node at (6,0) {6};
\node at (7,0) {7};
\node at (8,0) {8};
\node at (9,0) {9};
\node at (10,0) {10};
\node at (11,0) {11};
\node at (1.1,-1) {1};
\node at (2,-2) {2};
\node at (3,-3) {3};
\node at (4,-4) {4};
\node at (5,-5) {5};
\node at (6,-6) {6};
\node at (7,-7) {7};
\node at (8,-8) {8};
\node at (9,-9) {9};
\node at (10,-10) {10};
\node at (11,-11) {11};
\draw[red,thick] (6.5,-6.5) -- (6.5,-0.5) -- (11.5,-0.5) -- (11.5,-6.5) -- (6.5,-6.5);
\node at (2,-1) {\huge $\cdot$};
\node at (3,-1) {\huge $\blacksquare$};
\node at (4,-1) {\huge $\cdot$};
\node at (5,-1) {\huge $\cdot$};
\node at (6,-1) {\huge $\cdot$};
\node at (7,-1) {\huge $\square$};
\node at (8,-1) {\huge $\cdot$};
\node at (9,-1) {\huge $\cdot$};
\node at (10,-1) {\huge $\square$};
\node at (11,-1) {\huge $\cdot$};
\node at (3,-2) {\huge $\blacksquare$};
\node at (4,-2) {\huge $\square$};
\node at (5,-2) {\huge $\blacksquare$};
\node at (6,-2) {\huge $\square$};
\node at (7,-2) {\huge \color{green!90} $\square$};
\node at (8,-2) {\huge \color{green!90} $\square$};
\node at (9,-2) {\huge $\square$};
\node at (10,-2) {\huge \color{green!90} $\blacksquare$};
\node at (11,-2) {\huge $\square$};
\node at (4,-3) {\huge $\cdot$};
\node at (5,-3) {\huge $\cdot$};
\node at (6,-3) {\huge $\cdot$};
\node at (7,-3) {\huge $\square$};
\node at (8,-3) {\huge $\cdot$};
\node at (9,-3) {\huge $\cdot$};
\node at (10,-3) {\huge $\square$};
\node at (11,-3) {\huge $\cdot$};
\node at (5,-4) {\huge $\blacksquare$};
\node at (6,-4) {\huge $\blacksquare$};
\node at (7,-4) {\huge \color{green!90} $\blacksquare$};
\node at (8,-4) {\huge \color{green!90} $\blacksquare$};
\node at (9,-4) {\huge $\blacksquare$};
\node at (10,-4) {\huge \color{green!90} $\blacksquare$};
\node at (11,-4) {\huge $\square$};
\node at (6,-5) {\huge $\cdot$};
\node at (7,-5) {\huge \color{green!90} $\square$};
\node at (8,-5) {\huge \color{blue!90} $\square$};
\node at (9,-5) {\huge $\cdot$};
\node at (10,-5) {\huge \color{green!90} $\blacksquare$};
\node at (11,-5) {\huge $\cdot$};
\node at (7,-6) {\huge \color{green!90} $\blacksquare$};
\node at (8,-6) {\huge \color{green!90} $\square$};
\node at (9,-6) {\huge $\cdot$};
\node at (10,-6) {\huge \color{green!90} $\blacksquare$};
\node at (11,-6) {\huge $\cdot$};
\node at (8,-7) {\huge $\cdot$};
\node at (9,-7) {\huge $\cdot$};
\node at (10,-7) {\huge $\blacksquare$};
\node at (11,-7) {\huge $\cdot$};
\node at (9,-8) {\huge $\cdot$};
\node at (10,-8) {\huge $\blacksquare$};
\node at (11,-8) {\huge $\cdot$};
\node at (10,-9) {\huge $\blacksquare$};
\node at (11,-9) {\huge $\square$};
\node at (11,-10) {\huge $\cdot$};
\end{tikzpicture} };
\draw (-6.8,4) -- (-6.8,-4);
\node at (-5,0) {\begin{tikzpicture}[scale=0.6]
\node at (0,-1) {1};
\node at (0,-2) {2};
\node at (0,-3) {3};
\node at (0,-4) {4};
\node at (0,-5) {5};
\node at (0,-6) {6};
\node at (1,0) {7};
\node at (2,0) {8};
\node at (3,0) {9};
\node at (4,0) {10};
\node at (5,0) {11};
\node at (1,-1) {\huge $\square$};
\node at (2,-1) {\huge $\cdot$};
\node at (3,-1) {\huge $\cdot$};
\node at (4,-1) {\huge $\square$};
\node at (5,-1) {\huge $\cdot$};
\node at (1,-2) {\huge \color{green!90} $\square$};
\node at (2,-2) {\huge \color{green!90} $\square$};
\node at (3,-2) {\huge $\square$};
\node at (4,-2) {\huge \color{green!90} $\blacksquare$};
\node at (5,-2) {\huge $\square$};
\node at (1,-3) {\huge $\square$};
\node at (2,-3) {\huge $\cdot$};
\node at (3,-3) {\huge $\cdot$};
\node at (4,-3) {\huge $\square$};
\node at (5,-3) {\huge $\cdot$};
\node at (1,-4) {\huge \color{green!90} $\blacksquare$};
\node at (2,-4) {\huge \color{green!90} $\square$};
\node at (3,-4) {\huge $\blacksquare$};
\node at (4,-4) {\huge \color{green!90} $\blacksquare$};
\node at (5,-4) {\huge $\square$};
\node at (1,-5) {\huge \color{green!90} $\square$};
\node at (2,-5) {\huge \color{blue!90} $\square$};
\node at (3,-5) {\huge $\cdot$};
\node at (4,-5) {\huge \color{green!90} $\blacksquare$};
\node at (5,-5) {\huge $\cdot$};
\node at (1,-6) {\huge \color{green!90} $\blacksquare$};
\node at (2,-6) {\huge \color{green!90} $\blacksquare$};
\node at (3,-6) {\huge $\cdot$};
\node at (4,-6) {\huge \color{green!90} $\blacksquare$};
\node at (5,-6) {\huge $\cdot$};
\end{tikzpicture}};
\node at (-5,-2.5){$\mathrm{Inv}^{-1}(w')\cap\square_k$};
\node at (-3.1,0) {$\xrightarrow[]{\Pi}$};
\node at (-1.1,0) {\begin{tikzpicture}[scale=0.6]
\node at (0,-1) {3};
\node at (0,-2) {1};
\node at (0,-3) {5};
\node at (0,-4) {2};
\node at (0,-5) {6};
\node at (0,-6) {4};
\node at (1,0) {10};
\node at (2,0) {7};
\node at (3,0) {8};
\node at (4,0) {9};
\node at (5,0) {11};
\node at (1,-1) {\huge $\square$};
\node at (2,-1) {\huge $\square$};
\node at (3,-1) {\huge $\cdot$};
\node at (4,-1) {\huge $\cdot$};
\node at (5,-1) {\huge $\cdot$};
\node at (1,-2) {\huge $\square$};
\node at (2,-2) {\huge $\square$};
\node at (3,-2) {\huge $\cdot$};
\node at (4,-2) {\huge $\cdot$};
\node at (5,-2) {\huge $\cdot$};
\node at (1,-3) {\huge \color{green!90} $\blacksquare$};
\node at (2,-3) {\huge \color{green!90} $\square$};
\node at (3,-3) {\huge \color{blue!90} $\square$};
\node at (4,-3) {\huge $\cdot$};
\node at (5,-3) {\huge $\cdot$};
\node at (1,-4) {\huge \color{green!90} $\blacksquare$};
\node at (2,-4) {\huge \color{green!90} $\square$};
\node at (3,-4) {\huge \color{green!90} $\square$};
\node at (4,-4) {\huge $\square$};
\node at (5,-4) {\huge $\square$};
\node at (1,-5) {\huge \color{green!90} $\blacksquare$};
\node at (2,-5) {\huge \color{green!90} $\blacksquare$};
\node at (3,-5) {\huge \color{green!90} $\square$};
\node at (4,-5) {\huge $\cdot$};
\node at (5,-5) {\huge $\cdot$};
\node at (1,-6) {\huge \color{green!90} $\blacksquare$};
\node at (2,-6) {\huge \color{green!90} $\blacksquare$};
\node at (3,-6) {\huge \color{green!90} $\blacksquare$};
\node at (4,-6) {\huge $\blacksquare$};
\node at (5,-6) {\huge $\square$};
\end{tikzpicture}};
\end{tikzpicture}
\end{center}
\caption{Example~\ref{ex:rook_perm_one} showing an instance of $\widehat{R}_{(i,j)}(w')=1$. } \label{fig:rook_perm_one_ex1}
\end{figure}

\begin{figure}
\begin{center}
\begin{tikzpicture}
\node at (-10.25,0){\begin{tikzpicture}[scale=0.6]
\node at (1.1,0) {1};
\node at (2,0) {2};
\node at (3,0) {3};
\node at (4,0) {4};
\node at (5,0) {5};
\node at (6,0) {6};
\node at (7,0) {7};
\node at (8,0) {8};
\node at (9,0) {9};
\node at (10,0) {10};
\node at (11,0) {11};
\node at (1.1,-1) {1};
\node at (2,-2) {2};
\node at (3,-3) {3};
\node at (4,-4) {4};
\node at (5,-5) {5};
\node at (6,-6) {6};
\node at (7,-7) {7};
\node at (8,-8) {8};
\node at (9,-9) {9};
\node at (10,-10) {10};
\node at (11,-11) {11};
\draw[red,thick] (6.5,-6.5) -- (6.5,-0.5) -- (11.5,-0.5) -- (11.5,-6.5) -- (6.5,-6.5);
\node at (2,-1) {\huge $\cdot$};
\node at (3,-1) {\huge $\square$};
\node at (4,-1) {\huge $\cdot$};
\node at (5,-1) {\huge $\cdot$};
\node at (6,-1) {\huge $\cdot$};
\node at (7,-1) {\huge $\square$};
\node at (8,-1) {\huge $\cdot$};
\node at (9,-1) {\huge $\cdot$};
\node at (10,-1) {\huge $\square$};
\node at (11,-1) {\huge $\cdot$};
\node at (3,-2) {\huge $\blacksquare$};
\node at (4,-2) {\huge $\square$};
\node at (5,-2) {\huge $\square$};
\node at (6,-2) {\huge $\square$};
\node at (7,-2) {\huge \color{green!90} $\blacksquare$};
\node at (8,-2) {\huge \color{green!90} $\blacksquare$};
\node at (9,-2) {\huge $\square$};
\node at (10,-2) {\huge \color{green!90} $\square$};
\node at (11,-2) {\huge $\square$};
\node at (4,-3) {\huge $\cdot$};
\node at (5,-3) {\huge $\cdot$};
\node at (6,-3) {\huge $\cdot$};
\node at (7,-3) {\huge $\blacksquare$};
\node at (8,-3) {\huge $\cdot$};
\node at (9,-3) {\huge $\cdot$};
\node at (10,-3) {\huge $\square$};
\node at (11,-3) {\huge $\cdot$};
\node at (5,-4) {\huge $\blacksquare$};
\node at (6,-4) {\huge $\square$};
\node at (7,-4) {\huge \color{green!90} $\blacksquare$};
\node at (8,-4) {\huge \color{green!90} $\blacksquare$};
\node at (9,-4) {\huge $\square$};
\node at (10,-4) {\huge \color{green!90} $\blacksquare$};
\node at (11,-4) {\huge $\square$};
\node at (6,-5) {\huge $\cdot$};
\node at (7,-5) {\huge \color{green!90} $\blacksquare$};
\node at (8,-5) {\huge \color{blue!90} $\blacksquare$};
\node at (9,-5) {\huge $\cdot$};
\node at (10,-5) {\huge \color{green!90} $\blacksquare$};
\node at (11,-5) {\huge $\cdot$};
\node at (7,-6) {\huge \color{green!90} $\blacksquare$};
\node at (8,-6) {\huge \color{green!90} $\blacksquare$};
\node at (9,-6) {\huge $\cdot$};
\node at (10,-6) {\huge \color{green!90} $\blacksquare$};
\node at (11,-6) {\huge $\cdot$};
\node at (8,-7) {\huge $\cdot$};
\node at (9,-7) {\huge $\cdot$};
\node at (10,-7) {\huge $\square$};
\node at (11,-7) {\huge $\cdot$};
\node at (9,-8) {\huge $\cdot$};
\node at (10,-8) {\huge $\square$};
\node at (11,-8) {\huge $\cdot$};
\node at (10,-9) {\huge $\blacksquare$};
\node at (11,-9) {\huge $\blacksquare$};
\node at (11,-10) {\huge $\cdot$};
\end{tikzpicture} };
\draw (-6.8,4) -- (-6.8,-4);
\node at (-5,0) {\begin{tikzpicture}[scale=0.6]
\node at (0,-1) {1};
\node at (0,-2) {2};
\node at (0,-3) {3};
\node at (0,-4) {4};
\node at (0,-5) {5};
\node at (0,-6) {6};
\node at (1,0) {7};
\node at (2,0) {8};
\node at (3,0) {9};
\node at (4,0) {10};
\node at (5,0) {11};
\node at (1,-1) {\huge $\square$};
\node at (2,-1) {\huge $\cdot$};
\node at (3,-1) {\huge $\cdot$};
\node at (4,-1) {\huge $\square$};
\node at (5,-1) {\huge $\cdot$};
\node at (1,-2) {\huge \color{green!90} $\blacksquare$};
\node at (2,-2) {\huge \color{green!90} $\blacksquare$};
\node at (3,-2) {\huge $\square$};
\node at (4,-2) {\huge \color{green!90} $\square$};
\node at (5,-2) {\huge $\square$};
\node at (1,-3) {\huge $\blacksquare$};
\node at (2,-3) {\huge $\cdot$};
\node at (3,-3) {\huge $\cdot$};
\node at (4,-3) {\huge $\square$};
\node at (5,-3) {\huge $\cdot$};
\node at (1,-4) {\huge \color{green!90} $\blacksquare$};
\node at (2,-4) {\huge \color{green!90} $\blacksquare$};
\node at (3,-4) {\huge $\square$};
\node at (4,-4) {\huge \color{green!90} $\blacksquare$};
\node at (5,-4) {\huge $\square$};
\node at (1,-5) {\huge \color{green!90} $\blacksquare$};
\node at (2,-5) {\huge \color{blue!90} $\blacksquare$};
\node at (3,-5) {\huge $\cdot$};
\node at (4,-5) {\huge \color{green!90} $\blacksquare$};
\node at (5,-5) {\huge $\cdot$};
\node at (1,-6) {\huge \color{green!90} $\blacksquare$};
\node at (2,-6) {\huge \color{green!90} $\blacksquare$};
\node at (3,-6) {\huge $\cdot$};
\node at (4,-6) {\huge \color{green!90} $\blacksquare$};
\node at (5,-6) {\huge $\cdot$};
\end{tikzpicture}};
\node at (-5,-2.5){$\mathrm{Inv}^{-1}(w')\cap\square_k$};
\node at (-3.1,0) {$\xrightarrow[]{\Pi}$};
\node at (-1.1,0) {\begin{tikzpicture}[scale=0.6]
\node at (0,-1) {1};
\node at (0,-2) {3};
\node at (0,-3) {2};
\node at (0,-4) {5};
\node at (0,-5) {4};
\node at (0,-6) {6};
\node at (1,0) {7};
\node at (2,0) {8};
\node at (3,0) {10};
\node at (4,0) {11};
\node at (5,0) {9};
\node at (1,-1) {\huge $\square$};
\node at (2,-1) {\huge $\cdot$};
\node at (3,-1) {\huge $\square$};
\node at (4,-1) {\huge $\cdot$};
\node at (5,-1) {\huge $\cdot$};
\node at (1,-2) {\huge $\blacksquare$};
\node at (2,-2) {\huge $\cdot$};
\node at (3,-2) {\huge $\square$};
\node at (4,-2) {\huge $\cdot$};
\node at (5,-2) {\huge $\cdot$};
\node at (1,-3) {\huge \color{green!90} $\blacksquare$};
\node at (2,-3) {\huge \color{green!90} $\blacksquare$};
\node at (3,-3) {\huge \color{green!90} $\square$};
\node at (4,-3) {\huge $\square$};
\node at (5,-3) {\huge $\square$};
\node at (1,-4) {\huge \color{green!90} $\blacksquare$};
\node at (2,-4) {\huge \color{blue!90} $\blacksquare$};
\node at (3,-4) {\huge \color{green!90} $\blacksquare$};
\node at (4,-4) {\huge $\cdot$};
\node at (5,-4) {\huge $\cdot$};
\node at (1,-5) {\huge \color{green!90} $\blacksquare$};
\node at (2,-5) {\huge \color{green!90} $\blacksquare$};
\node at (3,-5) {\huge \color{green!90} $\blacksquare$};
\node at (4,-5) {\huge $\square$};
\node at (5,-5) {\huge $\square$};
\node at (1,-6) {\huge \color{green!90} $\blacksquare$};
\node at (2,-6) {\huge \color{green!90} $\blacksquare$};
\node at (3,-6) {\huge \color{green!90} $\blacksquare$};
\node at (4,-6) {\huge $\cdot$};
\node at (5,-6) {\huge $\cdot$};
\end{tikzpicture}};
\end{tikzpicture}
\end{center}
\caption{Example~\ref{ex:rook_perm_one} showing another instance of $\widehat{R}_{(i,j)}(w')=1$. } \label{fig:rook_perm_one_ex2}
\end{figure}

\begin{example} \label{ex:rook_perm_one}
Again, let $n =11$ and $w= 10\; 7\; 3\; 1\; 8\; 5\; 6\; 11\; 9\; 4\; 2$. And let~$(i,j) = (5,8)\in\mathrm{Inv}^{-1}(w)$, which is cross-saturated. Recall that the cross-saturation~$\mathcal{C}$ of $(5,8)$ in $\mathrm{Inv}^{-1}(w)$ is depicted in Figure~\ref{fig:cross_sat_rect}, and that $\mathcal{C}\subseteq \square_k$ for $k=6$. From~\eqref{eqn:rook_perm}, we get that
\begin{align} \label{eqn:rook_perm_ex}
\widehat{R}_{(5,8)} = &\mathcal{T}_{(6,7)}+\mathcal{T}^{+}_{(5,7)} +\mathcal{T}^{+}_{(6,8)}+\mathcal{T}^{+}_{(5,8)}+\mathcal{T}^{-}_{(5,8)}+ \hspace{-0.2cm} \sum_{A\subseteq \{6\}} \mathcal{T}_{g((4,7),\{5\}\cup A)} + \hspace{-0.4cm} \sum_{A\subseteq\{3,4,6\}} \hspace{-0.1cm} \mathcal{T}_{g((2,7),\{5\}\cup A)}\\ \nonumber
&+ \sum_{A\subseteq\{7\}} \mathcal{T}_{g((6,10),A)} + \sum_{A\subseteq\{7\}}\mathcal{T}^{-}_{g((5,10),\{8\}\cup A)} +  \hspace{-0.1cm} \sum_{A\subseteq\{7\}}\mathcal{T}^{+}_{g((5,10),A)} +  \hspace{-0.1cm} \sum_{A\subseteq \{6,7\}} \mathcal{T}^{-}_{g((4,8),A)} \\ \nonumber
&+\sum_{A\subseteq \{6,7\}} \mathcal{T}^{+}_{g((4,8),\{5\}\cup A)} + \sum_{A\subseteq \{3,4,6,7\}} \mathcal{T}^{-}_{g((2,8),A)} + \sum_{A\subseteq \{3,4,6,7\}} \mathcal{T}^{+}_{g((2,8),\{5\}\cup A)}  \\ \nonumber
&+ \sum_{A\subseteq\{3,4,6,7,9\}} \mathcal{T}_{g((2,10),\{5\}\cup A)} - \sum_{A\subseteq \{3,4,6,7,9\}}  \mathcal{T}_{g((2,10),\{8\}\cup A)}.
\end{align}

Let's check that~$\widehat{R}_{(5,8)}(w')=1$ for a couple of $w'\in[e,w]$. 

First let $w' = 3\; 1\; 10\; 5\; 2\; 7, \; 6\; 8\; 9\; 4\; 11$.  The left side of Figure~\ref{fig:rook_perm_one_ex1} depicts $\mathrm{Inv}^{-1}(w')$ as subset of $\mathrm{Inv}^{-1}(w)$: in this figure the boxes of $\mathrm{Inv}^{-1}(w')$ are filled-in, while the boxes of $\mathrm{Inv}^{-1}(w)$ which do not belong to $\mathrm{Inv}^{-1}(w')$ are outlined. On the right side of Figure~\ref{fig:rook_perm_one_ex1} we show that the permutation $w'$ corresponds, in the sense of Proposition~\ref{prop:perm_rect}, to the partition $(4,2,1,1)\in [\varnothing,6^5]$. Recalling the definition of the canonical edge labeling for the weak order from Section~\ref{sec:weak_order_edge_labels}, one can easily check that the $g((i,j),\mathbf{x})\in\mathrm{Irr}([e,w])$ with $\mathcal{T}^{+}_{g((i,j),\mathbf{x})}(w')=1$ are
\[g((1,10),\{3\}), g((2,7),\{3,5\}), g((6,8),\{7\}), g((8,9),\varnothing), g((4,11),\{5,6,7,8,9,10\}).\]
And the $g((i,j),\mathbf{x})\in\mathrm{Irr}([e,w])$ with $\mathcal{T}^{-}_{g((i,j),\mathbf{x})}(w')=1$ are
\[g((1,3),\varnothing), \, g((5,10),\varnothing), \, g((2,5),\{3\}), \, g((6,7),\varnothing), \, g((4,9),\{5,6,7,8\}).\]
So, using~\eqref{eqn:rook_perm_ex}, we see that the terms contributing to $\widehat{R}_{(5,8)}(w')$ are 
\[ \widehat{R}_{(5,8)}(w') = \mathcal{T}^{+}_{g((2,7),\{3,5\})}(w') + \mathcal{T}^{+}_{g((6,8),\{7\})} (w') - \mathcal{T}^{-}_{g((6,7),\varnothing)}(w') = 1+1-1 = 1.  \] 
Observe (by looking at the right side of Figure~\ref{fig:rook_perm_one_ex1}) how these terms exactly correspond to the terms in the evaluation $R_{(3,4)}( (4,2,1,1))$ of the rectangle rook $R_{(3,4)}\colon [\varnothing,6^5]\to\mathbb{Z}$:
\[R_{(3,4)}((4,2,1,1))=\mathcal{T}^{+}_{(2,3)}((4,2,1,1))+ \mathcal{T}^{+}_{(3,2)}((4,2,1,1))-\mathcal{T}^{-}_{(2,2)}((4,2,1,1))=1+1-1=1.\]
 Note that 
 \[(3,4)=\Psi(\Pi((5,8))), (2,3)=\Psi(\Pi((2,7))), (3,2)=\Psi(\Pi((6,8))), (2,2)=\Psi(\Pi((6,7))),\]
where $\Pi\colon\square_k\to\square_k$ is as depicted in Figure~\ref{fig:rook_perm_one_ex1} and $\Psi\colon\square_k\simeq k^{n-k}$ is the natural isomorphism. 

Next, let us consider $w' = 1\; 7\; 3\; 8\; 2\; 10\; 5\; 4\; 6\; 11\; 9$. The left side of Figure~\ref{fig:rook_perm_one_ex2} depicts $\mathrm{Inv}^{-1}(w')$ as subset of $\mathrm{Inv}^{-1}(w)$. On the right side of Figure~\ref{fig:rook_perm_one_ex2}, we see that~$w'$ corresponds to the partition $(5,4,3)\in [\varnothing,6^5]$. The $g((i,j),\mathbf{x})\in\mathrm{Irr}([e,w])$ with $\mathcal{T}^{+}_{g((i,j),\mathbf{x})}(w')=1$ are
\[g((1,7),\varnothing), \; g((2,10),\{3,7,8\}), \; g((4,6),\{5\}).\]
And the $g((i,j),\mathbf{x})\in\mathrm{Irr}([e,w])$ with $\mathcal{T}^{-}_{g((i,j),\mathbf{x})}(w')=1$ are
\[g((3,7),\varnothing), \; g((2,8),\{3,7\}),\; g((5,10),\{7,8\}), \; g((4,5),\varnothing), \; g((9,11),\{10\}).\]
So the terms contributing to $\widehat{R}_{(5,8)}(w')$ are 
\[ \widehat{R}_{(5,8)}(w') = -\mathcal{T}^{+}_{g((2,10),\{3,7,8\})}(w')+\mathcal{T}^{-}_{g((2,8),\{3,7\})}(w')+\mathcal{T}^{-}_{g((5,10),\{7,8\})} (w')=-1+1+1=1.\]
These terms exactly correspond to the terms in the evaluation $R_{(2,3)}((5,4,3))$ of the rectangle rook $R_{(2,3)}\colon [\varnothing,6^5]\to\mathbb{Z}$:
\[R_{(2,3)}((5,4,3))=-\mathcal{T}^{+}_{(3,4)}((5,4,3))+ \mathcal{T}^{-}_{(2,4)}((5,4,3))+\mathcal{T}^{-}_{(3,3)}((5,4,3))=-1+1+1=1.\]
 Note that 
 \[ (2,3)=\Psi(\Pi((5,8))), (3,4)=\Psi(\Pi((2,10))), (2,4)=\Psi(\Pi((2,8))), (3,3)=\Psi(\Pi((5,10))),\]
 where $\Pi\colon\square_k\to\square_k$ is as depicted in Figure~\ref{fig:rook_perm_one_ex2} and $\Psi\colon\square_k\simeq k^{n-k}$ is the natural isomorphism. 
\end{example}

We proceed to prove Lemma~\ref{lem:rook_perm_one}. To do this we will need a few more preliminary propositions. Continue to fix $w\in \mathfrak{S}_n$, with $(i,j)\in\mathrm{Inv}^{-1}(w)$ a cross-saturated box, and~$\mathcal{C}$ its cross-saturation. Also fix $k$ to be the $1 \leq k < n$ from Proposition~\ref{prop:cross_sat_rect} for which $\mathcal{C}\subseteq \square_k$. 

And now let us also fix some $w' \in [e,w]$. Let $(u,v)$ be the $(u,v) \in \mathfrak{S}_k\times \mathfrak{S}_{n-k}$ from Proposition~\ref{prop:perm_rect} applied to $w'$, and let $\Pi$ be the corresponding permutation $\Pi\colon\square_k\to\square_k$. Proposition~\ref{prop:perm_rect} tells us that, setting $S \coloneqq \Pi(\mathrm{Inv}^{-1}(w')\cap\square_k)$, we have $\Psi(S) \in [\varnothing,k^{n-k}]$.

\begin{prop} \label{prop:rook_help1}
Let $(i',j') \in \square_k$ with $(i',j') \notin \mathcal{C}$. Then $\Pi((i',j')) \not\leq_{\Phi^+} \Pi((i,j))$.
\end{prop}
\begin{proof}
Let $(i',j') \in \square_k$ with $(i',j')\notin \mathcal{C}$. Because $(i',j')\notin \mathcal{C}$, either $(i',j) \notin \mathrm{Inv}^{-1}(w)$ or $(i,j') \notin \mathrm{Inv}^{-1}(w)$. By Proposition~\ref{prop:cross_sat_rect} we have~$(i',j') \not \leq_{\Phi^{+}} (i,j)$. Hence, again by Proposition~\ref{prop:cross_sat_rect}, either $(i,j') \geq_{\Phi^{+}} (i,j)$ and $(i,j') \notin \mathrm{Inv}^{-1}(w)$, or $(i',j) \geq_{\Phi^{+}} (i,j)$ and $(i',j) \notin \mathrm{Inv}^{-1}(w)$. Assume by symmetry that $(i',j) \geq_{\Phi^{+}} (i,j)$ and $(i',j) \notin \mathrm{Inv}^{-1}(w)$. Note that this means~$i'<i$.

Now suppose to the contrary that $\Pi((i',j')) \leq_{\Phi^+} \Pi((i,j))$. For this to be the case, we need that~$(i',i)\in\mathrm{Inv}^{-1}(u)$. In particular, $(i',i)\in\mathrm{Inv}^{-1}(w)$. But if $(i',i)\in\mathrm{Inv}^{-1}(w)$, then since~$(i,j)\in\mathrm{Inv}^{-1}(w)$, by Lemma~\ref{lem:inv_classification} we get that $(i',j)\in\mathrm{Inv}^{-1}(w)$. This is a contradiction. Hence,  $\Pi((i',j')) \not\leq_{\Phi^+} \Pi((i,j))$, as desired.
\end{proof}

\begin{prop}  \label{prop:rook_help2}
Let $(i',j') \in \square_k$ with $(i',j') \notin \mathcal{C}$ and with $\Pi((i',j')) \geq_{\Phi^+} \Pi((i,j))$. Then:
\begin{itemize}
\item $\Pi((i',j')) \notin S$;
\item if $\Psi(S\cup\{\Pi((i',j'))\})\in [\varnothing,k^{n-k}]$, then $\Pi((i',j'))$ is either in the same row or same column of $\square_k$ as $\Pi((i,j))$.
\end{itemize}
\end{prop}
\begin{proof}
We begin as in the previous proposition. Let $(i',j') \in \square_k$ with $(i',j')\notin \mathcal{C}$. Because $(i',j')\notin \mathcal{C}$, either $(i',j) \notin \mathrm{Inv}^{-1}(w)$ or $(i,j') \notin \mathrm{Inv}^{-1}(w)$. By Proposition~\ref{prop:cross_sat_rect} we have~$(i',j') \not \leq_{\Phi^{+}} (i,j)$. Hence, again by Proposition~\ref{prop:cross_sat_rect}, either $(i,j') \geq_{\Phi^{+}} (i,j)$ and $(i,j') \notin \mathrm{Inv}^{-1}(w)$, or $(i',j) \geq_{\Phi^{+}} (i,j)$ and $(i',j) \notin \mathrm{Inv}^{-1}(w)$. Assume by symmetry that $(i',j) \geq_{\Phi^{+}} (i,j)$ and $(i',j) \notin \mathrm{Inv}^{-1}(w)$. 

Now suppose that~$\Pi((i',j')) \geq_{\Phi^+} \Pi((i,j))$. This means $\Pi((i',j')) \geq_{\Phi^+} \Pi((i',j))$. Since $(i',j) \notin \mathrm{Inv}^{-1}(w)$, certainly $\Pi((i',j)) \notin S$. But because $\Psi(S)\in[\varnothing,k^{n-k}]$, i.e., because $S$ is an order ideal of $\square_k$, we must have $\Pi((i',j'))\notin S$. This establishes the first bullet point.

Moreover, the only way $S\cup\{\Pi((i',j'))\}$ could be an order ideal of $\square_k$ is if we have~$(i',j')=(i',j)$. In other words, this is only possible if $\Pi((i',j'))$ is in the same column of $\square_k$ as $\Pi((i,j))$. This establishes the second bullet point.
\end{proof}

\begin{prop}  \label{prop:rook_help3}
Let $(i',j') \in \mathrm{Inv}^{-1}(w)$ be a box either in the same row or same column as~$(i,j)$, and with $(i',j')\leq_{\Phi^+} (i,j)$. Then $\Pi((i',j')) \leq_{\Phi^+} \Pi((i,j))$.
\end{prop}
\begin{proof}
Let~$(i',j') \in \mathrm{Inv}^{-1}(w)$ be a box either in the same row or same column as~$(i,j)$, and with $(i',j')\leq_{\Phi^+} (i,j)$. Either we have $(i',j') = (i',j)$ with $i' > i$ or $(i',j')=(i,j')$ with $j' < j$. Assume by symmetry that $(i',j') = (i,j')$ with $j' <j$. Note that we certainly have~$(i,j') \in \mathcal{C}$. Also note that $j'>k$, since $\mathcal{C}\subseteq\square_k$ (by Proposition~\ref{prop:cross_sat_rect}).

Suppose to the contrary that $\Pi((i',j')) \not\leq_{\Phi^+} \Pi((i,j))$.  For this to be the case, we need that~$(j',j)\in\mathrm{Inv}^{-1}(v)$. In particular, $(j',j)\in\mathrm{Inv}^{-1}(w)$.  But then $(j',j)$ would be a box of $\mathrm{Inv}^{-1}(w)$ in the same column as $(i,j)$, and hence we must have $(j',j)\in\mathcal{C}$. But then the fact that $j'>k$ means that $(j',j)\notin \square_k$. This contradicts that~$\mathcal{C}\subseteq\square_k$. Hence we conclude that $\Pi((i',j')) \not\leq_{\Phi^+} \Pi((i,j))$, as desired.
\end{proof}

We can now prove Lemma~\ref{lem:rook_perm_one}.

\begin{proof}[Proof of Lemma~\ref{lem:rook_perm_one}]
We will show that $\widehat{R}_{(i,j)}(w')=R_{\Psi(\Pi((i,j)))}(\Psi(S))$ by matching up their terms. 

First of all, we claim that if $(i',j')\in\square_k$ is a box for which $\mathcal{T}^{\pm}_{\Psi(\Pi((i',j')))}(\Psi(S))=1$ and $\mathcal{T}^{\pm}_{\Psi(\Pi((i',j')))}$ has nonzero coefficient in $R_{\Psi(\Pi((i,j)))}$ (where $\pm$ is some choice of sign $\pm\in\{+,-\}$), then $(i',j')\in\mathcal{C}$. In fact, this follows immediately from Propositions~\ref{prop:rook_help1} and~\ref{prop:rook_help2}, together with the definition~\eqref{eqn:rook_rect} of the rectangle rooks. 

Now let $(i',j') \in \mathcal{C}$ be a box for which $\mathcal{T}^{\pm}_{\Psi(\Pi((i',j')))}(\Psi(S))=1$ and $\mathcal{T}^{\pm}_{\Psi(\Pi((i',j')))}$ has nonzero coefficient in $R_{\Psi(\Pi((i,j)))}$ (where $\pm$ is some choice of sign $\pm\in\{+,-\}$). We know from Proposition~\ref{prop:perm_rect} that $\mathcal{T}^{\pm}_{(i',j')}(w')=1$. Hence in particular there is a unique~$g((i',j'),\mathbf{x})\in\mathrm{Irr}([e,w])$ such that $\mathcal{T}^{\pm}_{g((i',j'),\mathbf{x})}(w')=1$. Explicitly, the $\mathbf{x}$ in question is:
\[\mathbf{x} \coloneqq  \{x: i'+1 \leq x \leq j'-1, (i',x) \in \mathrm{Inv}^{-1}(w')\}.\]
What we now want is that the coefficient of $\mathcal{T}^{\pm}_{g((i',j'),\mathbf{x})}$ in $\widehat{R}_{(i,j)}$ is the same as the coefficient of~$\mathcal{T}^{\pm}_{\Psi(\Pi((i',j')))}$ in $R_{\Psi(\Pi((i,j)))}$. Let us explain why this is indeed the case. Proposition~\ref{prop:rook_help3} says that certain possibilities for the relative position in $\square_k$ of $\Pi((i',j'))$ and $\Pi((i,j))$ cannot occur: e.g., this proposition implies that if $i' < i$ and $j' < j$, then we cannot have $\Pi((i',j'))\geq_{\Phi^{+}}\Pi((i,j))$. Among the possibilities for the relative position in $\square_k$ of~$\Pi((i',j'))$ and~$\Pi((i,j))$ which do actually occur, this relative position is dictated by the status of $i$ and $j$ in $\mathbf{x}$. One can then check that the definition~\eqref{eqn:rook_perm} of the permutation rook is exactly what is required so that the coefficient of $\mathcal{T}^{\pm}_{g((i',j'),\mathbf{x})}$ in $\widehat{R}_{(i,j)}$ is the same as the coefficient of $\mathcal{T}^{\pm}_{\Psi(\Pi((i',j')))}$ in $R_{\Psi(\Pi((i,j)))}$. 

Conversely, if $g((i',j'),\mathbf{x}) \in \mathrm{Irr}([e,w])$ is such that $\mathcal{T}^{\pm}_{g((i',j'),\mathbf{x})}(w)=1$ and $\mathcal{T}^{\pm}_{g((i',j'),\mathbf{x})}$ has nonzero coefficient in $\widehat{R}_{(i,j)}$, then it similarly follows that $\mathcal{T}^{\pm}_{\Psi(\Pi((i',j')))}(\Psi(S))=1$ and $\mathcal{T}^{\pm}_{\Psi(\Pi((i',j')))}$ has the same coefficient in $R_{\Psi(\Pi((i,j)))}$.

So we have shown that the nonzero terms in  $\widehat{R}_{(i,j)}(w')$ and $R_{\Psi(\Pi((i,j)))}(\Psi(S))$ can be matched up, and thus that $\widehat{R}_{(i,j)}(w')=R_{\Psi(\Pi((i,j)))}(\Psi(S))$. By Lemma~\ref{lem:rook_rect_one} we have that~$R_{\Psi(\Pi((i,j)))}(\Psi(S))=1$, and hence that~$\widehat{R}_{(i,j)}(w')=1$.
\end{proof}

This completes our development of rooks for permutations. In the next subsection we will show how to place rooks on the inverse inversion set of a skew vexillary permutation of balanced shape to deduce that its initial weak order interval is tCDE.

\subsection{The rook placement for skew vexillary permutations of balanced shape}

Having done the hard work of constructing the rooks for permutations in the previous subsection, it will now be easy to use them to prove our main result.

First we prove a fact mentioned earlier: that balanced shapes have enough cross-saturated boxes to be able to place rooks on cross-saturated boxes so that every box is attacked the same number of times.

\begin{figure}
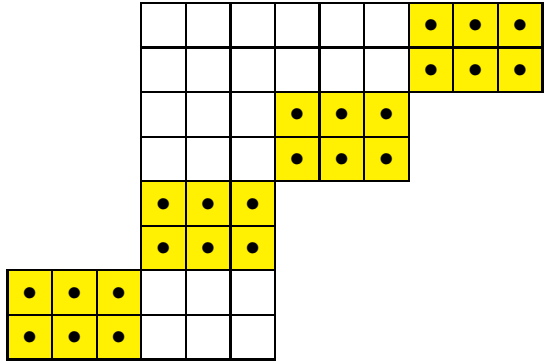

\begin{center}
\begin{ytableau} 
\none & \none & \none & & & & & & & *(yellow)  \bullet & *(yellow)  \bullet & *(yellow)  \bullet \\ 
\none & \none & \none & & & & & & & *(yellow)  \bullet & *(yellow)  \bullet & *(yellow)  \bullet  \\ 
\none & \none & \none & & & &   *(yellow)  \bullet & *(yellow)  \bullet & *(yellow)  \bullet  \\ 
\none & \none & \none & & & &   *(yellow)  \bullet & *(yellow)  \bullet & *(yellow)  \bullet \\ 
\none & \none & \none &   *(yellow)  \bullet & *(yellow)  \bullet & *(yellow)  \bullet \\
\none & \none & \none &   *(yellow)  \bullet & *(yellow)  \bullet & *(yellow)  \bullet \\
*(yellow)  \bullet & *(yellow) \bullet & *(yellow)  \bullet & & & \\
*(yellow)  \bullet & *(yellow)  \bullet & *(yellow)  \bullet & & & \\
\end{ytableau}
\end{center}
\caption{The ``diagonal'' cross-saturated rectangles of a balanced shape, highlighted in yellow (and with bullets in them).} \label{fig:balanced_cross_sat}
\end{figure}

\begin{lemma} \label{lem:rook_placement}
Let $\sigma = \lambda/\nu$ be a balanced shape of height $a$ and width $b$. Then there exist coefficients $c_{(i,j)} \in \mathbb{Z}$ for the boxes $(i,j) \in \sigma$ such that:
\begin{itemize}
\item $c_{(i,j)} \neq 0$ only if $(i,j)$ is a cross-saturated box of $\sigma$;
\item for any fixed box $(i,j)\in \sigma$, we have~$\sum_{(i',j)\in \sigma} c_{(i',j)}=a$ and $\sum_{(i,j')\in \sigma} c_{(i,j')}=b$;
\item $\sum_{(i,j)\in \sigma} c_{(i,j)} = ab$.
\end{itemize}
\end{lemma}
\begin{proof}
Let $m \coloneqq  \mathrm{gcd}(a,b)$, and let $a' \coloneqq  a/m$ and $b' \coloneqq  b/m$. By considering the possible locations of outward corners in $\sigma$, it is easy to see that we must have $\sigma = \sigma' \circ (b')^{(a')}$ for $\sigma'$ a balanced shape of height and width both equal to $m$. Because $\sigma'$ is balanced of square dimensions, the ``diagonal'' boxes of $\sigma'$ (i.e., the boxes $(m,1),(m-1,2),\ldots,(1,m)$) are cross-saturated in $\sigma'$. Moreover, for each diagonal box of $\sigma'$, all the boxes in the $a'\times b'$ rectangle that this box becomes when we obtain $\sigma$ from $\sigma'$ are cross-saturated in $\sigma$. Figure~\ref{fig:balanced_cross_sat} depicts these ``diagonal'' cross-saturated $a'\times b'$ rectangles for a balanced shape. Within each of these diagonal $a'\times b'$ rectangles of cross-saturated boxes, let us define the coefficients $c_{(i,j)}$ as in Figure~\ref{fig:rect_coeffs} (if the box is empty in the figure, that means the corresponding coefficient is zero). It is easy to see that the sum of the coefficients in each row of the rectangle is $b$, and the sum of the coefficients in each column is $a$: this requires checking
\[ (b'-1)a + (a + b - \frac{ab}{m}) = \frac{ab}{m} - a + a + b - \frac{ab}{m} = b\]
and
\[(a'-1)b + (a + b - \frac{ab}{m}) = \frac{ab}{m} - b + a + b - \frac{ab}{m} = a.\]
Then observe that each row of $\sigma$ intersects a unique one of these diagonal $a'\times b'$ rectangles, and so does each column. Thus these coefficients satisfy the first two bulleted conditions. For the third bullet point: we compute that the sum of the coefficients in each $a'\times b'$ rectangle is
\[ (b'-1)a+(a'-1)b + (a + b - \frac{ab}{m}) = \frac{ab}{m} - a + \frac{ab}{m} - b + a + b - \frac{ab}{m} = \frac{ab}{m}; \]
and there are $m$ such rectangles in $\sigma$, so the total sum of the $c_{(i,j)}$ coefficients is $ab$, as required.
\end{proof}

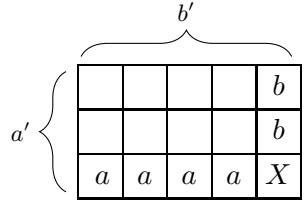
\begin{figure}
\begin{tikzpicture}
\node at (0,0) {\begin{ytableau}
 \, & & & &  b \\
 \, & & & &  b \\
 \, a & a & a & a &  X \\
\end{ytableau}};
\draw [decorate,decoration={brace,amplitude=10pt},xshift=-4pt,yshift=0pt] (-1.5,-0.8) -- (-1.5,0.8) node [black,midway,xshift=-0.6cm]  {\footnotesize $a'$};
\draw [decorate,decoration={brace,amplitude=10pt},xshift=-4pt,yshift=0pt] (-1.3,1) -- (1.5,1) node [black,midway,yshift=0.6cm]  {\footnotesize $b'$};
\end{tikzpicture}
\caption{An illustration of how to define coefficients $c_{(i,j)}$ in an $a'\times b'$ diagonal rectangle, as in Lemma~\ref{lem:rook_placement}. Here $X \coloneqq  a + b - \frac{ab}{m}$.}  \label{fig:rect_coeffs}
\end{figure}

Since there is a good way to place rooks on cross-saturated boxes of a balanced shape, there is also a good way to place rooks on the inverse inversion set of a skew vexillary permutation of balanced shape. This is because if $D$ is a diagram and $(i,j)$ a cross-saturated box of $D$, and if $\Pi$ is a permutation of the rows and columns of $D$, then $\Pi((i,j))$ is a cross-saturated box of $\Pi(D)$. Thus, we have the following:

\begin{cor} \label{cor:rook_placement}
Let $\sigma = \lambda/\nu$ be a balanced shape of height $a$ and width $b$, and $w\in \mathfrak{S}_n$ be a skew vexillary permutation of shape~$\sigma$. Then there exist coefficients $c_{(i,j)} \in \mathbb{Z}$ for the boxes $(i,j) \in \mathrm{Inv}^{-1}(w)$ such that:
\begin{itemize}
\item $c_{(i,j)} \neq 0$ only if $(i,j)$ is a cross-saturated box of $\mathrm{Inv}^{-1}(w)$;
\item for any fixed box $(i,j)\in \mathrm{Inv}^{-1}(w)$, we have~$\sum_{(i',j)\in \mathrm{Inv}^{-1}(w)} c_{(i',j)}=b$ and $\sum_{(i,j')\in \mathrm{Inv}^{-1}(w)} c_{(i,j')}=a$;
\item $\sum_{(i,j)\in \mathrm{Inv}^{-1}(w)} c_{(i,j)} = ab$.
\end{itemize}
\end{cor}
\begin{proof}
That $w$ is a skew vexillary permutation of shape~$\sigma$ means that $\mathrm{Inv}^{-1}(w)$ can be transformed by some permutation of rows and columns into $\sigma^t$, a balanced shape of height~$b$ and width~$a$. Let $\Pi\colon\mathrm{Inv}^{-1}(w)\to\sigma^t$ be the permutation that achieves this. Let the coefficients $\widetilde{c}_{(i,j)}\in\mathbb{Z}$ for $(i,j)\in\sigma^t$ be as guaranteed by Lemma~\ref{lem:rook_placement}. Then define $c_{(i,j)} \coloneqq  \widetilde{c}_{\Pi((i,j))}$ for each $(i,j)\in \mathrm{Inv}^{-1}$.  These $c_{(i,j)}$ will satisfy the required properties because $\Pi$ preserves rows and columns, and hence preserves the property of being cross-saturated.
\end{proof}

Now we can put everything together and prove the main result:

\begin{thm} \label{thm:main}
Let $\sigma = \lambda/\nu$ be a balanced shape of height $a$ and width $b$, and $w\in \mathfrak{S}_n$ a skew vexillary permutation of shape~$\sigma$. Then $[e,w]$ is tCDE with edge density~$ab/(a+b)$.
\end{thm}
\begin{proof}
Let $\mu$ be a toggle-symmetric distribution on $[e,w]$. Let the coefficients $c_{(i,j)} \in \mathbb{Z}$ for boxes $(i,j) \in \mathrm{Inv}^{-1}(w)$ be as in Corollary~\ref{cor:rook_placement}. Consider the statistic 
\[f \coloneqq  \sum_{(i,j)\in \mathrm{Inv}^{-1}(w)}c_{(i,j)}\widehat{R}_{(i,j)}\colon [e,w]\to\mathbb{Z},\] 
where the rooks $\widehat{R}_{(i,j)}$ are defined by~\eqref{eqn:rook_perm} (this statistic is well-defined because the~$(i,j)$ for which $c_{(i,j)}\neq 0$ are cross-saturated boxes). For any fixed box $(i,j)\in \mathrm{Inv}^{-1}(w)$, we have~$\sum_{(i',j)\in \mathrm{Inv}^{-1}(w)} c_{(i',j)}=b$ and $\sum_{(i,j')\in \mathrm{Inv}^{-1}(w)} c_{(i,j')}=a$, so Lemma~\ref{lem:rook_perm_attack} tells us that
\[ \mathbb{E}(\mu; f) = \sum_{(i,j) \in  \mathrm{Inv}^{-1}(w)}(a+b) \mathbb{E}(\mu;\mathcal{T}^{-}_{(i,j)})=(a+b) \mathbb{E}(\mu;\mathrm{ddeg}).\]
On the other hand, since $\sum_{(i,j)\in \mathrm{Inv}^{-1}(w)} c_{(i,j)} = ab$, Lemma~\ref{lem:rook_perm_one} tells us that $f(w') = ab$ for any $w'\in[e,w]$. Thus, $\mathbb{E}(\mu; f) = ab$. Putting these two expressions for $\mathbb{E}(\mu; f)$ together, we get that $(a+b)\mathbb{E}(\mu;\mathrm{ddeg})=ab$, or in other words, $\mathbb{E}(\mu;\mathrm{ddeg})=ab/(a+b)$, as required.
\end{proof}

\begin{cor} \label{cor:main}
Let $\sigma = \lambda/\nu$ be a balanced shape of height $a$ and width $b$, and $w\in \mathfrak{S}_n$ a skew vexillary permutation of shape~$\sigma$. Then $[e,w]$ is CDE with edge density~$ab/(a+b)$.
\end{cor}
\begin{proof}
This follows from Theorem~\ref{thm:main} and Corollary~\ref{cor:tcde_implies_cde}.
\end{proof}

\subsection{Refined down-degree expectations for the full weak order} \label{sec:full_weak_order}

Our main result (Theorem~\ref{thm:main}) applies to the full weak order $\mathfrak{S}_n$. This is because $\mathfrak{S}_n=[e,w_0]$ for the longest word $w_0 \in \mathfrak{S}_n$, and $w_0$ is a vexillary permutation of the staircase shape~$\delta_n = (n-1,n-2,\ldots,1)$. In fact, the original result of Reiner-Tenner-Yong~\cite[Theorem~1.1]{reiner2018poset} mentioned in the introduction also applies to the full weak order $\mathfrak{S}_n$ since $w_0$ is dominant of shape~$\delta_n$. 

But actually, the CDE property is not interesting for the full weak order because the full weak order $\mathfrak{S}_n$ is self-dual and has constant Hasse diagram degree~$n-1$. As mentioned in the introduction, for such posets the CDE property holds for simple reasons (see~\cite[Proposition 2.20]{reiner2018poset}). Moreover, for the tCDE property we do not even need to assume self-duality; that is to say, any semidistributive lattice~$L$ which has constant Hasse diagram degree is always tCDE for simple reasons. Let us explicitly spell out why this is. Suppose~$L$ is a semidistributive lattice of constant Hasse diagram degree~$d$ and let $\mu$ be a toggle-symmetric distribution on~$L$. Then certainly $\mathbb{E}(\mu;\sum_{p \in \mathrm{Irr}(L)} \mathcal{T}^{-}_p) = \mathbb{E}(\mu;\sum_{p \in \mathrm{Irr}(L)} \mathcal{T}^{+}_p)$ and hence
\[\mathbb{E}(\mu;\sum_{p \in \mathrm{Irr}(L)} \mathcal{T}^{-}_p) = \frac{1}{2}\left(\mathbb{E}(\mu;\hspace{-0.2cm}\sum_{p \in \mathrm{Irr}(L)} \hspace{-0.3cm} \mathcal{T}^{-}_p)+\mathbb{E}(\mu;\hspace{-0.2cm}\sum_{p \in \mathrm{Irr}(L)} \hspace{-0.3cm} \mathcal{T}^{+}_p)\right) = \frac{1}{2}\mathbb{E}(\mu;\hspace{-0.2cm} \sum_{p \in \mathrm{Irr}(L)}\hspace{-0.3cm} \mathcal{T}^{-}_p+\sum_{p \in \mathrm{Irr}(L)} \hspace{-0.3cm} \mathcal{T}^{+}_p).\]
But $(\sum_{p \in \mathrm{Irr}(L)} \mathcal{T}^{-}_p+\sum_{p \in \mathrm{Irr}(L)} \mathcal{T}^{+}_p)(w)=d$ for any $w\in \mathfrak{S}_n$ since $L$ has constant Hasse diagram degree~$d$, and thus we have~$\mathbb{E}(\mu;\sum_{p \in \mathrm{Irr}(L)} \mathcal{T}^{-}_p)=\frac{d}{2}$.

So Theorem~\ref{thm:main} does not yield any interesting result when applied to all of~$\mathfrak{S}_n$.

However, via our method of rooks we can obtain an interesting result which ``refines'' this down-degree expectation for the full weak order. Namely, define for each $1 \leq k < n$ the statistic $f_k\colon \mathfrak{S}_n\to\mathbb{Z}$ by
\[f_k \coloneqq \sum_{j=1}^{k} \mathcal{T}^{-}_{(j,k+1)} + \sum_{j=k+1}^{n} \mathcal{T}^{-}_{(k,j)}.\]
Note that $\sum_{k=1}^{n-1}f_k = 2\cdot \mathrm{ddeg}$, and this is the sense in which the $f_k$ ``refine'' the down-degree statistic. The refined down-degree expectation for the full weak order is:

\begin{thm} \label{thm:full_weak_order}
Let $\mu$ be a toggle-symmetric distribution on $\mathfrak{S}_n$. Then $\mathbb{E}(\mu; f_k) = 1$ for any $1\leq k < n$.
\end{thm}
\begin{proof}
Let $1\leq k < n$. It is easy to see that $(k,k+1)$ is a cross-saturated box of $\mathrm{Inv}^{-1}(w_0)$ because $\mathrm{Inv}^{-1}(w_0)=\Phi^{+}$. Hence we can consider the rook $R_{(k,k+1)}\colon \mathfrak{S}_n\to\mathbb{Z}$. By Lemma~\ref{lem:rook_perm_attack} we have that $\mathbb{E}(\mu; R_{(k,k+1)})=\mathbb{E}(\mu; f_k)$. And by Lemma~\ref{lem:rook_perm_one} we have that~$\mathbb{E}(\mu; R_{(k,k+1)})(w)=1$ for all~$w\in\mathfrak{S}_n$. Hence $\mathbb{E}(\mu; f_k) = 1$, as claimed.
\end{proof}

It would be interesting to find any applications of Theorem~\ref{thm:full_weak_order}. In the next section we will explain how this theorem does give a homomesy result for rowmotion acting on the full weak order.

\section{Down-degree homomesy for rowmotion on semidistributive lattices} \label{sec:rowmotion}

In this section we explain an application of our result to dynamical algebraic combinatorics and the study of the rowmotion operator.

\subsection{Review of rowmotion for distributive lattices}

Let $L = \mathcal{J}(P)$ be a distributive lattice. \emph{Rowmotion} on $L$ is the map $\mathrm{row}\colon \mathcal{J}(P)\to\mathcal{J}(P)$ defined by
\[ \mathrm{row}(I) \coloneqq  \{p \in P\colon p \leq q \textrm{ for some $q \in \mathrm{min}(P\setminus I)$}\}, \]
where $\mathrm{min}(P\setminus I)$ denotes the minimal elements of $P$ not in $I$. Rowmotion and its generalizations have been the focus of research of many authors~\cite{brouwer1974period, fonderflaass1993orbits, cameron1995orbits, panyushev2009orbits, armstrong2013uniform, striker2012promotion, rush2013orbits, rush2015orbits, grinberg2016birational1, grinberg2015birational2, striker2018rowmotion}.

The first thing to observe about rowmotion is that it is invertible. This might not be immediately obvious from the above definition, but it does follow from a description, due to Cameron and Fon-der-Flaass~\cite{cameron1995orbits}, of rowmotion as a composition of toggles:

\begin{thm}[{Cameron-Fon-der-Flaass~\cite[Lemma 1]{cameron1995orbits}}] \label{thm:rowmotion_toggles}
Let $p_1, p_2,\ldots, p_{\#P}$ be any linear extension of $P$. Then $\mathrm{row} = \tau_{p_1}\circ\tau_{p_2}\circ \cdots \circ \tau_{p_{\#P}}$.
\end{thm}

The poset on which the action of rowmotion has been studied the most is the distributive lattice $L = \mathcal{J}([a]\times[b]) = [\varnothing,b^a]$ corresponding to the product of two chains.

\begin{example}
Let $a=b=2$ and consider rowmotion acting on $\mathcal{J}([2]\times[2]) = [\varnothing,2^2]$. Then rowmotion has the following two orbits (where we depict an order ideal $\nu \in [\varnothing,2^2]$ by shading its boxes in yellow):
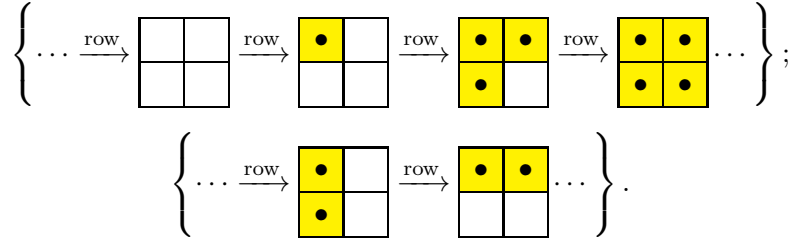
\[  \left\{\cdots \xrightarrow{\mathrm{row}} \begin{ytableau} \, & \, \\ \, & \, \end{ytableau}  \xrightarrow{\mathrm{row}} \begin{ytableau} *(yellow) \bullet  & \, \\ \, & \, \end{ytableau} \xrightarrow{\mathrm{row}} \begin{ytableau} *(yellow) \bullet  & *(yellow) \bullet   \\ *(yellow) \bullet   & \, \end{ytableau} \xrightarrow{\mathrm{row}} \begin{ytableau} *(yellow) \bullet  & *(yellow) \bullet   \\ *(yellow) \bullet   & *(yellow) \bullet \end{ytableau} \cdots\right \} ;\]
\[  \left\{ \cdots \xrightarrow{\mathrm{row}} \begin{ytableau} *(yellow)\bullet & \, \\ *(yellow)\bullet & \, \end{ytableau}  \xrightarrow{\mathrm{row}} \begin{ytableau} *(yellow) \bullet  & *(yellow) \bullet  \\ \, & \, \end{ytableau} \cdots\right\} . \]
In particular, observe that the order of rowmotion is $4$, and that the average value of~$\mathrm{ddeg}$ along each orbit is~$1$.
\end{example}

Initially the main interest was in understanding the orbit structure of rowmotion acting on $\mathcal{J}([a]\times[b])$, and in particular in computing its order. For example, Brouwer and Schrijver~\cite{brouwer1974period} proved the following:

\begin{thm}[{Brouwer-Schrijver~\cite[Theorem 3.6]{brouwer1974period}}]
The order of $\mathrm{row}$ acting on $\mathcal{J}([a]\times[b])$ is $a+b$.
\end{thm}

Recently, in the context of dynamical algebraic combinatorics, various authors have become interested in other aspects of rowmotion beyond its orbit structure. One particular goal has been to exhibit ``homomesies'' for rowmotion. So let's review the homomesy paradigm of Propp-Roby~\cite{propp2015homomesy, einstein2013combinatorial}.

\begin{definition}
Let $X$ be a finite combinatorial set, $\Phi\colon X\to X$ an invertible operator, and $f\colon X \to \mathbb{R}$ some statistic. We say that the triple $(X,\Phi,f)$ exhibits \emph{homomesy} if the average of $f$ along every $\Phi$-orbit of $X$ is the same. In this case we also say that $f$ is \emph{homomesic} with respect to the action of $\Phi$ on $X$, and we say $f$ is \emph{$c$-mesic} if the average of $f$ along every $\Phi$-orbit is $c\in\mathbb{R}$.
\end{definition}

Propp-Roby~\cite{propp2015homomesy} exhibited homomesies with a number of different statistics for rowmotion acting on $\mathcal{J}([a]\times[b])$. The statistic that will be most relevant for us is the ``antichain cardinality'' statistic: this is the statistic $\mathcal{J}(P)\to \mathbb{Z}$ defined by $I \mapsto \#\mathrm{max}(I)$, where $\mathrm{max}(I)$ is the set of maximal elements of $I$.

\begin{thm}[{Propp-Roby~\cite[Theorem 27]{propp2015homomesy}}] \label{thm:propp-roby}
The statistic $\#\mathrm{max}(I)$ is $ab/(a+b)$-mesic with respect to the action of $\mathrm{row}$ on $\mathcal{J}([a]\times[b])$.
\end{thm} 

Note that for $I \in L=\mathcal{J}(P)$, the antichain cardinality $\#\mathrm{max}(I)$ is just $\mathrm{ddeg}_L(I)$. This observation, together with an observation of Striker~\cite{striker2015toggle}, connects the study of tCDE distributive lattices to rowmotion homomesies. Let's review Striker's observation:

\begin{lemma}[{Striker~\cite[Lemma~6.2]{striker2015toggle}}] \label{lem:striker}
Let $\mathcal{O}$ be an orbit of $\mathrm{row}$ acting on $\mathcal{J}(P)$. Then the distribution $\mu$ which is uniform on $\mathcal{O}$ and zero outside of~$\mathcal{O}$ is toggle-symmetric.
\end{lemma}
\begin{proof}
For any $p\in P$ and $I\in\mathcal{J}(P)$, we have $\mathcal{T}_p(I)=1$ if and only if $\mathcal{T}_p(\mathrm{row}(I))=-1$. So along a rowmotion orbit, the nonzero values of $\mathcal{T}_p$ must alternate $1,-1,1,-1,\ldots$. Hence the average value of $\mathcal{T}_p$ along this orbit must be zero.
\end{proof}

\begin{cor} \label{cor:striker}
Suppose $L$ is tCDE with edge density $c$. Then the antichain cardinality statistic $\#\mathrm{max}(I)$ is $c$-mesic with respect to the action of $\mathrm{row}$ on~$L$.
\end{cor}
\begin{proof}
Let $\mathcal{O}$ be an orbit of $\mathrm{row}$ acting on $L$ and let $\mu$ be the distribution which uniform on $\mathcal{O}$ and zero outside of~$\mathcal{O}$. By Lemma~\ref{lem:striker}, $\mu$ is toggle-symmetric. Hence because $L$ is tCDE with edge density $c$, $\mathbb{E}(\mu; \mathrm{ddeg}) = c$. But for $I \in L=\mathcal{J}(P)$, $\mathrm{ddeg}(I)=\#\mathrm{max}(I)$. So in other words, the average of $\#\mathrm{max}(I)$ along the orbit $\mathcal{O}$ is~$c$, which is precisely what was claimed.
\end{proof}

Corollary~\ref{cor:striker} allowed Chan-Haddadan-Hopkins-Moci~\cite{chan2017expected} to deduce from their main result (Theorem~\ref{thm:chhm}) that the antichain cardinality statistic is homomesic with respect to the action of rowmotion on the interval of Young's lattice corresponding to a balanced shape:

\begin{cor}[{Chan-Haddadan-Hopkins-Moci~\cite[Corollary 3.11]{chan2017expected}}] \label{cor:balanced_homo}
Let $\sigma=\lambda/\nu$ be a balanced shape of height $a$ and width $b$. Then the antichain cardinality statistic $\#\mathrm{max}(I)$ is $ab/(a+b)$-mesic with respect to the action of $\mathrm{row}$ on $[\nu,\lambda]$.
\end{cor}

Observe that Corollary~\ref{cor:balanced_homo} is a generalization of Theorem~\ref{thm:propp-roby} to many shapes beyond rectangles.

\subsection{Rowmotion for semidistributive lattices} \label{sec:semi_rowmotion}

We want to generalize the story in the previous subsection to semidistributive lattices. Barnard~\cite{barnard2016canonical} and Thomas-Williams~\cite{thomas2017rowmotion} recently explained how rowmotion does generalize in a natural way to the semidistributive setting.

So now let $L$ be a semidistributive lattice, with $\gamma$ the canonical $\gamma$-labeling of the edges of $L$ from Section~\ref{sec:tcde_def}. Recall that $\gamma(x\lessdot y)\in\mathrm{Irr}(L)$ for every cover relation $x\lessdot y \in L$. Following Thomas-Williams, we define the sets $D^{\gamma}(y), U^{\gamma}(y)\subseteq \mathrm{Irr}(L)$ of \emph{downwards} and \emph{upwards labels at $y$} for each $y \in L$ to be
\begin{align*}
D^{\gamma}(y) &\coloneqq  \{\gamma(x\lessdot y)\colon x\in L \textrm{ with } x\lessdot y\}; \\
U^{\gamma}(y) &\coloneqq  \{\gamma(y\lessdot z)\colon z\in L \textrm{ with } y\lessdot z\}.
\end{align*}
It follows from the work of Barnard~\cite{barnard2016canonical} that $\gamma$ is a \emph{descriptive} labeling in the sense of Thomas-Williams; this means that
\begin{itemize}
\item $y\in L$ is determined by $D^{\gamma}(y)$ (in the sense that if $D^{\gamma}(x)=D^{\gamma}(y)$ for $x,y\in L$, then $x=y$);
\item $y\in L$ is also determined by $U^{\gamma}(y)$;
\item $\{D^{\gamma}(y)\colon y \in L\} = \{U^{\gamma}(y)\colon y\in L\}$.
\end{itemize}
See~\cite[\S6.2]{thomas2017rowmotion} for a proof that the $\gamma$-labeling is descriptive.

We define \emph{rowmotion} to be the map $\mathrm{row}\colon L \to L$ as follows: 
\[\mathrm{row}(y) \coloneqq  \textrm{ the unique $x\in L$ with $D^{\gamma}(x) = U^{\gamma}(y)$}.\]
Rowmotion is well-defined because the $\gamma$-labeling is descriptive. And, also because the $\gamma$-labeling is descriptive, we have that~$\mathrm{row}$ is invertible.

Note, however, that for arbitrary semidistributive lattices (or even just for our case of interest, namely,  for weak order intervals), rowmotion {\bf cannot} be computed as a sequence of toggles like in Theorem~\ref{thm:rowmotion_toggles}. In the language of Thomas-Williams, rowmotion for intervals of weak order cannot be computed ``in slow motion.'' 

At any rate, for our purposes what matters is that the observation of Striker continues to apply to this semidistributive rowmotion. 

\begin{lemma}\label{lem:striker_semi}
Let $\mathcal{O}$ be an orbit of $\mathrm{row}$ acting on the semidistributive lattice $L$. Then the distribution $\mu$ which is uniform on $\mathcal{O}$ and zero outside of~$\mathcal{O}$ is toggle-symmetric.
\end{lemma}
\begin{proof}
Exactly the same proof as the proof of Lemma~\ref{lem:striker_semi} works. For any $p\in \mathrm{Irr}(L)$ and $x\in L$, we have $\mathcal{T}_p(x)=1$ if and only if $\mathcal{T}_p(\mathrm{row}(x))=-1$. So along a rowmotion orbit, the nonzero values of $\mathcal{T}_p$ must alternate $1,-1,1,-1,\ldots$. Hence the average value of $\mathcal{T}_p$ along this orbit must be zero.
\end{proof}

\begin{cor} \label{cor:striker_semi}
Suppose the semidistributive lattice $L$ is tCDE with edge density $c$. Then $\mathrm{ddeg}$ is $c$-mesic with respect to the action of $\mathrm{row}$ on~$L$.
\end{cor}
\begin{proof}
This follows from Lemma~\ref{lem:striker_semi} just as Corollary~\ref{cor:striker} follows from Lemma~\ref{lem:striker}.
\end{proof}

\begin{example}
Let $L$ be the semidistributive lattice $L$ from Example~\ref{ex:tcde_not_cde} depicted in Figure~\ref{fig:tcde_not_cde}. Then there are two orbits of $\mathrm{row}$ acting on $L$:
\[\{ \cdots \xrightarrow{\mathrm{row}} 1  \xrightarrow{\mathrm{row}} 9   \xrightarrow{\mathrm{row}} 4  \xrightarrow{\mathrm{row}} 10  \xrightarrow{\mathrm{row}} 3  \xrightarrow{\mathrm{row}} 6   \xrightarrow{\mathrm{row}} 2  \xrightarrow{\mathrm{row}} 8  \xrightarrow{\mathrm{row}} 12 \cdots \}; \]
\[\{ \cdots \xrightarrow{\mathrm{row}} 5  \xrightarrow{\mathrm{row}}  11  \xrightarrow{\mathrm{row}} 7\cdots \}.\]
The average value of $\mathrm{ddeg}$ along the first orbit is
\[ \frac{1}{9}(0+2+1+2+1+1+1+2+2)=\frac{12}{9} = \frac{4}{3},\]
while the average value of $\mathrm{ddeg}$ along the second orbit is
\[ \frac{1}{3}(1+2+1) = \frac{4}{3}.\]
This agrees with Corollary~\ref{cor:striker_semi}: recall that we showed in Example~\ref{ex:tcde_not_cde} that $L$ is tCDE with edge density~$\frac{4}{3}$.
\end{example}

Our main result (Theorem~\ref{thm:main}) together with Corollary~\ref{cor:striker_semi} yields the following homomesy corollary:

\begin{cor} \label{cor:balanced_vex_homo}
Let $\sigma=\lambda/\nu$ be a balanced shape of height $a$ and width $b$, and $w\in \mathfrak{S}_n$ a skew vexillary permutation of shape~$\sigma$. Then $\mathrm{ddeg}$ is $ab/(a+b)$-mesic with respect to the action of $\mathrm{row}$ on $[e,w]$.
\end{cor}

\begin{figure}
\begin{tikzpicture}
\node (1) at (2,0) {$12345$};
\node (2) at (1.25,1) {$13245$};
\node (3) at (3,1) {$12354$};
\node (4) at (-0.5,2) {$13425$};
\node (5) at (1.25,2) {$31245$};
\node (6) at (3,2) {$13254$};
\node (7) at (-1,3) {$13452$};
\node (8) at (0.5,3) {$31425$};
\node (9) at (2,3) {$31254$};
\node (10) at (3.5,3) {$13524$};
\node (11) at (-0.5,4) {$31452$};
\node (12) at (1.25,4) {$13542$};
\node (13) at (3,4) {$31524$};
\node (14) at (1.25,5) {$31542$};
\node (15) at (3,5) {$35124$};
\node (16) at (2,6) {$35142$};
\path (1) edge node[pos=0.5] {\tiny \mbox{\boldmath\color{red} $13245$}} (2);
\path (1) edge node[pos=0.5] {\tiny \mbox{\boldmath\color{red} $12354$}} (3);
\path (2) edge node[pos=0.5] {\tiny \mbox{\boldmath\color{red} $13425$}} (4);
\path (2) edge node[pos=0.5] {\tiny \mbox{\boldmath\color{red} $31245$}} (5);
\path (2) edge node[pos=0.5] {\tiny \mbox{\boldmath\color{red} $12354$}} (6);
\path (3) edge node[pos=0.5] {\tiny \mbox{\boldmath\color{red} $13245$}} (6);
\path (4) edge node[pos=0.5,xshift=-0.1cm] {\tiny \mbox{\boldmath\color{red} $13452$}} (7);
\path (4) edge node[pos=0.5] {\tiny \mbox{\boldmath\color{red} $31245$}} (8);
\path (5) edge node[pos=0.5] {\tiny \mbox{\boldmath\color{red} $13425$}} (8);
\path (5) edge node[pos=0.5,xshift=0.1cm] {\tiny \mbox{\boldmath\color{red} $12354$}} (9);
\path (6) edge node[pos=0.5,xshift=0.1cm] {\tiny \mbox{\boldmath\color{red} $31245$}} (9);
\path (6) edge node[pos=0.5,xshift=0.2cm] {\tiny \mbox{\boldmath\color{red} $13524$}} (10);
\path (7) edge node[pos=0.5,xshift=-0.1cm] {\tiny \mbox{\boldmath\color{red} $31245$}} (11);
\path (7) edge node[pos=0.8] {\tiny \mbox{\boldmath\color{red} $12354$}} (12);
\path (8) edge node[pos=0.1,xshift=0.2cm] {\tiny \mbox{\boldmath\color{red} $13452$}} (11);
\path (9) edge node[pos=0.1,xshift=-0.2cm] {\tiny \mbox{\boldmath\color{red} $13524$}} (13);
\path (10) edge node[pos=0.8] {\tiny \mbox{\boldmath\color{red} $13425$}} (12);
\path (10) edge node[pos=0.5,xshift=0.1cm] {\tiny \mbox{\boldmath\color{red} $31245$}} (13);
\path (11) edge node[pos=0.5] {\tiny \mbox{\boldmath\color{red} $12354$}} (14);
\path (12) edge node[pos=0.5] {\tiny \mbox{\boldmath\color{red} $31245$}} (14);
\path (13) edge node[pos=0.5] {\tiny \mbox{\boldmath\color{red} $13425$}} (14);
\path (13) edge node[pos=0.5,xshift=0.1cm] {\tiny \mbox{\boldmath\color{red} $35124$}} (15);
\path (14) edge node[pos=0.5,xshift=0.1cm] {\tiny \mbox{\boldmath\color{red} $35124$}} (16);
\path (15) edge node[pos=0.5,xshift=0.1cm] {\tiny \mbox{\boldmath\color{red} $13425$}} (16);
\end{tikzpicture}
\caption{Example~\ref{ex:wo_rowmotion} of rowmotion acting on a weak order interval. Here the $\gamma$-labels are written on the edges in red.} \label{fig:wo_rowmotion}
\end{figure}
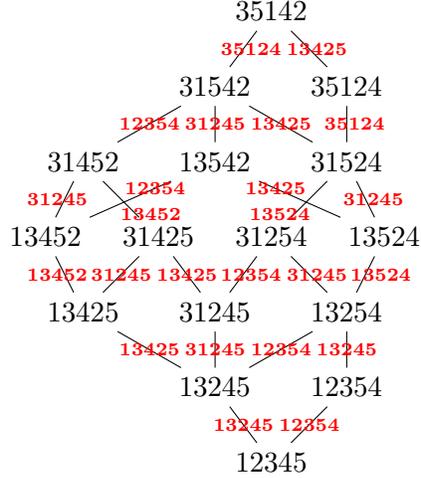

\begin{example} \label{ex:wo_rowmotion}
Let $w=35142\in\mathfrak{S}_5$, which is a vexillary permutation of the balanced shape $\lambda=\delta_4=(3,2,1)$. The weak order interval $[e,w]$ is depicted in Figure~\ref{fig:wo_rowmotion}, with its canonical $\gamma$-labels written on the edges in red. By Corollary~\ref{cor:balanced_vex_homo} we should have that the action of rowmotion on $[e,w]$ is $3/2$-mesic. Let's check this is the case. The four orbits of $\mathrm{row}$ acting on $[e,w]$ are:
\[ \{ \cdots \xrightarrow{\mathrm{row}} 12345 \xrightarrow{\mathrm{row}} 13254 \xrightarrow{\mathrm{row}} 31524 \xrightarrow{\mathrm{row}} 35142 \cdots\}; \]
\[ \{ \cdots \xrightarrow{\mathrm{row}} 13245 \xrightarrow{\mathrm{row}} 31542 \xrightarrow{\mathrm{row}} 35124 \xrightarrow{\mathrm{row}} 13425 \xrightarrow{\mathrm{row}} 31452 \xrightarrow{\mathrm{row}} 12354 \cdots\}; \]
\[ \{ \cdots \xrightarrow{\mathrm{row}}13452 \xrightarrow{\mathrm{row}} 31254 \xrightarrow{\mathrm{row}}13524\xrightarrow{\mathrm{row}} 31425\cdots\};\]
\[ \{ \cdots \xrightarrow{\mathrm{row}} 13542 \xrightarrow{\mathrm{row}} 31245 \cdots\}.\]
We can compute that the average down-degrees for these orbits are
\begin{align*}
\frac{1}{4}(0+2+2+2) &= \frac{6}{4} = \frac{3}{2};\\
\frac{1}{6}(1+3+1+1+2+1) &=\frac{9}{6} = \frac{3}{2};\\
\frac{1}{4}(1+2+1+2) &= \frac{6}{4} = \frac{3}{2};\\
\frac{1}{2}(2+1) &=\frac{3}{2}.
\end{align*}
This agrees with Corollary~\ref{cor:balanced_vex_homo}.
\end{example}

Our rowmotion down-degree homomesy result (Corollary~\ref{cor:balanced_vex_homo}) applies to rowmotion acting on the full weak order~$\mathfrak{S}_n$. But, as discussed in Section~\ref{sec:full_weak_order}, because the full weak order has constant Hasse diagram degree, this homomesy result is not very interesting (indeed, it is easy to see that the down-degree statistic along any rowmotion orbit will be $d,n-1-d,d,n-1-d,\ldots$ for some $0\leq d \leq n-1$). However, we can apply our result about the expectation of the ``refined'' down-degree statistics (Theorem~\ref{thm:full_weak_order}) to obtain a more interesting ``refined'' down-degree homomesy for rowmotion acting on the full weak order. (This is reminiscent of various homomesy refinements obtained by Propp-Roby~\cite{propp2015homomesy} in their original paper.)

Recall the statistics $f_k\colon \mathfrak{S}_n\to\mathbb{Z}$ for $1 \leq k < n$ defined by
\[f_k \coloneqq \sum_{j=1}^{k} \mathcal{T}^{-}_{(j,k+1)} + \sum_{j=k+1}^{n} \mathcal{T}^{-}_{(k,j)}.\]

\begin{cor} \label{cor:full_weak_order_homo}
For any $1\leq k < n$, the statistic $f_k$ is $1$-mesic  with respect to the action of $\mathrm{row}$ on the full weak order $\mathfrak{S}_n$.
\end{cor}
\begin{proof}
This follows by combining Lemma~\ref{lem:striker_semi} and Theorem~\ref{thm:full_weak_order}.
\end{proof}

\begin{example}
Let $n=3$ and consider the full weak order $\mathfrak{S}_3$ as depicted in Figure~\ref{fig:321_labels}. There are two orbits of rowmotion acting on $\mathfrak{S}_3$:
\[ \{ \cdots \xrightarrow{\mathrm{row}} 123 \xrightarrow{\mathrm{row}} 321 \cdots\}; \]
\[ \{ \cdots \xrightarrow{\mathrm{row}} 132 \xrightarrow{\mathrm{row}} 312 \xrightarrow{\mathrm{row}} 213 \xrightarrow{\mathrm{row}} 231 \cdots\}; \]
With $k=1$ we have $f_1 = 2\cdot \mathcal{T}^{-}_{(1,2)} + \mathcal{T}^{-}_{(1,3)}$.  We can compute that the averages of $f_1$ for these orbits are
\begin{align*}
\frac{1}{2}(0+2) &= \frac{2}{2} = 1;\\
\frac{1}{4}(0+1+2+1) &=\frac{4}{4} = 1;\\
\end{align*}
With $k=2$ we have $f_2 = \mathcal{T}^{-}_{(1,3)} + 2\cdot \mathcal{T}^{-}_{(2,3)}$.  We can compute that the averages of $f_2$ for these orbits are
\begin{align*}
\frac{1}{2}(0+2) &= \frac{2}{2} = 1;\\
\frac{1}{4}(2+1+0+1) &=\frac{4}{4} = 1;\\
\end{align*}
This agrees with Corollary~\ref{cor:full_weak_order_homo}.
\end{example}

An equivalent way to state Corollary~\ref{cor:full_weak_order_homo}, which avoids the ``homomesy'' terminology, is the following.

\begin{cor} \label{cor:full_weak_order_homo_equiv}
Let $\mathcal{O}$ be an orbit of $\mathrm{row}$ acting on the full weak order~$\mathfrak{S}_n$. Let us use the convention that for any $w\in \mathfrak{S}_n$, we have $w(0)=0$ and $w(n+1)=n+1$. Then for any~$1\leq k < n$ we have
\[ \#\{w\in \mathcal{O}\colon w(w^{-1}(k)-1) > k\} = \#\{w\in\mathcal{O}\colon w(w^{-1}(k+1)+1) > k+1\}.\]
(In words, the set on the left-hand side consists of the $w\in\mathcal{O}$ for which the letter $k$ forms a descent with the letter to its left, and the set on the right-hand side consists of the $w\in\mathcal{O}$ for which the letter $k+1$ forms an ascent with the letter to its right.)
\end{cor}
\begin{proof}
The statistic $\sum_{j=k+1}^{n} \mathcal{T}^{-}_{(k,j)}$ is the indicator function for the $w \in \mathfrak{S}_n$ which have the property $w(w^{-1}(k)-1) > k$. And $\sum_{j=1}^{k} \mathcal{T}^{-}_{(j,k+1)}$ is the indicator function for the property $w(w^{-1}(k+1)+1) < k+1$. So $\mathbf{1}-\sum_{j=1}^{k} \mathcal{T}^{-}_{(j,k+1)}$ is the indicator function for the complementary property $ w(w^{-1}(k+1)+1) > k+1$ (where $\mathbf{1}\colon\mathfrak{S}_n\to\mathbb{Z}$ is the constant function $\mathbf{1}(w)=1$). By Theorem~\ref{thm:full_weak_order} we have
\[ \mathbb{E}\left(\mu; \sum_{j=k+1}^{n} \mathcal{T}^{-}_{(k,j)}\right) = \mathbb{E}\left(\mu; \mathbf{1}-\sum_{j=1}^{k} \mathcal{T}^{-}_{(j,k+1)}\right)\]
for any toggle-symmetric distribution $\mu$ on $\mathfrak{S}_n$. By taking $\mu$ to be the distribution which is uniform on~$\mathcal{O}$ and zero outside~$\mathcal{O}$ (which is toggle-symmetric thanks to Lemma~\ref{lem:striker_semi}), we obtain the desired equality.
\end{proof}

\section{Future directions} \label{sec:future}

In this section we briefly discuss some possible threads of future research.

\subsection{Constructing all skew vexillary permutations of given shape} \label{sec:constructing_skew_vex}

Given a connected shape~$\sigma=\lambda/\nu$, we showed in Section~\ref{sec:skew_vex_basic_props} that there are finitely many permutations~$w$ which are skew vexillary of shape~$\sigma$ up to some trivial equivalences (namely, adding or removing initial or terminal fixed points). But how do we find all of them? We showed (Proposition~\ref{prop:finite}) that if $\sigma$ has height~$a$ and width~$b$ then they all occur in $\mathfrak{S}_{a+b}$, but checking each permutation in~$\mathfrak{S}_{a+b}$ is computationally unreasonable. Is there a better method than brute force?

Let's show what can be done in the vexillary case. Recall that for $w\in\mathfrak{S}_n$, the \emph{code} of $w$, denoted $c(w)$, is the vector $c(w)=(c_1,c_2,\ldots,c_{n}) \in\mathbb{N}^n$ where 
\[c_i \coloneqq  \#\{(i,j)\in\mathrm{Inv}(w)\}.\] 
A vector $c=(c_1,\ldots,c_n) \in \mathbb{N}^n$ is the code of some permutation in $\mathfrak{S}_n$ if and only if~$c_i < n-i$ for all $i=1,\ldots,n$. 

Let $\lambda$ be a straight shape. Then clearly a permutation $w\in\mathfrak{S}_n$ is vexillary of shape~$\lambda$ if and only if it is vexillary (of some shape) and the weakly decreasing rearrangement of $c(w)$ is equal to $\lambda$. 

Moreover, there is an explicit description of the codes of vexillary permutations. A code $c=(c_1,\ldots,c_n)\in\mathbb{N}^n$ corresponds to a vexillary permutation if and only if it satisfies the following two conditions:
\begin{itemize}
\item if $1\leq i < j \leq n$ and $c_i > c_j$ then $\#\{k\colon i < k < j \textrm{ and } c_k < c_j\} \leq c_i - c_j$;
\item if if $1\leq i < j\leq n$ and $c_i \leq c_j$, then $c_k \geq c_i$ whenever $i < k < j$.
\end{itemize}
This description of vexillary codes appears for instance in the monograph of Macdonald~\cite[(1.32)]{macdonald1991notes}. With this description of vexillary codes, it becomes easy to find all the permutations which are vexillary of a given shape. We would like a similar description for skew vexillary permutations (possibly involving both the code and the ``cocode.'')

A related problem is to enumerate, exactly or approximately, the number of skew vexillary permutations in~$\mathfrak{S}_n$. The number of vexillary permutations in~$\mathfrak{S}_n$ has an exact formula due to the work of Gessel~\cite{gessel1990symmetric} and West~\cite{west1995generating}, and the asymptotics are also well-understood. However, as we mentioned in Remark~\ref{rem:pattern_avoid}, we do not believe there is a pattern avoidance description of skew vexillary, so enumerating the skew vexillary permutations might be quite different from the vexillary permutations.

\subsection{Stable Grothendieck polynomials for skew vexillary permutations}

As discussed in more detail in Remark~\ref{rem:stable_grothendieck}, we suspect that for a skew vexillary permutation $w\in\mathfrak{S}_n$ of shape~$\sigma=\lambda/\nu$, we have the equality  $G_w=G_{\lambda/\nu}$, where $G_w$ is the stable Grothendieck polynomial of $w$, and $G_{\lambda/\nu}$ is the skew stable Grothendieck polynomial of skew shape~$\lambda/\nu$. It would be nice to verify or disprove this suspicion.

\subsection{Other semidistributive lattices where tCDE implies CDE} \label{sec:tcde_implies_cde}

We showed that for initial weak order intervals, being tCDE implies CDE (Corollary~\ref{cor:tcde_implies_cde}), but we also showed that this is not true for general semidistributive lattices (Example~\ref{ex:tcde_not_cde}). It would be nice to find more necessary and/or sufficient conditions for tCDE to imply CDE in a semidistributive lattice. Note that the counterexample in Example~\ref{ex:tcde_not_cde} is even a graded semidistributive lattice, so gradedness alone is not sufficient. In conversations with Emily Barnard, she explained to us that one should be able to extend the results of Section~\ref{sec:wo_maxchain} (showing that the maxchain distribution is toggle-symmetric for weak order intervals) to the intervals of the semidistributive lattice associated to any simplicial hyperplane arrangement using the theory of \emph{shards}, as in~\cite{reading2003lattice}. We thank her for this very helpful explanation. Note in particular that the class of semidistributive lattice associated to simplicial hyperplane arrangements includes the weak order for all finite Coxeter groups.

\subsection{Rowmotion on weak order intervals}

Our homomesy result for rowmotion acting on intervals of weak order (Corollary~\ref{cor:balanced_vex_homo}) suggests that it might be worth further exploring rowmotion acting on intervals of weak order. The problem is knowing what questions to ask. For instance, it would be nice to give a formula for the order of rowmotion. The most obvious weak order intervals to consider in this context would be~$[e,w]$ for $w$ a vexillary permutation of rectangular shape~$\lambda=b^a$. However, as mentioned in Section~\ref{sec:skew_vex_basic_props}, all of these posets are actually distributive lattices isomorphic to $[\varnothing,b^a]$. So we don't get any new examples from vexillary permutations of rectangular shape. The next most obvious case to consider would be $[e,w]$ for $w$ a vexillary permutation of staircase shape~$\lambda=\delta_d$. However, by taking $w=w_0 \in \mathfrak{S}_d$, we get as one poset in this family the full weak order $[e,w_0]$ for the symmetric group~$\mathfrak{S}_d$. And as observed by Thomas-Williams~\cite[\S7.11]{thomas2017rowmotion}, the order of rowmotion acting on the full weak order seems unpredictable. So, as we said, we cannot think of any questions to ask about rowmotion acting on weak order intervals which look like they might have nice answers. It would be nice to apply Corollary~\ref{cor:full_weak_order_homo_equiv} to say something more about rowmotion acting on the full weak order.

\subsection{Other types}

A natural question is how much of this work can be extended to ``other types,'' i.e., to the weak order of other finite Coxeter groups. It is known that weak order intervals for all finite Coxeter groups are semidistributive lattices~\cite{lecontedepolybarbut1994coxeter}. Hence a first step, related to the discussion in Section~\ref{sec:tcde_implies_cde}, would be to show that tCDE implies CDE for intervals of weak order in other types. Next, one would want an appropriate analog of the notion of ``vexillary of balanced shape'' for elements of a general finite Coxeter group. As for what these ``shapes'' should be in other types, it is reasonable to suspect that, for instance in Type D the shape will be a shifted shape, i.e., a strict partition. The shifted shapes are $d$-complete posets in the sense of Proctor (see~\cite{proctor1999minuscule, proctor1999dynkin}) and hence their corresponding posets of order ideals arise as distributive lattices $[e,w]$ for $w$ a fully commutative element in these other Coxeter groups. A notion of ``shifted balanced shape'' was introduced in~\cite{hopkins2017cde}. But the appropriate Coxeter group analog of being ``vexillary of a given shape'' remains elusive. 

\bibliography{cde_vexillary}{}

\begin{thebibliography}{10}

\bibitem{armstrong2013uniform}
Drew Armstrong, Christian Stump, and Hugh Thomas.
\newblock A uniform bijection between nonnesting and noncrossing partitions.
\newblock {\em Trans. Amer. Math. Soc.}, 365(8):4121--4151, 2013.

\bibitem{barnard2016canonical}
Emily Barnard.
\newblock The canonical join complex.
\newblock {\em Electron. J. Combin.}, 26(1):Paper 1.24, 25, 2019.

\bibitem{billera2006decomposable}
Louis~J. Billera, Hugh Thomas, and Stephanie van Willigenburg.
\newblock Decomposable compositions, symmetric quasisymmetric functions and
  equality of ribbon {S}chur functions.
\newblock {\em Adv. Math.}, 204(1):204--240, 2006.

\bibitem{billey1993schubert}
Sara~C. Billey, William Jockusch, and Richard~P. Stanley.
\newblock Some combinatorial properties of {S}chubert polynomials.
\newblock {\em J. Algebraic Combin.}, 2(4):345--374, 1993.

\bibitem{bjorner1984orderings}
Anders Bj\"{o}rner.
\newblock Orderings of {C}oxeter groups.
\newblock In {\em Combinatorics and algebra ({B}oulder, {C}olo., 1983)},
  volume~34 of {\em Contemp. Math.}, pages 175--195. Amer. Math. Soc.,
  Providence, RI, 1984.

\bibitem{bjorner2005coxeter}
Anders Bj\"orner and Francesco Brenti.
\newblock {\em Combinatorics of {C}oxeter groups}, volume 231 of {\em Graduate
  Texts in Mathematics}.
\newblock Springer, New York, 2005.

\bibitem{brouwer1974period}
A.~E. Brouwer and A.~Schrijver.
\newblock {\em On the period of an operator, defined on antichains}.
\newblock Mathematisch Centrum, Amsterdam, 1974.
\newblock Mathematisch Centrum Afdeling Zuivere Wiskunde ZW 24/74.

\bibitem{buch2002littlewood}
Anders~Skovsted Buch.
\newblock A {L}ittlewood-{R}ichardson rule for the {$K$}-theory of
  {G}rassmannians.
\newblock {\em Acta Math.}, 189(1):37--78, 2002.

\bibitem{cameron1995orbits}
P.~J. Cameron and D.~G. Fon-Der-Flaass.
\newblock Orbits of antichains revisited.
\newblock {\em European J. Combin.}, 16(6):545--554, 1995.

\bibitem{chan2017expected}
Melody Chan, Shahrzad Haddadan, Sam Hopkins, and Luca Moci.
\newblock The expected jaggedness of order ideals.
\newblock {\em Forum Math. Sigma}, 5:e9, 27, 2017.

\bibitem{chan2018genera}
Melody Chan, Alberto L\'{o}pez~Mart\'{i}n, Nathan Pflueger, and Montserrat
  Teixidor~i Bigas.
\newblock Genera of {B}rill-{N}oether curves and staircase paths in {Y}oung
  tableaux.
\newblock {\em Trans. Amer. Math. Soc.}, 370(5):3405--3439, 2018.

\bibitem{dittmer2018counting}
Samuel Dittmer and Igor Pak.
\newblock Counting linear extensions of restricted posets.
\newblock \arxiv{1802.06312}, 2018.

\bibitem{duquenne1994permutation}
Vincent Duquenne and Ameziane Cherfouh.
\newblock On permutation lattices.
\newblock {\em Math. Social Sci.}, 27(1):73--89, 1994.

\bibitem{dyer2011weak}
Matthew Dyer.
\newblock On the weak order for {C}oxeter groups.
\newblock \arxiv{1108.5557}, 2011.

\bibitem{edelman1987balanced}
Paul Edelman and Curtis Greene.
\newblock Balanced tableaux.
\newblock {\em Adv. in Math.}, 63(1):42--99, 1987.

\bibitem{einstein2013combinatorial}
David Einstein and James Propp.
\newblock Combinatorial, piecewise-linear, and birational homomesy for products
  of two chains.
\newblock \arxiv{1310.5294}, 2013.

\bibitem{fonderflaass1993orbits}
D.~G. Fon-Der-Flaass.
\newblock Orbits of antichains in ranked posets.
\newblock {\em European J. Combin.}, 14(1):17--22, 1993.

\bibitem{freese1995free}
Ralph Freese, Jaroslav Je\v{z}ek, and J.~B. Nation.
\newblock {\em Free lattices}, volume~42 of {\em Mathematical Surveys and
  Monographs}.
\newblock American Mathematical Society, Providence, RI, 1995.

\bibitem{gessel1990symmetric}
Ira~M. Gessel.
\newblock Symmetric functions and {P}-recursiveness.
\newblock {\em J. Combin. Theory Ser. A}, 53(2):257--285, 1990.

\bibitem{grinberg2015birational2}
Darij Grinberg and Tom Roby.
\newblock Iterative properties of birational rowmotion {II}: rectangles and
  triangles.
\newblock {\em Electron. J. Combin.}, 22(3):Paper 3.40, 49, 2015.

\bibitem{grinberg2016birational1}
Darij Grinberg and Tom Roby.
\newblock Iterative properties of birational rowmotion {I}: generalities and
  skeletal posets.
\newblock {\em Electron. J. Combin.}, 23(1):Paper 1.33, 40, 2016.

\bibitem{hopkins2017cde}
Sam Hopkins.
\newblock The {CDE} property for minuscule lattices.
\newblock {\em J. Combin. Theory Ser. A}, 152:45--103, 2017.

\bibitem{klein2014counting}
Aaron~J. Klein, Joel~Brewster Lewis, and Alejandro~H. Morales.
\newblock Counting matrices over finite fields with support on skew {Y}oung
  diagrams and complements of {R}othe diagrams.
\newblock {\em J. Algebraic Combin.}, 39(2):429--456, 2014.

\bibitem{knutson2009grobner}
Allen Knutson, Ezra Miller, and Alexander Yong.
\newblock Gr\"{o}bner geometry of vertex decompositions and of flagged
  tableaux.
\newblock {\em J. Reine Angew. Math.}, 630:1--31, 2009.

\bibitem{kraskiewicz2004schubert}
Witold Kra\'{s}kiewicz and Piotr Pragacz.
\newblock Schubert functors and {S}chubert polynomials.
\newblock {\em European J. Combin.}, 25(8):1327--1344, 2004.

\bibitem{lascoux1982polynomes}
Alain Lascoux and Marcel-Paul Sch\"{u}tzenberger.
\newblock Polyn\^{o}mes de {S}chubert.
\newblock {\em C. R. Acad. Sci. Paris S\'{e}r. I Math.}, 294(13):447--450,
  1982.

\bibitem{lecontedepolybarbut1994coxeter}
C.~Le~Conte~de Poly-Barbut.
\newblock Sur les treillis de {C}oxeter finis.
\newblock {\em Math. Inform. Sci. Humaines}, (125):41--57, 1994.

\bibitem{macdonald1991notes}
I.G. Macdonald.
\newblock {\em Notes on Schubert polynomials}.
\newblock Montr{\'e}al: D{\'e}p. de math{\'e}matique et d'informatique,
  Universit{\'e} du Qu{\'e}bec {\`a} Montr{\'e}al, 1991.
\newblock Volume 6 of the publications of the Laboratoire de Combinatoire et
  d'Informatique Math\'{e}matique (LACIM).

\bibitem{mcnamara2009towards}
Peter R.~W. McNamara and Stephanie van Willigenburg.
\newblock Towards a combinatorial classification of skew {S}chur functions.
\newblock {\em Trans. Amer. Math. Soc.}, 361(8):4437--4470, 2009.

\bibitem{panyushev2009orbits}
Dmitri~I. Panyushev.
\newblock On orbits of antichains of positive roots.
\newblock {\em European J. Combin.}, 30(2):586--594, 2009.

\bibitem{postnikov2009chains}
Alexander Postnikov and Richard~P. Stanley.
\newblock Chains in the {B}ruhat order.
\newblock {\em J. Algebraic Combin.}, 29(2):133--174, 2009.

\bibitem{proctor1999dynkin}
Robert~A. Proctor.
\newblock Dynkin diagram classification of {$\lambda$}-minuscule {B}ruhat
  lattices and of {$d$}-complete posets.
\newblock {\em J. Algebraic Combin.}, 9(1):61--94, 1999.

\bibitem{proctor1999minuscule}
Robert~A. Proctor.
\newblock Minuscule elements of {W}eyl groups, the numbers game, and
  {$d$}-complete posets.
\newblock {\em J. Algebra}, 213(1):272--303, 1999.

\bibitem{propp2015homomesy}
James Propp and Tom Roby.
\newblock Homomesy in products of two chains.
\newblock {\em Electron. J. Combin.}, 22(3):Paper 3.4, 29, 2015.

\bibitem{reading2003lattice}
Nathan Reading.
\newblock Lattice and order properties of the poset of regions in a hyperplane
  arrangement.
\newblock {\em Algebra Universalis}, 50(2):179--205, 2003.

\bibitem{reading2015noncrossing}
Nathan Reading.
\newblock Noncrossing arc diagrams and canonical join representations.
\newblock {\em SIAM J. Discrete Math.}, 29(2):736--750, 2015.

\bibitem{reiner2007coincidences}
Victor Reiner, Kristin~M. Shaw, and Stephanie van Willigenburg.
\newblock Coincidences among skew {S}chur functions.
\newblock {\em Adv. Math.}, 216(1):118--152, 2007.

\bibitem{reiner1995key}
Victor Reiner and Mark Shimozono.
\newblock Key polynomials and a flagged {L}ittlewood-{R}ichardson rule.
\newblock {\em J. Combin. Theory Ser. A}, 70(1):107--143, 1995.

\bibitem{reiner2018poset}
Victor Reiner, Bridget~Eileen Tenner, and Alexander Yong.
\newblock Poset edge densities, nearly reduced words, and barely set-valued
  tableaux.
\newblock {\em J. Combin. Theory Ser. A}, 158:66--125, 2018.

\bibitem{roby2016dynamical}
Tom Roby.
\newblock Dynamical algebraic combinatorics and the homomesy phenomenon.
\newblock In {\em Recent trends in combinatorics}, volume 159 of {\em IMA Vol.
  Math. Appl.}, pages 619--652. Springer, [Cham], 2016.

\bibitem{rush2016orbits}
David~B. Rush.
\newblock On order ideals of minuscule posets {III}: The {CDE} property.
\newblock \arxiv{1607.08018}, 2016.

\bibitem{rush2013orbits}
David~B. Rush and XiaoLin Shi.
\newblock On orbits of order ideals of minuscule posets.
\newblock {\em J. Algebraic Combin.}, 37(3):545--569, 2013.

\bibitem{rush2015orbits}
David~B. Rush and Kelvin Wang.
\newblock On orbits of order ideals of minuscule posets {II}: Homomesy.
\newblock \arxiv{1509.08047}, 2015.

\bibitem{Sage-Combinat}
The {S}age-{C}ombinat community.
\newblock {S}age-{C}ombinat: enhancing {S}age as a toolbox for computer
  exploration in algebraic combinatorics, 2018.
\newblock \href{http://combinat.sagemath.org}{\tt
  http://combinat.sagemath.org}.

\bibitem{sagemath}
The {S}age {D}evelopers.
\newblock {\em {S}ageMath, the {S}age {M}athematics {S}oftware {S}ystem
  ({V}ersion 8.4)}, 2018.
\newblock \href{http://www.sagemath.org}{\tt http://www.sagemath.org}.

\bibitem{stanley1984reduced}
Richard~P. Stanley.
\newblock On the number of reduced decompositions of elements of {C}oxeter
  groups.
\newblock {\em European J. Combin.}, 5(4):359--372, 1984.

\bibitem{stanley1999ec2}
Richard~P. Stanley.
\newblock {\em Enumerative combinatorics. {V}ol. 2}, volume~62 of {\em
  Cambridge Studies in Advanced Mathematics}.
\newblock Cambridge University Press, Cambridge, 1999.
\newblock With a foreword by Gian-Carlo Rota and appendix 1 by Sergey Fomin.

\bibitem{stanley2012ec1}
Richard~P. Stanley.
\newblock {\em Enumerative combinatorics. {V}olume 1}, volume~49 of {\em
  Cambridge Studies in Advanced Mathematics}.
\newblock Cambridge University Press, Cambridge, second edition, 2012.

\bibitem{stembridge1995fully}
John~R. Stembridge.
\newblock On the fully commutative elements of {C}oxeter groups.
\newblock {\em J. Algebraic Combin.}, 5(4):353--385, 1996.

\bibitem{striker2015toggle}
Jessica Striker.
\newblock The toggle group, homomesy, and the {R}azumov-{S}troganov
  correspondence.
\newblock {\em Electron. J. Combin.}, 22(2):Paper 2, 57, 2015.

\bibitem{striker2017dynamical}
Jessica Striker.
\newblock Dynamical algebraic combinatorics: promotion, rowmotion, and
  resonance.
\newblock {\em Notices Amer. Math. Soc.}, 64(6):543--549, 2017.

\bibitem{striker2018rowmotion}
Jessica Striker.
\newblock Rowmotion and generalized toggle groups.
\newblock {\em Discrete Math. Theor. Comput. Sci.}, 20(1):Paper No. 17, 26,
  2018.

\bibitem{striker2012promotion}
Jessica Striker and Nathan Williams.
\newblock Promotion and rowmotion.
\newblock {\em European J. Combin.}, 33(8):1919--1942, 2012.

\bibitem{thomas2017rowmotion}
H.~Thomas and N.~Williams.
\newblock Rowmotion in slow motion.
\newblock {\em P. Lond. Math. Soc.}, 119(5):1149--1178, 2019.

\bibitem{wachs1985flagged}
Michelle~L. Wachs.
\newblock Flagged {S}chur functions, {S}chubert polynomials, and symmetrizing
  operators.
\newblock {\em J. Combin. Theory Ser. A}, 40(2):276--289, 1985.

\bibitem{west1995generating}
Julian West.
\newblock Generating trees and the {C}atalan and {S}chr\"{o}der numbers.
\newblock {\em Discrete Math.}, 146(1-3):247--262, 1995.

\end{thebibliography}
\bibliographystyle{plain}

\end{document}